\documentclass[review,onefignum,onetabnum]{siamart190516}
\renewcommand{\theequation}{\thesection.\arabic{equation}}
\renewcommand{\thefigure}{\thesection.\arabic{figure}}
\renewcommand{\thetable}{\thesection.\arabic{table}}
       
\usepackage{amsmath} 
\usepackage{cleveref}
\usepackage{amsfonts}
\usepackage{amssymb}
\usepackage{multirow}
\usepackage{stmaryrd}
\usepackage{subcaption}
\usepackage{makecell}
\usepackage{cite}
\usepackage{boldline}
\usepackage{graphics}
\usepackage{epstopdf}
\usepackage{arydshln}
\graphicspath{{./pic/}}
\usepackage{mathtools}
\usepackage{caption}
\usepackage{subcaption}
\usepackage{hyperref}  
\usepackage{tikz}
\newcommand{\stencilpt}[4][]{\node[rectangle,draw,minimum width=1.1cm,minimum height=0.5cm,font=\tiny,#1] at (#2) (#3) {#4}}
 
\usepackage{lipsum}
\usepackage{amsfonts}
\usepackage{graphicx}
\usepackage{epstopdf}
\usepackage{algorithmic}
\ifpdf
  \DeclareGraphicsExtensions{.eps,.pdf,.png,.jpg}
\else
  \DeclareGraphicsExtensions{.eps}
\fi

\headers{The compact RKDG method for conservation laws}{Q. Chen, Z. Sun and Y. Xing}

\title{The Runge--Kutta discontinuous Galerkin method with compact stencils for hyperbolic conservation laws}

\author{Qifan Chen\thanks{Department of Mathematics, The Ohio State University, Columbus, OH 43210, USA. (\email{chen.11010@osu.edu}) }
  \and Zheng Sun\thanks{Department of Mathematics, The University of Alabama,
		Tuscaloosa, AL 35487, USA. (\email{zsun30@ua.edu})} \and Yulong Xing\thanks{Department of Mathematics, The Ohio State University,
		Columbus, OH 43210, USA. (\email{xing.205@osu.edu}) }}

\usepackage{amsopn}

\newsiamremark{remark}{Remark}

\newcommand{\dd}{\mathrm{d}}
\newcommand{\DG}{\mathrm{DG}}

\newcommand{\dxDG}{\nabla^\DG\cdot}
\newcommand{\dxl}{\nabla^\mathrm{loc}\cdot}

\newcommand{\hf}{\frac{1}{2}}
\newcommand{\ext}{\mathrm{ext}}
\newcommand{\inte}{\mathrm{int}}
\newcommand{\nm}[1]{\left\|#1\right\|}
\newcommand{\quand}{\quad \text{and} \quad}

\newsiamremark{exmp}{Example}
\newtheorem{prop}{Proposition}[section]
\newtheorem{lem}{Lemma}[section]

\usepackage[belowskip=-2pt]{caption}
\begin{document}

\maketitle 

\begin{abstract}
In this paper, we develop a new type of Runge--Kutta (RK) discontinuous Galerkin (DG) method for solving hyperbolic conservation laws. Compared with the original RKDG method, the new method features improved compactness and allows simple boundary treatment. The key idea is to hybridize two different spatial operators in an explicit RK scheme, utilizing local projected derivatives for inner RK stages and the usual DG spatial discretization for the final stage only. Limiters are applied only at the final stage for the control of spurious oscillations. We also explore the connections between our method and Lax--Wendroff DG schemes and ADER-DG schemes. 
Numerical examples are given to confirm that the new RKDG method is as accurate as the original RKDG method, while being more compact, for problems including two-dimensional Euler equations for compressible gas dynamics.  
\end{abstract}

\section{Introduction}

In this paper, we present a novel class of high-order Runge--Kutta (RK) discontinuous Galerkin  (DG) methods for solving hyperbolic conservation laws.
Compared with the original RKDG method proposed in \cite{rkdg1,rkdg2,rkdg3,rkdg4,rkdg5}, the new method features more compact stencil sizes and will be hence referred to as the compact RKDG (cRKDG) method throughout the paper.

The RKDG method for conservation laws was originally proposed by Cockburn et al. in a series of papers \cite{rkdg1,rkdg2,rkdg3,rkdg4,rkdg5}. The method combines the DG finite element spatial discretization \cite{reed1973triangular} with the strong-stability-preserving (SSP) RK time discretization \cite{gottlieb2001strong,gottlieb2011strong}. A limiting procedure is employed to control oscillations near physical discontinuities. The method naturally preserves the local conservation, features good $hp$ adaptivity, and can be fitted into complex geometries. Due to its various advantages, the RKDG method has become one of the primary numerical methods for the simulation of hyperbolic conservation laws. 

This paper aims to further improve the RKDG method by reducing its stencil size within each time step, which can reduce its data communication and lead to potential advantages in parallel computing or implicit time marching. The DG spatial discretization is typically more compact when compared to the finite difference method of the same order. For example, when approximating the first-order spatial derivatives of a function on a given cell, the DG method only utilizes data from immediate neighbors, while the finite difference method may require data from farther nodes for a high-order approximation. However, this spatial discretization advantage is not preserved by the one-step multi-stage RK time stepping. In each temporal stage, the spatial operator calls for information from neighboring cells, hence the stencil of the scheme will be expanded after each stage. For example, for one-dimensional scalar conservation laws with Lax--Friedrichs flux, the stencil size is 3 for the first-order forward-Euler time stepping, but becomes $2s+1$ for the high-order RK time stepping with $s$ stages. See \cref{fig1} for an illustration. This issue could be more evident with a very high-order RK method in multidimensions. 

There are some existing techniques in the literature circumventing the aforementioned issue. The main idea is to construct a one-step one-stage method for time marching. For instance, the Lax--Wendroff temporal discretization can be used as an alternative, which leads to the so-called Lax--Wendroff DG (LWDG) method \cite{qiu2005discontinuous,qiu2007numerical}. This method utilizes only information from immediate neighbors regardless of its temporal order. However, implementing this method can be very tedious, especially for high-order schemes for multidimensional systems, as one needs to compute the high-order derivatives of the flux function in the Cauchy--Kowalewski procedure. Another avenue is to employ the Arbitrary DERivative (ADER) time stepping  \cite{toro2001towards,titarev2002ader}, resulting in the so-called ADER-DG method, which is presented in the spacetime integral form of the conservation laws and is known to be closely connected to the LWDG scheme. See \cite{dumbser2006building,dumbser2008unified,rannabauer2018ader,gaburro2023high} and references therein. Besides LWDG and ADER-DG methods, there is also a stream of research addressing the compactness of DG methods for discretizing the second or higher-order spatial derivatives \cite{bassi1997high2,bassi2005discontinuous,van2005discontinuous,van2007discontinuous,peraire2008compact,luo2010reconstructed,yan2002local,cheng2008discontinuous,BCKX2013}. These studies focus on reducing the stencils of spatial operators but are less related with the issue arising from the multi-stage RK time stepping. If RK methods are used to discretize the corresponding
semi-discrete DG schemes for a time-dependent problem, the stencil size of the fully discrete schemes still grows with the number of RK stages.

In this paper, we propose a very different novel approach to tackle this issue. Our method is still based on the RK methods for temporal discretization, specifically using the Butcher form instead of the Shu--Osher form \cite{gottlieb2011strong}. The key idea is to hybridize two different types of spatial operators within each time step. For the inner stage(s) of the RK method, we employ the local derivative operator, which returns the $L^2$ projection of the spatial derivative of the flux function. While for the very last stage, we use the DG operator as in the original DG scheme. A limiter will be applied only once at the end of each time step, if necessary. The proposed cRKDG method has the following desirable properties. 
\begin{itemize}
    \item Stencil size: In each time step, the stencil of the cRKDG method only contains the current cell and its immediate neighbors, resulting in a compact stencil. 
    \item Convergence: We prove that if the cRKDG method converges boundedly, then its limit is a weak solution of the conservation laws. 
    \item Accuracy: We numerically observe that the cRKDG method attains $(k+1)$th order convergence rate when we couple a $(k+1)$th order RK method with a DG method using $k$th order spatial polynomials. This matches the optimal convergence of the original RKDG method. 
    \item Boundary error: We numerically observe that the cRKDG method attains the same optimal convergence rate when nonhomogeneous Dirichlet boundary conditions are imposed. In contrast, the original RKDG method will suffer the accuracy degeneracy under this setting  \cite{zhang2011third}. 
    
\end{itemize}

In addition to the aforementioned analysis and numerical observations, we conduct a linear stability analysis for the cRKDG method and identify its maximum CFL numbers for both second- and third-order cases. We employ the standard von Neumann analysis and investigate the eigen-structures of the amplification matrices. This approach aligns with the methodologies employed in previous works such as \cite{cockburn2001runge} for the RKDG method and \cite{qiu2005discontinuous} for the LWDG method. The details are omitted due to the space constraint. The maximum CFL numbers are documented in \cref{tab:CFL}. We can see from the table that the CFL number of the cRKDG method is the same as that of the RKDG method for the second-order case, and is slightly smaller for the third-order case. While comparing to the LWDG method, the CFL numbers of the cRKDG method are larger. 
\begin{table}[h!]
    \centering
    \begin{tabular}{c|c|c|c}
    \hline
        Maximum CFL& cRKDG & RKDG \cite{cockburn2001runge}& LWDG \cite{qiu2005discontinuous}\\
        \hline
         Second-order& 0.333 & 0.333 & 0.223 \\
         Third-order& 0.178 & 0.209 & 0.127 \\
    \hline
    \end{tabular}
    \caption{Maximum CFL numbers of the cRKDG method, the RKDG method, and the LWDG method. Here the upwind flux is used for cRKDG method and the RKDG method, and the Lax--Friedrichs flux \eqref{eq:LWDG:flux} is used for the LWDG method \cite{qiu2005discontinuous}. }
    \label{tab:CFL}
\end{table}

For general nonlinear conservation laws, the cRKDG method proposed in this paper is different from the LWDG and ADER-DG methods. However these methods are closely connected. Indeed, for linear conservation laws with constant coefficients, the cRKDG method and the LWDG method are actually equivalent under certain choice of the numerical fluxes,\footnote{While with the numerical flux in \cite{qiu2005discontinuous}, the LWDG method is not equivalent to the cRKDG method even in the linear case. See \cref{remark:LWDGflux}. This fact explains the different CFL numbers in \cref{tab:CFL}.} although the two methods are designed from totally different perspectives. Furthermore, the cRKDG method with a special implicit RK method is equivalent to the ADER-DG method with a special local predictor. As a result, one can expect that the cRKDG method may share some similar properties with the LWDG and ADER-DG methods. These connections may bring in new perspectives that could contribute to the further development and understanding of existing methods. 

The rest of the paper is organized as follows. In \Cref{sec:construction}, we review the original RKDG method and explain the formulation of the novel cRKDG method. In \Cref{sec:prop}, we summarize the theoretical properties of the cRKDG schemes and postpone some of the technical proofs to the appendix. Numerical tests are provided in \Cref{sec:Numerical results} and the conclusions are given in \Cref{sec:conclusions}.

\section{Numerical schemes}
\label{sec:construction}

In this section, we start by briefly reviewing the RKDG method and then describe in detail the construction of the cRKDG method. For ease of notation, we focus on scalar conservation laws, but the method can be extended to systems of conservation laws straightforwardly. 

\subsection{RKDG schemes}
Consider the hyperbolic conservation laws
\begin{equation}
	\partial_t u +\nabla \cdot f(u) = 0, \qquad u(x,0) = u_0(x).  
	\label{eq:conservation law}
\end{equation}
Let $\mathcal{T}_h = \{K\}$ be a partition of the spatial domain in $d$ dimension. We denote by $h_K$ the diameter of $K$ and $h = \max_{K\in \mathcal{T}_h} h_K$. Let $\partial K$ be the boundary of $K$. For each edge $e \in \partial K$, $\nu_{e,K}$ is the unit outward normal vector along $e$ with respect to $K$. The finite element space of the DG approximation is defined as $
	V_h = \{v:v|_{K}\in \mathcal{P}^k(K),\, \forall K \in \mathcal{T}_h\}$,
where $\mathcal{P}^k(K)$ denotes the set of polynomials of degree up to $k$ on the cell $K$. The standard semi-discrete DG method for solving \eqref{eq:conservation law} is defined as follows: find $u_h \in V_h$ such that on each $K\in \mathcal{T}_h$, 
\begin{equation}\label{eq:DGsemiweak}
	\int_K\left(u_{h}\right)_{t} v_h \dd x-\int_K f\left(u_{h}\right) \cdot \nabla v_h \dd x+\sum_{e\in \partial K} \int_{e} \widehat{f\cdot \nu_{e,K}} v_h \dd l   = 0, \quad \forall v_h\in V_h.
\end{equation}
Here $\widehat{f\cdot \nu_{e,K}}$ is the numerical flux, which can be computed from the exact or approximate Riemann solver defined at the cell interface. For example, we can choose the Lax--Friedrichs flux of the form
\begin{equation*}
	\widehat{f\cdot \nu_{e,K}} 
        = \hf\left(f(u_h^\inte)\cdot\nu_{e,K} + f(u_h^\ext)\cdot\nu_{e,K} - {\alpha_{e,K}}\left(u_h^\ext - u_h^\inte\right)\right), 
\end{equation*}
with $\alpha_{e,K} = \max \left|\partial_u f\cdot \nu_{e,K}\right|$. Here $u_h^\inte$ and $u_h^\ext$ are limits of $u_h$ along $e$ from the interior and exterior of the cell $K$.

We introduce the discrete operator $\dxDG f:V_h\to V_h$, defined by
\begin{equation}
	\int_K \dxDG f(u_h) v_h \dd x = -\int_K f(u_h)\cdot \nabla v_h \dd x + \sum_{e\in \partial K} \int_e \widehat{f\cdot \nu_{e,K}} v_h \dd l, \quad \forall v_h \in V_h.\label{eq:DG operator}
\end{equation}
Therefore the semi-discrete DG scheme \eqref{eq:DGsemiweak} can be rewritten in the strong form
\begin{equation}\label{eq:semiDG}
	\partial_t u_h + \dxDG f(u_h) = 0.
\end{equation}

Then we apply an explicit RK method to discretize \eqref{eq:semiDG} in time. Consider an explicit RK method associated with the Butcher Tableau
\begin{equation}\label{eq:butcher}
	\begin{array}{c|c}
		c&A\\\hline
		&b\\
	\end{array},\quad A = (a_{ij})_{s\times s}, \quad b = (b_1,\cdots, b_s),
\end{equation}
where $A$ is a lower triangular matrix in \eqref{eq:butcher}, namely, $a_{ij} = 0$ if $i>j$. The corresponding RKDG scheme is given by 
\begin{subequations}\label{eq:rkdg}
	\begin{align}
		u_h^{(i)} =&\, u_h^n -  \Delta t\sum_{j = 1}^{i-1}  a_{ij} \dxDG f\left(u_h^{(j)}\right), \quad  i = 1, 2, \cdots, s,\label{eq-rkdg1}\\
		u_h^{n+1} =&\, u_h^n - \Delta t \sum_{i = 1}^s b_i \dxDG  f\left(u_h^{(i)} \right).\label{eq:rkdg-2}
	\end{align}
\end{subequations}
Note we have $u_h^{(1)} = u_h^n$ for the explicit RK method. In the case that the problem is nonautonomous, for example, when a time-dependent source term $q(t)$ is included, $q(t+c_j \Delta t)$ should be included at appropriate places of the RK stages. 

\subsection{cRKDG schemes} 
Similarly to \eqref{eq:DG operator}, we define a local discrete spatial operator $\dxl f: V_h \to V_h$ such that 
\begin{equation}\label{eq:dxl}
	\int_K \dxl f\left(u_{h}\right) v_h \dd x=-\int_K f\left(u_{h}\right)\cdot \nabla v_h \dd x+\sum_{e\in \partial K} \int_e {f(u_h)\cdot \nu_{e,K}} v_h \dd l, \quad \forall v_h \in V_h.
\end{equation}
In other words, instead of using a numerical flux involving $u_h$ on both sides of the cell interfaces, the values of $u_h$ along $\partial K$ are taken from the interior of the cell $K$. Therefore it is a local operation defined within $K$ without couplings with the neighboring cells. It is easy to see that when all integrals in \eqref{eq:dxl} are computed exactly, $\dxl f$ indeed returns the projected local derivative $$
\dxl f\left(u_{h}\right)=\Pi\, \nabla \cdot f\left(u_{h}\right),
$$
where $\Pi$ is the $L^{2}$ projection to $V_{h}$.  

With $\dxl f$ defined above, we propose our new cRKDG scheme in the following Butcher Tableau form: 
\begin{subequations}\label{eq:crkdg}
	\begin{align}
		u_h^{(i)} =& u_h^n -  \Delta t\sum_{j = 1}^{i-1} a_{ij} \dxl f\left(u_h^{(j)}\right), \quad  i = 1,2,\cdots, s,\label{eq:crkdg-1}\\
		u_h^{n+1} =& u_h^n - \Delta t\sum_{i = 1}^{s} b_i \dxDG f\left(u_h^{(i)}\right).\label{eq:crkdg-last}
	\end{align}
\end{subequations}
The main difference with the original RKDG scheme \eqref{eq:rkdg} is to use the local operator $\dxl f$ instead of $\dxDG f$ when evaluating the inner stage values $u_h^{(i)}$ in \eqref{eq:crkdg-1}.

For optimal convergence, we will couple $(k+1)$th order RK method with $\mathcal{P}^{k}$ spatial elements. The resulted fully discrete scheme is $(k+1)$th order accurate, which is the same optimal rate as that of the original RKDG method. For clarity, we list second- and third-order cRKDG schemes below as examples. \\
\emph{Second-order scheme ($k = 1$).}
\begin{subequations}\label{eq:2}
	\begin{align}
		u_h^{(2)} =& u_h^n - \frac{\Delta t}{2} \dxl f\left(u_h^n\right),\\
		u_h^{n+1} =& u_h^n - \Delta t \dxDG f\left(u_h^{(2)}\right).
	\end{align}
\end{subequations}
\emph{Third-order scheme ($k = 2$).}
\begin{subequations}\label{eq:3}
	\begin{align}
		u_h^{(2)} =& u_h^n - \frac{1}{3} \Delta t\dxl f\left(u_h^n\right),\qquad
		u_h^{(3)} = u_h^n - \frac{2}{3}\Delta t \dxl f\left(u_h^{(2)}\right),\\
		u_h^{n+1} =& u_h^n - \Delta t \left(\frac{1}{4}\dxDG f\left(u_h^{n}\right) + \frac{3}{4}  \dxDG f\left(u_h^{(3)}\right)\right).
	\end{align}
\end{subequations}

\begin{remark}[RK methods in Butcher form]\label{rmk:butcher}
	For the original RKDG method, strong-stability-preserving Runge--Kutta (SSP-RK) methods are often adopted as time stepping methods. However, we address that the cRKDG method \eqref{eq:crkdg} has to be written based on the Butcher form of RK methods, in order to preserve local conservation and achieve optimal accuracy. We refer to  \Cref{sec:butcher-RK} and \cref{ex:1dremark} for further details. 

    Since the cRKDG method is not based on SSP-RK time stepping, one can choose from a larger class of RK schemes without worrying about order barriers brought by the SSP property \cite{gottlieb2011strong}. For example, the classical four-stage fourth-order RK method can be used for the fourth-order scheme. 
\end{remark}

\begin{remark}[Limiters]
	To suppress spurious oscillation near discontinuities, a minmod or WENO type limiter is often applied after the update of each inner stage value $u_h^{(i)}$ in the original RKDG method \eqref{eq:rkdg}. For the cRKDG method, the limiter can be applied at the end of each time step. This is similar to the limiting strategy for LWDG method in \cite{qiu2005discontinuous} and will not change the stencil size of the cRKDG scheme. But at the same time, it's worth noting that the total-variation boundedness property, guaranteed for the original RKDG method \cite{rkdg2}, may not hold in this case.
\end{remark}

\section{Properties of the cRKDG method}\label{sec:prop}
\subsection{Stencil size}
In the cRKDG method, all inner stages are discretized with a local operator only using the information on the cell $K$. As a result, the stencil of the cRKDG scheme is determined by that of the last stage \eqref{eq:crkdg-last} only, and its size is the same as the forward-Euler--DG scheme. For example, in the one-dimensional case with Lax--Friedrichs fluxes, regardless of the number of RK stages, the stencil size of the cRKDG scheme is identically 3. 

In contrast, the stencil size of the RKDG scheme grows with the number of RK inner stages. In the one-dimensional case, the DG operator with Lax--Friedrichs flux has the stencil size $3$. While after the temporal discretization with an $s$-stage RK method, the stencil size becomes $2s+1$.  
See \cref{fig1} for an example with the second-order RK method. The difference on the stencil sizes between the two methods could be more significant in multidimensions. 
 
\begin{figure}
 \resizebox{1\textwidth}{!}{
\begin{tikzpicture}
\node[anchor=west] at (0,0) {$u^{(1)}=u_h^n$};
\node[anchor=west] at (0,1) {$u_h^{(2)} = u_h^n - \frac{\Delta t}{2} \dxDG  f\left(u_h^n\right)$};
\node[anchor=west] at (0,2) {$u_h^{n+1} = u_h^n - \Delta t  \dxDG f\left(u_h^{(2)}\right)$};
	\begin{scope}[xshift=6.5cm]
	\stencilpt{-2.2,0}{-2=0}{$u_{h,j-2}^{n}$};
	\stencilpt{-1.1,0}{-1=0}{$u_{h,j-1}^{n}$};
		\stencilpt{-1.1,1}{-1=1}{$u_{h,j-1}^{(2)}$};
	\stencilpt{ 0,0}{0=0}  {$u_{h,j }^{n}$};
	\stencilpt{ 1.1,0}{1=0}{$u_{h,j+1}^{n}$};
	\stencilpt{ 1.1,1}{1=1}{$u_{h,j+1}^{(2)}$};
	\stencilpt{ 2.2,0}{2=0}{$u_{h,j+2}^{n}$};
	\stencilpt{0, 1}{0=1}{$u_{h,j }^{(2)}$};
	\stencilpt{0, 2}{0=2}{$u_{h,j}^{n+1}$};
	\draw 	 [->] (-2=0) edge (-1=1)
  (-1=0) edge (-1=1)
   (0=0) edge (-1=1)
    (-1=0) edge (0=1)
      (0=0) edge (0=1) 
      (1=0) edge (0=1)
        (0=0) edge (1=1)
        (1=0) edge (1=1)
          (2=0)  edge (1=1)
          (-1=1) edge(0=2)
              (0=1) edge (0=2)
                (1=1) edge (0=2)	;
                \end{scope}
            
            	\begin{scope}[xshift=9.5cm,yshift=0cm]
            		\node[anchor=west] at (0,0) {$u^{(1)}=u_h^n$};
\node[anchor=west] at (0,1) {$u_h^{(2)} = u_h^n - \frac{\Delta t}{2}  \dxl f\left(u_h^n\right)$};
\node[anchor=west] at (0,2) {$u_h^{n+1} = u_h^n - \Delta t  \dxDG  f\left(u_h^{(2)}\right)$};

  \end{scope}

\begin{scope}[xshift=16cm,yshift=0cm]
            	 
            	\stencilpt{-1.1,0}{-1=0}{$u_{h,j-1}^{n}$};
            	\stencilpt{-1.1,1}{-1=1}{$u_{h,j-1}^{(2)}$};
            	\stencilpt{ 0,0}{0=0}  {$u_{h,j }^{n}$};
            	\stencilpt{ 1.1,0}{1=0}{$u_{h,j+1}^{n}$};
            	\stencilpt{ 1.1,1}{1=1}{$u_{h,j+1}^{(2)}$};
            	\stencilpt{0, 1}{0=1}{$u_{h,j }^{(2)}$};
            	\stencilpt{0, 2}{0=2}{$u_{h,j}^{n+1}$};
            	\draw 	 [->] (-1=0) edge (-1=1)
            	(1=0) edge (1=1)            
            	(0=0) edge (0=1) 
            	(-1=1) edge(0=2)
            	(0=1) edge (0=2)
            	(1=1) edge (0=2)	;
            \end{scope} 
\end{tikzpicture}
}
\caption{Stencils of RKDG and cRKDG methods with a second-order RK method.  
}
    \label{fig1}
\end{figure}
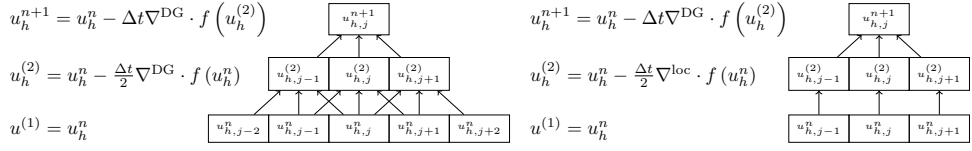

\begin{prop}
	The stencil of a cRKDG method of any temporal order only involves the current mesh cell and its immediate neighbors.
\end{prop}

\subsection{Convergence}\label{sec:butcher-SSP} 
\subsubsection{RK methods in Butcher form} \label{sec:butcher-RK} 
SSP-RK methods are usually used for time discretization in the original RKDG method. For example, the widely used second-order SSP-RKDG method is given by
	\begin{align*}
		u_h^{(1)} =& u_h^n - \Delta t \dxDG f(u_h^n),\\
		u_h^{n+1} =& \hf u_h^{n} + \hf \left(u_h^{(1)} - \Delta t \dxDG f\left(u_h^{(1)}\right)\right).
	\end{align*}
As above, an SSP-RK method can be written as convex combinations of forward-Euler steps, which is also referred to as the \emph{Shu--Osher form} in the literature. The method can also be written in the equivalent \emph{Butcher form} as those in \eqref{eq:rkdg}. 

In the design of the cRKDG method, one cannot directly replace $\dxDG f$ with $\dxl f$ in the inner stages of an SSP-RKDG method in its Shu--Osher form, as this will cause accuracy reduction and violation of local conservation. For example, the following scheme is suboptimal and nonconservative: 
\begin{subequations}\label{eq:ssprk2}
	\begin{align}
		u_h^{(1)} =& u_h^n - \Delta t \dxl f(u_h^n),\label{eq:ssprk2-1}\\
		u_h^{n+1} =& \hf u_h^{n} + \hf \left(u_h^{(1)} - \Delta t \dxDG f\left(u_h^{(1)}\right)\right).\label{eq:ssprk2-2}
	\end{align}
\end{subequations}
This can be made clear by substituting \eqref{eq:ssprk2-1} into \eqref{eq:ssprk2-2}:
\begin{equation}
	u_h^{n+1} = u_h^{n} -\frac{\Delta t}{2} \left( \dxl f\left(u_h^n\right)+ \dxDG f\left(u_h^{(1)}\right)\right).
	\label{eq:ssprk2-wrong_scheme}
\end{equation}
One can see that a low-order and nonconservative spatial operator $\dxl f$ is used to update $u_h^{n+1}$, which will cause trouble. Indeed, if we consider its Butcher form, then the resulting cRKDG method can be written as 
\begin{subequations} \label{eq:ssprk2-correct_scheme}
	\begin{align}
		u_h^{(1)} =& u_h^n - \Delta t \dxl f(u_h^n),\label{eq:ssprk2-correct_scheme-1}\\
		u_h^{n+1} =& u_h^{n} -\frac{\Delta t}{2} \left( \dxDG f\left(u_h^n\right)+ \dxDG f\left(u_h^{(1)}\right)\right),\label{eq:ssprk2-correct_scheme-2}
	\end{align}
\end{subequations}
i.e., replacing the last step \eqref{eq:ssprk2-2} or \eqref{eq:ssprk2-wrong_scheme} by \eqref{eq:ssprk2-correct_scheme-2}. The resulting scheme will have optimal convergence and the provable local conservation property. Some numerical tests will be given in \cref{ex:1dremark}  to illustrate the different convergence rates of these methods. Note that \eqref{eq:ssprk2-correct_scheme} and \eqref{eq:2} are both second-order cRKDG schemes, although they are different as the associated RK methods are different. 

In general, for the RKDG method in Butcher form, we can prove that replacing inner stages with $\dxl f$ will retain local conservation property (see  \cref{thm:lwthm}), and the optimal convergence rate can still be observed numerically (see \cref{ex:1dremark}). 

\subsubsection{Convergence}
Although all inner stages of the cRKDG method \eqref{eq:crkdg-1} do not preserve local conservation, the fully discrete numerical scheme is still conservative as long as the update of $u_h^{n+1}$ in the last stage is discretized with a conservative method. Indeed, for explicit RK methods, the cRKDG scheme can be formally written as a one-step scheme after recursive substitutions. By taking $v_h = 1_{K}$, the characteristic function on $K$, in \eqref{eq:crkdg-last} and combining with \eqref{eq:DG operator}, we obtain
\begin{equation}\label{eq:avg}
	\bar{u}_{h,K}^{n+1} =\bar{u}_{h,K}^n - \frac{\Delta t}{|K|}\sum_{e\in \partial K} g_{e,K}(u_h^n),
\end{equation}
where the bar notation is used to represent the cell average, 
\begin{equation}\label{eq:g}
	g_{e,K}(u_h^n) = \int_e \left(\sum_{i=1}^s b_i\widehat{f\cdot \nu_{e,K}}\left(u_h^{(i)}\right)\right) \dd l
\end{equation}
and $u_h^{(i)}$ can be explicitly computed from \eqref{eq:crkdg-1}.
\eqref{eq:avg} is written in the conservative form. Following \cite{shi2018local}, we obtain a Lax--Wendroff type theorem in \cref{thm:lwthm}, with its detailed proof given in \cref{app:lwthm}. Here we consider scalar conservation laws for simplicity, and the results can also be extended to multidimensional systems. We will use $C$ for a generic constant that may depend on the polynomial order $k$, the stage number $s$, the bound of $u$ and $u_h$, etc., but is independent of $\Delta t$ and $h$. 
\begin{theorem}\label{thm:lwthm}
	Consider the cRKDG scheme \eqref{eq:crkdg} for the hyperbolic conservation law \eqref{eq:conservation law} on quasi-uniform meshes with the following assumptions:
	\begin{enumerate}
		\item $f$ is Lipschitz continuous and $f'$, $f''$ are uniformly bounded in $L^\infty$.
		\item The numerical flux $\widehat{f\cdot\nu_{e,K}}(u_h)$ has the following properties: 
		\begin{enumerate}
			\item Consistency: if $u_h = u_h^\inte = u_h^\ext$, then $\widehat{f\cdot\nu_{e,K}}(u_h)= {f(u_h)\cdot\nu_{e,K}}$.
			\item Lipschitz continuity: $|\widehat{f\cdot\nu_{e,K}}(u_h) -  \widehat{f\cdot\nu_{e,K}}(v_h)|\leq C\nm{u_h^n-v_h^n}_{L^\infty(B_K)}$, where $B_K = K\cup K^\ext$ is the union of $K$ and its neighboring cell $K^\ext$. 
		\end{enumerate}
		\item The CFL condition $\Delta t/h \leq C$ is satisfied.
	\end{enumerate}
	If $u_h^{n}$ converges boundedly almost everywhere to some function $u$ as $\Delta t, h \to 0$, then the limit $u$ is a weak solution to the conservation law \eqref{eq:conservation law}, namely
	\begin{equation}\label{eq:weak}
		\int_{\mathbb{R}^d}u_0 \phi\dd x + \int_{\mathbb{R}^d\times \mathbb{R}^+} u \phi_t \dd x \dd t+ \int_{\mathbb{R}^d\times \mathbb{R}^+}f(u) \cdot \nabla \phi \dd x\dd t= 0, \quad \forall  \phi\in C_0^\infty(\mathbb{R}^d\times\mathbb{R}^+).
	\end{equation} 
\end{theorem}

\subsection{Accuracy}    

In the cRKDG scheme, we apply $\dxl f$ to approximate spatial derivatives of $f$ for all inner stages. However, $\dxl f$ may not have the same approximation property as $\dxDG f$. (For example, for $\mathcal{P}^0$ elements, $\dxl f(u_h) = \Pi\, \nabla \cdot f(u_h) \equiv 0$.) A natural question to ask is whether the cRKDG scheme will still admit the optimal convergence rate in both space and time.  

In \cite{grant2022perturbed}, Grant analyzed perturbed RK schemes with mixed precision. It is pointed out that replacing certain stages in an RK scheme with a low-precision approximation may not affect the overall accuracy of the solver. The order conditions of the methods were systematically studied there. In particular, the work by Grant implies the following results. 

\begin{theorem}[Grant, 2022. \cite{grant2022perturbed}]\label{thm:grant}
	Let $F^\varepsilon(u) = F(u) + \mathcal{O}(\varepsilon)$ be a low-precision perturbation of $F(u)$. Consider a $p$th order RK method for solving the ordinary differential equation $u_t = F(u)$:
	\begin{subequations}\label{eq:ode-zg}
		\begin{align}
			u^{(i)} =& u^n +  \sum_{j = 1}^s a_{ij}  F^\varepsilon\left(u_h^{(j)}\right), \quad  i = 1,2,\cdots, s,\\
			u^{n+1} =& u^n + \sum_{j = 1}^s b_j F\left(u_h^{(j)} \right).
		\end{align}
	\end{subequations}
	The local truncation error of the scheme \eqref{eq:ode-zg} is $ \mathcal{O}\left(\Delta t^{p+1}\right) + \mathcal{O}(\varepsilon \Delta t^2).$
\end{theorem}

The numerical error of the cRKDG scheme \eqref{eq:crkdg} can be analyzed by considering $F = \dxDG f$ and $F^\varepsilon = \dxl f$ in \cref{thm:grant}. Although $\dxl f$ could be low-order accurate, it may yield $\nm{\dxDG f(u_h) - \dxl f(u_h)} \leq C h^{k}$. Hence from \cref{thm:grant}, we expect that the local truncation error after one step is $\mathcal{O}(\Delta t^{p+1} + h^k \Delta t^2+h^{k+1}\Delta t)$. Under the standard CFL condition for hyperbolic conservation laws, we have $\Delta t\leq C h$, hence the local truncation error is $\mathcal{O}(\Delta t^{\min(p+1,k+2)})$. Hence a heuristic global error estimate is $\nm{u_h^n(\cdot) - u(\cdot,t^n)} = \mathcal{O}(\Delta t^{\min(p,k+1)})$, which is the same as the original RKDG scheme. 

However, the above argument is far from rigorous error estimates. For example, estimates in \cref{thm:grant} rely on the derivatives of $F$ and $F-F^\varepsilon$, which are not defined for $\dxDG f$ and $\dxDG f- \dxl f$. A detailed fully discrete error analysis using an energy type argument is still needed and is postponed to our future work. 

\subsection{Boundary Error} 
When the nonhomogeneous Dirichlet boundary condition is used, it is known that, if one directly uses the exact inflow data for RK inner stages, the RKDG method for hyperbolic conservation laws may suffer order degeneration of the accuracy \cite{zhang2011third}. This order degeneration is not specific to DG schemes but can also arise with other spatial discretization methods \cite{carpenter1995theoretical}. 

The reason for such order deductions relates to the fact that some RK stages are designed to be of low stage order, which means that they should be low-order approximations of the true solutions at the corresponding stages. These low-order stages are combined in a subtle way with the coefficients in the Butcher Tableau so that they can build up a high-order accurate solution at $t^{n+1}$. If we replace the low-order RK stage with the exact boundary data, we also break the subtle cancellation of the error terms, which will lead to a low-order accurate approximation at $t^{n+1}$. 

Compared with the RKDG method, the cRKDG method uses the local operator to approximate the values at the inner stages. By doing so, it does not need any exterior information and will avoid introducing the boundary data in the update of the inner stage values. Therefore, the boundary condition is only needed in the last stage to update $u_h^{n+1}$, and this will automatically maintain the optimal convergence rate of the cRKDG method. A detailed analysis involving the nonhomogeneous boundary condition will be provided in future work.  

\subsection{Connections with other DG methods}
\subsubsection{Equivalence with Lax--Wendroff DG in special cases}
DG methods with Lax--Wendroff type time discretization were studied in \cite{qiu2005discontinuous}. The main idea there was to consider the high-order temporal Taylor expansion 
\begin{equation}\label{eq:taylor}
	u(x, t+\Delta t) = u(x, t) + \Delta t u_t (x,t) + \frac{\Delta t^2}{2!} u_{tt}(x,t) + \frac{\Delta t^3}{3!} u_{ttt}(x,t) + \cdots
\end{equation}
and then apply the Lax--Wendroff procedure (or the so-called Cauchy–Kowalewski procedure) to convert all temporal derivatives to spatial derivatives. For example, for the one-dimensional problem with a third-order expansion, this gives
\begin{equation*}
	u^{n+1} = u^n - \Delta t F(u)_x, 
\end{equation*}
where
\begin{equation*}
	\begin{aligned}
		&F(u) = f(u) + \frac{\Delta t}{2}f'(u) u_t + \frac{\Delta t^2}{6}\left(f''(u)(u_t)^2 + f'(u) u_{tt}\right) + \cdots,\\
		&u_t = -f(u)_x, \, u_{tt} = -\left(f'(u) u_t\right)_x, \, u_{ttt} = -\left(f''(u)(u_t)^2 + f'(u) u_{tt}\right)_x, \quad \text{etc.}. 
	\end{aligned}
\end{equation*}
The LWDG method \cite{qiu2005discontinuous} is then given as
\begin{equation*}
	u_h^{n+1} = u_h^n - \Delta t \dxDG F\left(u_h^n\right)
\end{equation*}
with a suitable choice of the numerical flux $\widehat{F\cdot \nu_{e,K}}$.

Although for a generic nonlinear problem, the RK methods suffer the so-called order barrier and we typically need $s>p$ for high-order RK methods. While for linear autonomous problems, it is possible to construct a $p$-stage and $p$th order RK method \cite{gottlieb2001strong}. The following theorem states that  for linear conservation laws with constant coefficients and with a certain choice of the numerical flux, the LWDG method is the same as the cRKDG method with such $p$-stage and $p$th order RK method. 
\begin{theorem}
\label{equivalence}
	Consider linear conservation laws with constant coefficients.  
Suppose $\dxDG f$ is linear in the sense that 
	\begin{equation*}
	    \dxDG f\left(\sum_{j=1}^s b_j u_h^{(j)}\right) = \sum_{j=1}^s b_j \dxDG f\left(u_h^{(j)}\right).
	\end{equation*} Then when a specific numerical flux is chosen, as indicated in \eqref{eq:sameflux}, the LWDG method with $p$th temporal order is equivalent to a cRKDG method using an explicit RK method of $p$th order with $p$ stages. 
\end{theorem}
\begin{proof}
	For ease of notation, we will only consider the scalar conservation law $\partial_t u + \nabla\cdot(\beta u) = 0$, where $\beta$ is a constant vector. But the argument can be similarly generalized to multidimensional linear systems with constant coefficients. We will also denote by $\dxDG f(v_h) = \dxDG (\beta v_h)$ and $\dxl f(v_h) = \dxl (\beta v_h)$ for any $v \in V_h$.
	
	We will show that both schemes can be written in the following form. 
	\begin{equation}\label{eq:linear}
		u_h^{n+1} = u_h^n - \Delta t \dxDG\left(\beta \sum_{i = 0}^{p-1} \frac{\Delta t^{i}}{(i+1)!} L_h^{i} u_h^n\right),\quad \text{with} \quad L_h u_h^n := -\dxl (\beta u_h^n). 
	\end{equation}
	
	Define $L u = -\nabla \cdot (\beta u)$. Using the conservation law, we have
	\begin{equation*}
		u_t = - \nabla \cdot (\beta u) = Lu \quand u_{tt} = - \nabla \cdot \left(\beta u_t\right) =- \nabla \cdot \left(\beta L u\right).
	\end{equation*}
	Using a recursive argument, we have $\partial_t^j u = -\nabla \cdot \left(\beta L^{j-1}u\right)$, for all $j = 0,1,\cdots, p$. Substituting it into the Taylor expansion \eqref{eq:taylor}, we have
	\begin{equation*}
		u^{n+1} = u^n - \Delta t \nabla \cdot F (u^n), \quad F(u^n) = \beta \sum_{i = 1}^p \frac{\Delta t^{i-1}}{i!}L^{i-1} u^n.
	\end{equation*}
	In the LWDG method, we take the solution to be from $V_h$ and approximate the outer spatial derivative $\nabla \cdot F$ with the DG operator $\dxDG F$ to obtain
	\begin{equation}\label{eq:lwdg-1}
		u_h^{n+1} = u_h^n - \Delta t \dxDG F (u_h^n), \quad F(u_h^n) = \beta \sum_{i = 0}^{p-1} \frac{\Delta t^{i}}{(i+1)!} L^{i} u_h^n,
	\end{equation}
 after changing the summation index in the definition of $F$.
	Finally, note that when $\beta$ is constant, we have $Lu_h^n = -\nabla \cdot (\beta u_h^n) = \Pi (-\nabla \cdot (\beta u_h^n)) = -\dxl (\beta u_h^n) := L_h u_h^n$. Hence the LWDG scheme \eqref{eq:lwdg-1} can be written in the form of \eqref{eq:linear}.
	
	Now consider the cRKDG method with explicit time stepping \eqref{eq:crkdg}. Using the linearity of $\dxDG f$, we have
	\begin{equation}\label{eq:rkdg-proof1}
		u_h^{n+1} = u_h^n - \Delta t\dxDG \left(\beta \sum_{i = 1}^p b_i  u_h^{(i)}\right).
	\end{equation}  
	For explicit methods, we can use forward substitution to obtain
	\begin{equation}\label{eq:uhj}
		u_h^{(i)} = \sum_{l = 0}^{i-1} d_{il} \left(\Delta t L_h\right)^l u_h^n
	\end{equation}
 from \eqref{eq:crkdg-1},	where $d_{il}$ is defined recursively as $d_{i0} = 1$ and $d_{il} = \sum_{m=l}^{i-1} a_{im}d_{m,l-1}$.
	Substituting \eqref{eq:uhj} into \eqref{eq:rkdg-proof1} and exchanging the summation indices, we have 
 \begin{equation}\label{eq:rkdg-proof2}
	\begin{aligned}
		u_h^{n+1} &= u_h^n - \Delta t\dxDG \left( \beta \sum_{i = 1}^p\sum_{l = 0}^{i-1}b_id_{il}\left(\Delta t L_h\right)^l u_h^n\right) \\
        &= u_h^n - \Delta t\dxDG \left( \beta \sum_{l = 0}^{p-1}\left(\sum_{i = l+1}^pb_id_{il}\right)\left(\Delta t L_h\right)^l u_h^n\right). 
	\end{aligned}
 \end{equation}
	Note that both $-\dxDG(\beta v_h)$ and $L_h v_h$ are approximations to $\partial_t v_h$. To achieve $p$th order temporal accuracy, we need \eqref{eq:rkdg-proof2} to match the Taylor series after replacing both spatial operators to $\partial_t$. Hence $\sum_{i = l+1}^pb_id_{il} = 1/(l+1)!$. Substituting it into \eqref{eq:rkdg-proof2} gives \eqref{eq:linear}.
    
    Finally, we consider the numerical fluxes of the LWDG method and the cRKDG method. The numerical flux for the LWDG method is $\widehat{F\cdot \nu_{e,K}}(u_h^n)$. For the cRKDG method, from \eqref{eq:rkdg-2}, one can see that the numerical flux in the update of $u_h^{n+1}$ is $\sum_{j = 1}^s b_j \widehat{f\cdot \nu_{e,K}}(u_h^{(j)})$. Therefore, the LWDG method and the cRKDG method are equivalent if we have 
    \begin{equation}\label{eq:sameflux}
        \widehat{F\cdot \nu_{e,K}}\left(u_h^n\right) = \sum_{j = 1}^s b_j \widehat{f\cdot \nu_{e,K}}\left(u_h^{(j)}\right).
    \end{equation}
\end{proof}

\begin{remark}\label{remark:LWDGflux}
For the LWDG method studied in \cite{qiu2005discontinuous}, the authors adopted the Lax--Friedrichs flux
\begin{equation}\label{eq:LWDG:flux}
	\widehat{F\cdot \nu_{e,K}} 
        = \hf\left(F(u_h^\inte)\cdot\nu_{e,K} + F(u_h^\ext)\cdot\nu_{e,K} - {\alpha_{e,K}}\left(u_h^\ext - u_h^\inte\right)\right), 
\end{equation}
where $\alpha_{e,K} = \max|\partial_u f\cdot \nu_{e,K}|$.
Note that the resulted scheme will be different from the cRKDG method with the Lax--Friedrichs flux. This is because the cRKDG method will include the jump terms of the inner stages $u_h^{(j)}$, $j=1,\cdots,s$, but the LWDG method will only include the jump terms of $u_h^n$, which leads to different numerical viscosity used in these numerical fluxes (hence, \eqref{eq:sameflux} is not satisfied). For the linear scalar case, the Lax--Friedrichs flux for the cRKDG method retrieves the upwind flux, whereas the Lax--Friedrichs flux for the LWDG method \eqref{eq:LWDG:flux} does not retrieve the upwind flux. However, if we also choose the upwind flux for $\widehat{F\cdot \nu_{e,K}}$ in the LWDG method, the resulting scheme will be identical to the cRKDG method with the upwind flux (or, equivalently, the Lax--Friedrichs flux). From \cref{equivalence} and the discussion here, one can see that the cRKDG method is closely related to the LWDG method in the linear case, which has been carefully investigated in \cite{qiu2005discontinuous}.
\end{remark}

\begin{remark}
     The maximum CFL numbers of the cRKDG method and the LWDG method in \cite{qiu2005discontinuous}, in comparison to those of the RKDG method, are given in \cref{tab:CFL}. The CFL numbers of the LWDG method (with Lax--Friedrichs flux) are about 60\% to 67\% of those of the corresponding RKDG method. While the cRKDG method achieves the same CFL number as the RKDG method for the second-order case, its CFL number for the third-order case is about 85\% of that of the RKDG method. The difference in the CFL numbers of the cRKDG method and the LWDG method is indeed caused by the difference in the numerical fluxes. 
\end{remark}

\begin{remark}
	For nonlinear problems, the LWDG and the cRKDG methods are completely different. The LWDG method requires to precompute the Jacobian and the high-order derivatives of the flux function, which can be very cumbersome in the implementation of the multidimensional systems. However, the cRKDG method does not involve such complications and can be programmed as that of the original RKDG method. 
\end{remark}

\begin{remark}
	A ``new" LWDG method was proposed in \cite{guo2015new} and then analyzed in \cite{sun2017stability}. Note this method is different from the LWDG method in \cite{qiu2005discontinuous} in the linear case and is also different from the cRKDG method. 
\end{remark}

\subsubsection{Connections with the ADER-DG schemes with a local predictor}
$\,$
The ADER approach was proposed in \cite{toro2001towards,titarev2002ader} as a high-order extension of the classical Godunov scheme. Then the methods have been extended in different ways and under both the finite volume and the DG framework. Here we consider one version of the ADER-DG scheme with a local predictor in \cite{dumbser2008unified,dumbser2008finite}. 

Multiply the equation \eqref{eq:conservation law} with a test function $v_h\in V_h$ over a spacetime element $K \times [t^n,t^{n+1}]$. Replacing $u$ with $u_h \in V_h$ yields
\begin{equation*}
	\int_{t^n}^{t^{n+1}} \int_K \partial_t u_h v_h \dd x \dd t + \int_{t^n}^{t^{n+1}} \int_K \nabla \cdot f(u_h) v_h \dd x \dd t = 0.
\end{equation*}
Note $v_h = v_h(x)$ is independent of $t$. Then we apply Newton-Lebniz rule to $\partial_t$ and perform integration by parts for $\nabla \cdot$, which yields
\begin{equation*}
	\int_K \left(u_h^{n+1} - u_h^{n}\right) v_h \dd x  - \int_{t^n}^{t^{n+1}} \int_K f(u_h) \cdot \nabla v_h \dd x \dd t + \int_{t^n}^{t^{n+1}} \int_{\partial K} f(u_h)\cdot \nu_{e,K} v_h\dd l \dd t = 0.
\end{equation*}
The ADER-DG method with a local predictor can be obtained by replacing $u_h$ in the spatial flux with a local prediction $q_h$ and then replacing the cell interface term with the numerical flux. To be more specific, letting $\mathcal{P}_{K,n}^{k,s}:=\mathcal{P}^k(K)\times \mathcal{P}^s\left([t^n,t^{n+1}]\right)$ be the polynomial space on the spacetime element $K\times[t^n,t^{n+1}]$, the ADER-DG method is written as the following: find $u_h \in V_h$ such that
\begin{equation}\label{eq-ader1}
	\int_K \left(u_h^{n+1} - u_h^{n}\right) v_h \dd x  - \int_{t^n}^{t^{n+1}}\hspace*{-2mm}\int_K f(q_h) \cdot \nabla v_h \dd x \dd t + \int_{t^n}^{t^{n+1}}\hspace*{-2mm}\sum_{e\in \partial K} \int_{e} \widehat{f\cdot\nu_{e,K}}(q_h) v_h\dd l \dd t = 0
\end{equation}
holds for all $v_h \in V_h$, where on each spacetime element $K\times [t^n,t^{n+1}]$, $q_h = q_h(x,t)\in \mathcal{P}_{K,n}^{k,s}$ is a local predictor obtained through a local spacetime Galerkin method
\begin{equation}\label{eq-ader2}
	\int_{t^n}^{t^{n+1}}\int_{K} \left(\partial_t q_h +  \nabla\cdot f(q_h)\right) w_h \dd x \dd t = 0, \quad \forall w_h = w_h(x,t) \in  \mathcal{P}_{K,n}^{k,s},
\end{equation}
with $q_h(x,t^n)=u_h^n(x)$. Note \eqref{eq-ader2} can then be converted into a system of nonlinear equations on each element $K$ and solved locally. 
By adopting the notations in \Cref{sec:construction}, we can write the ADER-DG scheme \eqref{eq-ader1} with the local predictor \eqref{eq-ader2} into the following strong form:
\begin{subequations}\label{eq:aderdg-strong}
	\begin{align}
		u_h^{n+1} =& u_h^n - \int_{t^n}^{t^{n+1}} \dxDG f(q_h) \dd t,\\
		\partial_t q_h =& - \Pi_{\mathrm{st}} \nabla \cdot f \left(q_h\right), \label{eq:aderdg-strong-2}    
	\end{align}
\end{subequations}
where $u_h^n, u_h^{n+1} \in V_h$, $q_h \in \mathcal{P}_{K,n}^{k,s}$, and $\Pi_{\mathrm{st}}$ is the local $L^2$ projection to the spacetime polynomial space $\mathcal{P}_{K,n}^{k,s}$.

To see the connection between \eqref{eq:aderdg-strong} and the cRKDG scheme \eqref{eq:crkdg}, we rewrite both equations in \eqref{eq:aderdg-strong} in the integral form, which yield
\begin{equation}\label{eq:aderdg-strong-int}
	u_h^{n+1} = u_h^n - \int_{t^n}^{t^{n+1}} \dxDG f(q_h) \dd t \quand
	q_h^{n+1} = q_h^n - \int_{t^n}^{t^{n+1}} \Pi_{\mathrm{st}}\nabla \cdot f \left(q_h\right)\dd t.
\end{equation}
Next, we apply a collocation method with $s$ points to integrate \eqref{eq:aderdg-strong-int}. Suppose this method is associated with the Butcher Tableau \eqref{eq:butcher}. Then with $q_h^n = u_h^n$, we get (the equations for $q_h^{n+1}$ and $u_h^{(i)}$ are omitted)
\begin{subequations}\label{eq:aderdg-col}
	\begin{align}
		q_h^{(i)} =&\, u_h^n -  \Delta t\sum_{j = 1}^{s} a_{ij}\left(\Pi_{\mathrm{st}}\nabla\cdot f\right)^{(j)}, \quad  i = 1,2,\cdots, s,\label{eq:aderdg-1}\\
		u_h^{n+1} =&\, u_h^n - \Delta t\sum_{i = 1}^{s} b_i \dxDG f\left(q_h^{(i)}\right),\label{eq:aderdg-last}
	\end{align}
\end{subequations}
where 
\begin{equation*}
	\left(\Pi_{\mathrm{st}}\nabla\cdot f\right)^{(j)} = \left(\Pi_{\mathrm{st}}\nabla\cdot f\left(\sum_{l = 1}^s q_h^{(l)}(x)b_l(t)\right)\right)\bigg|_{t = t+c_j\Delta t},
\end{equation*}
and $\{b_l(t)\}_{l=1}^{s}$ are the Lagrange basis functions associated with the collocation points.

Comparing \eqref{eq:aderdg-col} with \eqref{eq:crkdg}, we can see that after applying a collocation type approximation, the ADER-DG method with a local predictor can be written in a similar form as that of the cRKDG method. The main difference between the two methods comes from the definition of the local operator. In the ADER-DG method, the local operator involves the spacetime projection $K\times [t^n,t^{n+1}]$, which couples all $s$ stages in each of the terms of $\left(\Pi_{\mathrm{st}}\nabla\cdot f\right)^{(j)}$, hence an implicit approach is needed to solve for $q_h^{(i)}$. While in the cRKDG method, the local operator only involves projection in space and hence is decoupled among different temporal stages -- only the term $q_h^{(j)}$ appears in $\dxl f\left(q_h^{(j)}\right)$.

In light of the above analysis, we have the following theorem. 
\begin{theorem}
	Suppose we replace $\Pi_{\mathrm{st}}$ with the spatial only projection $\Pi$ and apply a collocation method with $s$ points to integrate \eqref{eq:aderdg-strong-int}. Then the ADER-DG method with a $\mathcal{P}_{K,n}^{k,s}$ local predictor \eqref{eq:aderdg-strong} becomes a cRKDG method associated with a collocation type RK method with $s$ stages. 
\end{theorem} 

Moreover, for linear conservation laws with constant coefficients, we have the following equivalence theorem. 
\begin{theorem}
    For linear hyperbolic conservation laws with constant coefficients,  the ADER-DG method \eqref{eq-ader1} with a $\mathcal{P}_{K,n}^{k,s}$ local predictor \eqref{eq-ader2} is equivalent to the cRKDG method with the corresponding RK method being the (implicit) Gauss method. 
\end{theorem}
\begin{proof}
For ease of notation, we will only consider the scalar case $f(u) = \beta u$ with $\beta$ being a constant vector. But the system case can be proved similarly. Note that 
\begin{equation*}
    \Pi_{\mathrm{st}} \nabla \cdot f(q_h) =\Pi_{\mathrm{st}} \nabla \cdot (\beta q_h) = \nabla \cdot (\beta q_h) = \Pi \nabla \cdot (\beta q_h)  = \dxl{f(q_h)}\in  \mathcal{P}_{K,n}^{k,s}.
\end{equation*}
Therefore, only considering its dependence on $t$, $\Pi_{\mathrm{st}} \nabla \cdot f(q_h)$ is a temporal polynomial of degree at most $s$. Using the exactness of Gauss quadrature, the temporal integration in \eqref{eq:aderdg-strong-int} can be replaced by an $s$-stage Gauss collocation method. Hence \eqref{eq:aderdg-strong-int} is equivalent to \eqref{eq:aderdg-col} with the RK method chosen as the Gauss method. Moreover, note that
\begin{equation*}
\begin{aligned}
	\left(\Pi_{\mathrm{st}}\nabla\cdot f\right)^{(j)}  
	=& \left(\nabla\cdot \left(\beta \sum_{l = 1}^s q_h^{(l)}(x)b_l(t)\right)\right)\bigg|_{t = t+c_j\Delta t}\\
	=&  \sum_{l = 1}^s \nabla\cdot\left(\beta q_h^{(l)}(x)\right)b_l(t+c_j\Delta t)
	= \nabla\cdot\left(\beta q_h^{(j)}(x)\right) = \dxl f\left(q_h^{(j)}\right). 
\end{aligned}
\end{equation*}    
We can replace $\left(\Pi_{\mathrm{st}}\nabla\cdot f\right)^{(j)}$ by $\dxl f\left(q_h^{(j)}\right)$ in \eqref{eq:aderdg-col}. Then \eqref{eq:aderdg-col} becomes the cRKDG method, where the RK method is the $s$-stage Gauss method.
\end{proof}

\section{Numerical results}
\label{sec:Numerical results}

In this section, we present the numerical results of our cRKDG schemes and compare them with those of the original RKDG schemes. We always couple a $\mathcal{P}^k$-DG method with a $(k+1)$th-order RK method. Unless otherwise stated, for the second- and the third-order schemes, we use the SSP-RK time discretization for the original RKDG method, and \eqref{eq:2} and \eqref{eq:3} for the cRKDG method, respectively. For the fourth- and the fifth-order schemes, we use the classical forth-order RK and the fifth-order Runge–Kutta–Fehlberg methods for both the original and cRKDG methods, respectively. 

\subsection{Accuracy tests}
In this section, we test the accuracy of cRKDG schemes in different settings and demonstrate the effectiveness of cRKDG schemes for handling nonhomogenous Dirichlet boundary conditions.

\subsubsection{One-dimensional tests}
\begin{exmp}[Burgers equations]\label{ex:1dburgers}
	We solve the nonlinear Burgers equation in one dimension with periodic boundary conditions:
	$\partial_{t} u+\partial_{x}\left({u^{2}}/{2}\right)=0$, $x \in(-\pi,  \pi)$, $u(x, 0)=\sin (x)$.
	We use the Godunov flux and compute to $t=0.2$ on both uniform and nonuniform meshes with spatial polynomial degrees $k = 1,2,3,4$.  We compare the numerical results of RKDG and cRKDG methods. Their $L^2$ errors and convergence rates are given in \cref{table_burgers1d}. It can be observed that both schemes could achieve the designed optimal order of accuracy with comparable errors on the same meshes. We have also tested our schemes on randomly perturbed meshes, whose numerical results are omitted due to the space limit, and similar convergence rates are observed.

\begin{table}[h!]
		\centering
		\resizebox{1.0\textwidth}{!}{
			\begin{tabular}{   c|c|c|c  c| c c|cc|cc }
				\hline
		   	&&&\multicolumn{2}{c|}{$k=1$}&\multicolumn{2}{c|}{$k=2$}&\multicolumn{2}{c|}{$k=3$}&\multicolumn{2}{c}{$k=4$} \\
			\cline{4-11}
			 	&&$N$  &$L^2$ error &order&$L^2$ error &order&$L^2$ error &order&$L^2$ error &order\\
				\hline
                \multirow{8}{*}{\rotatebox[origin=c]{90}{{\centering uniform  }}}&
                \multirow{4}{*}{\rotatebox[origin=c]{90}{{\centering RKDG  }}}
                &40& 2.7386e-03& -   &3.8131e-05& -   & 6.3822e-07& -   &1.0505e-08& -   \\
				&&80& 6.9998e-04& 1.97&4.9991e-06& 2.95& 4.1961e-08& 3.93&3.5188e-10& 4.90\\
				&&160&1.7637e-04& 1.99&6.4554e-07& 2.95& 2.7101e-09& 3.95&1.1821e-11& 4.90\\
				&&320&4.4366e-05& 1.99&8.2632e-08& 2.97& 1.7286e-10& 3.97&3.8814e-13& 4.93\\
				\cline{2-11}				
                &\multirow{4}{*}{\rotatebox[origin=c]{90}{{\centering cRKDG}}}
				&40& 2.3502e-03& -   &3.4537e-05& -   & 5.9497e-07& -   &1.0241e-08& -   \\
				&&80& 5.9868e-04& 1.97&4.5379e-06& 2.93& 3.8796e-08& 3.94&3.3912e-10& 4.92\\
				&&160&1.5073e-04& 1.99&5.8341e-07& 2.96& 2.4857e-09& 3.96&1.1335e-11& 4.90\\
				&&320&3.7882e-05& 1.99&7.4902e-08& 2.96& 1.5801e-10& 3.98&3.7040e-13& 4.94\\
                \hline

                \multirow{8}{*}{\rotatebox[origin=c]{90}{{\centering nonuniform  }}}&
                \multirow{4}{*}{\rotatebox[origin=c]{90}{{\centering RKDG }}}
                &40& 4.2044e-03& -&7.2335e-05& -&1.6005e-06& -&3.5190e-08& -\\
				&&80& 1.0118e-03& 2.06&9.6082e-06& 2.91&1.0456e-07& 3.94&1.1728e-09& 4.91\\
				&&160&2.5507e-04& 1.99&1.2302e-06& 2.97&6.8121e-09& 3.94&3.9468e-11& 4.89\\
				&&320&6.4143e-05& 1.99&1.5724e-07& 2.97&4.3541e-10& 3.97&1.2971e-12& 4.93\\
				\cline{2-11}				
                &\multirow{4}{*}{\rotatebox[origin=c]{90}{{\centering cRKDG}}}
				&40&3.7976e-03& -& 6.8122e-05& -&1.5490e-06& -&3.4695e-08& -\\
				&&80&9.0218e-04& 2.07& 8.9388e-06& 2.93&9.8699e-08& 3.97&1.1449e-09& 4.92\\
				&&160&2.2598e-04& 2.00&1.1464e-06& 2.96&6.4244e-09& 3.94&3.8321e-11& 4.90\\
				&&320&5.6822e-05& 1.99&1.4645e-07& 2.97&4.0891e-10& 3.97&1.2563e-12& 4.93\\
                \hline

                
		\end{tabular}
        }
		\caption{$L^2$ error of RKDG and cRKDG methods for the one-dimensional Burgers equation on uniform and nonuniform meshes in \cref{ex:1dburgers}. The nonuniform meshes are generated by perturbing every other node by $h/3$. $\Delta t = 0.1h$ for $k = 1, 2$ and $\Delta t = 0.05 h$ for $k = 3,4$.} \label{table_burgers1d}
	\end{table}
\end{exmp}

\begin{exmp}[Euler equations and validation for \Cref{sec:butcher-RK}]\label{ex:1dremark}
In this test, we solve   the following nonlinear system of one-dimensional Euler equations $\partial_t{u}+\partial_x{f}({u})=0$ on $(0,2)$, 
where ${u}=(\rho, \rho w, E)^{\mathrm{T}}$,  ${f}({u})=\left(\rho w, \rho w^{2}+p, w(E+p)\right)^{\mathrm{T}}$, $E={p}/{(\gamma-1)}+\rho w^{2}/2$ with $\gamma=1.4$. The initial condition is set as
$\rho(x, 0) = 1 + 0.2\sin(\pi x)$, $w(x, 0) = 1$, $p(x, 0) = 1$, and the periodic boundary condition is imposed. The exact solution is $\rho(x, t) = 1 + 0.2\sin(\pi(x- t))$, $w(x, t) = 1$, $p(x, t) = 1$. We use the local Lax--Friedrichs flux to compute to $t = 2$.  

First, in \cref{table:euler accuracy}, we show the numerical error of the cRKDG method with the CFL numbers close to those obtained from the linear stability analysis, which is $0.3$ for the second-order case and $0.16$ for the third-order case. As anticipated, the cRKDG method can  achieve their designed orders of accuracy and remains stable for this nonlinear system when refining the mesh.

Next, we clarify the comments in \Cref{sec:butcher-RK} and explain why the alternations in the cRKDG scheme should be made based on the Butcher form but not the Shu--Osher form. We implement the schemes \eqref{eq:ssprk2-wrong_scheme} and \eqref{eq:ssprk2-correct_scheme} and observe suboptimal convergence for scheme \eqref{eq:ssprk2-wrong_scheme} and optimal convergence for scheme \eqref{eq:ssprk2-correct_scheme}, as shown in \cref{table:ssprk2-wrong_scheme}. We have also done a similar test to examine the schemes based on the third-order SSP-RK (SSP-RK3) time discretization coupled with $\mathcal{P}^2$ spatial polynomials. A suboptimal convergence rate is observed when we modify the scheme in its Shu--Osher form (similar to \eqref{eq:ssprk2}, or equivalently, \eqref{eq:ssprk2-wrong_scheme}), and the optimal convergence rate is observed when we modify the scheme in its Butcher form (similar to \eqref{eq:ssprk2-correct_scheme}).

\begin{table}[h!]
		\centering
		\resizebox{1.\textwidth}{!}{
		\begin{tabular}{  c | cc | cc|cc | cc }
			\hline  
             &\multicolumn{4}{c|}{\eqref{eq:2}, $k=1$}&\multicolumn{4}{c}{\eqref{eq:3}, $k=2$}\\
            \cline{2-9}
			 {$N$}  &$L^2$ error &order&$L^\infty$ error &order&$L^2$ error &order&$L^\infty$ error &order\\
            
   \hline  
20 &  8.6401e-04 &        -  &  1.5693e-03 &        -  &  4.8592e-05 &        -  &  9.6982e-05 &        -  \\
40 &  2.1391e-04 &         2.01 &  4.3033e-04 &         1.87 &  6.3337e-06 &         2.94 &  1.3105e-05 &         2.89 \\
80 &  5.3413e-05 &         2.00 &  1.1235e-04 &         1.94 &  7.9905e-07 &         2.99 &  1.6569e-06 &         2.98 \\
160 &  1.3096e-05 &         2.03 &  2.6785e-05 &         2.07 &  9.9311e-08 &         3.01 &  2.0228e-07 &         3.03 \\
320 &  3.3054e-06 &         1.99 &  7.0220e-06 &         1.93 &  1.2477e-08 &         2.99 &  2.5741e-08 &         2.97 \\
640 &  8.3321e-07 &         1.99 &  1.8091e-06 &         1.96 &  1.5656e-09 &         2.99 &  3.2598e-09 &         2.98 \\
1280& 2.0304e-07 &         2.04 &  4.1002e-07 &         2.14  & 1.9242e-10 &         3.02 &  3.8229e-10 &         3.09     \\
2560&  5.1018e-08 &         1.99 &  1.0552e-07 &         1.96 &    2.4061e-11 &         3.00 &  4.7861e-11 &         3.00        \\
\hline
		\end{tabular}}
		\caption{$L^2$ and $L^\infty$  error for one-dimensional Euler equations in \cref{ex:1dremark}. CFL is $0.3$ for $k=1$ and $0.16$ for   $k=2$.} \label{table:euler accuracy}
	\end{table}

	\begin{table}[h!]
		\resizebox{1.\textwidth}{!}{
		\begin{tabular}{ c|c | c | c | c | c|c|c|c  }
			\hline
            &\multicolumn{4}{c|}{Wrong implementation with Shu--Osher form}&\multicolumn{4}{c}{Correct implementation with Butcher form}\\
            \cline{2-9}
            &\multicolumn{2}{c|}{Scheme \eqref{eq:ssprk2-wrong_scheme}, $k = 1$}&\multicolumn{2}{c|}{SSP-RK3, $k = 2$}&\multicolumn{2}{c|}{Scheme \eqref{eq:ssprk2-correct_scheme}, $k = 1$}&\multicolumn{2}{c}{SSP-RK3, $k = 2$}\\
            \hline
			$N$  &$L^2$ error& order &$L^2$ error &order&$L^2$ error& order &$L^2$ error &order\\
			\hline  
			 20  &  2.2503e-002 &    	     - &  8.42142e-04 &   - 		   &  8.3248e-04  &    -	    &  4.7661e-05 &  		 -     \\
40  &  1.1013e-002 &         1.03  &  2.4793e-04 &         1.93 &  1.9946e-04  &         2.06&  6.1420e-06 &         2.96   \\
80  &  5.4484e-003 &         1.03  &  6.3363e-05 &         1.97 &   4.9608e-05 &         2.01&   7.7938e-07 &         2.98  \\    
			\hline  
		\end{tabular}}
		\caption{$L^2$ error for one-dimensional Euler equations in \cref{ex:1dremark}. CFL is $0.1$.} \label{table:ssprk2-wrong_scheme}
	\end{table}
	
\end{exmp}

\begin{exmp}[Boundary error]\label{ex:1derror} 
    In this example, we test the problem \cite[Section 4]{zhang2011third} to examine the possible accuracy degeneration due to the nonhomogeneous boundary condition. We use the $\mathcal{P}^2$-DG method with upwind flux and the third-order RK scheme to solve $\partial_tu+\partial_xu=0$ on domain $(0,4 \pi)$. The initial condition is set as $u(x,0) = \sin(x)$ and the exact solution is given by $u(x, t)=\sin (x-t)$. Both the periodic and the inflow boundary conditions are considered in our test.
	
	We set  {$\Delta t=0.16 h$} and compute to $t = 20$. Numerical errors and convergence rates are listed in \cref{table:boundary_treatment}. It could be seen that the original RKDG method achieves the optimal convergence rate for the periodic boundary condition but a degenerated rate for the inflow boundary condition. While in contrast, the cRKDG method is able to achieve optimal convergence rates for both types of boundary conditions. 
	
	\begin{table}[h!]
		\centering
		\resizebox{1.\textwidth}{!}{
		\begin{tabular}{ c|c | cc | cc|cc | cc }
			\hline  
            &&\multicolumn{4}{c|}{periodic boundary}&\multicolumn{4}{c}{inflow boundary}\\
            \cline{3-10}
			&{$N$}  &$L^2$ error &order&$L^\infty$ error &order&$L^2$ error &order&$L^\infty$ error &order\\
            \hline
			\multirow{6}{*}{\rotatebox[origin=c]{90}{{\centering RKDG}}}
     &40   & 4.5605E-04&  -  &4.0696E-04&  -    & 3.8572E-04&  -  &3.8889E-04&  -  \\
	&80   & 5.5726E-05& 3.03&5.1832E-05& 2.97  & 4.8763E-05& 2.98&4.8941E-05& 2.99\\
	&160  & 6.9243E-06& 3.01&6.5389E-06& 2.99  & 6.3065E-06& 2.95&7.2957E-06& 2.75\\
	&320  & 8.6412E-07& 3.00&8.2105E-07& 2.99  & 8.4142E-07& 2.91&1.7155E-06& 2.09\\
	&640  & 1.0796E-07& 3.00&1.0286E-07& 3.00  & 1.1738E-07& 2.84&4.1662E-07& 2.04\\
	&1280 & 1.3493E-08& 3.00&1.2873E-08& 3.00  & 1.7331E-08& 2.76&1.0270E-07& 2.02\\
	\hline  
	\multirow{6}{*}{\rotatebox[origin=c]{90}{{\centering cRKDG}}}
	&40   &  1.7656E-03&  -  &7.3382E-04&  -    & 7.3651E-04&  -  &4.6481E-04&  -   \\
	&80   &  2.2030E-04& 3.00&9.0392E-05& 3.02  & 9.0921E-05& 3.02&5.9025E-05& 2.98 \\
	&160  &  2.7536E-05& 3.00&1.1209E-05& 3.01  & 1.1296E-05& 3.01&7.4400E-06& 2.99 \\
	&320  &  3.4428E-06& 3.00&1.3953E-06& 3.01  & 1.4079E-06& 3.00&9.3401E-07& 2.99 \\
	&640  &  4.3036E-07& 3.00&1.7415E-07& 3.00  & 1.7576E-07& 3.00&1.1701E-07& 3.00 \\
	&1280 &  5.3797E-08& 3.00&2.1783E-08& 3.00  & 2.1957E-08& 3.00&1.4643E-08& 3.00 \\
			\hline
		\end{tabular}}
		\caption{Error table for the one-dimensional linear advection equation with periodic and inflow boundary conditions in \cref{ex:1derror}. For the periodic boundary, $u(0, t)=u(4 \pi, t)$; for the inflow boundary, $u(0, t)=\sin (-t)$. $k = 2$ and {$\Delta t = 0.16h$.}}\label{table:boundary_treatment}
	\end{table}
\end{exmp}
\subsubsection{Two-dimensional tests}
The triangular meshes in this section are generated by taking a cross in each cell of the $N\times N$ uniform square meshes.
\begin{exmp}[Euler equations in two dimensions] \label{ex:2deuler}
	We solve the nonlinear Euler equations in two dimensions with the periodic boundary condition: 
	$\partial_t{u}+\partial_x{f}({u})+\partial_y{g}({u})=0$, where $u = (\rho, \rho w, \rho v, E)^T$, $f(u) = (\rho w, \rho w^2 + p, \rho w v, w(E+p))^T$, $g(u) = (\rho v, \rho wv, \rho v^2 + p, v(E+p))^T$, $E={p}/{(\gamma-1)}+\rho (w^2+v^{2})/2$ with $\gamma=1.4$.
	The initial condition is set as $\rho(x, y, 0)=1+0.2 \sin (\pi(x+y))$, $w(x, y, 0)=0.7$, $v(x, y, 0)=0.3$, $p(x, y, 0)=1$. The exact solution is $\rho(x, y, t)=1+0.2 \sin (\pi(x+y-(w+v) t))$, $w=0.7$, $v=0.3$, $p=1$. We use the local Lax--Friedrichs flux and compute the solution up to $t=0.5$. Here the CFL numbers are taken as $0.2$ and $0.12$ for $\mathcal{P}^1$ and $\mathcal{P}^2$ cRKDG methods, and as $0.3$ and $0.18$ for $\mathcal{P}^1$ and $\mathcal{P}^2$ RKDG methods. We list numerical results in \cref{table:2deuler-periodic-rectangle}.  We can observe that both RKDG and cRKDG schemes achieve their expected order of optimal accuracy with comparable numerical errors on both triangular and rectangular meshes.
	
	\begin{table}[h!]
		\resizebox{1.\textwidth}{!}{
			\begin{tabular}{c|c|c|c|c|c|c|c|c|c}
				\hline  
				&&\multicolumn{4}{c|}{RKDG}   &\multicolumn{4}{c}{cRKDG} \\
				\cline{3-10}  
                &&\multicolumn{2}{c|}{$k = 1$}&\multicolumn{2}{c|}{$k = 2$}&\multicolumn{2}{c|}{$k = 1$}&\multicolumn{2}{c}{$k = 2$}\\
                \cline{3-10}
				&$N$ & $L^2$ error & order & $L^2$ error & order & $L^2$ error & order& $L^2$ error & order \\ \hline
	   	\multirow{4}{*}{\rotatebox[origin=c]{90}{{\centering triangular}}}
				 &20    & 4.5685e-04 &        -  & 5.1570e-05 &        -   & 4.4847e-04 &        -   & 4.9510e-05 &        -   \\
			     &40    & 1.1073e-04 &         2.04 & 6.1085e-06 &         3.08  & 1.0842e-04 &         2.05  & 5.8449e-06 &         3.08  \\
				 &80    & 2.7508e-05 &         2.01 & 7.7595e-07 &         2.98  & 2.6859e-05 &         2.01  & 7.4251e-07 &         2.98  \\
				 &160   &  6.8934e-06 &         2.00&  9.6981e-08 &         3.00 &  6.6652e-06 &         2.01 &  9.2728e-08 &         3.00 \\
				\hline  
	
\multirow{4}{*}{\rotatebox[origin=c]{90}{{\centering rectangular}}}
				 &20    & 2.4343e-03 & -& 1.1101e-04 & - & 2.4662e-03 & - & 1.1300e-04 & -  \\
			     &40    & 4.2736e-04 & 2.51& 1.3885e-05 & 3.00 & 4.2767e-04 & 2.53 & 1.4213e-05 & 2.99  \\
				 &80    & 9.0669e-05 & 2.24& 1.7297e-06 & 3.00 & 8.9727e-05 & 2.25 & 1.7737e-06 & 3.00  \\
				 &160   & 2.1445e-05 & 2.08&            2.1586e-07 &3.00   & 2.1140e-05 & 2.09 &           2.2173e-07 &         3.00\\
				\hline  
		\end{tabular}}
		\caption{$L^2$ error for two-dimensional Euler equations with the periodic boundary condition on triangular and rectangular meshes in \cref{ex:2deuler}.}
		\label{table:2deuler-periodic-rectangle}
	\end{table}
\end{exmp}
\begin{exmp}[Boundary error]\label{ex:2dadv} 
	Consider the  linear advection equation in two dimensions
	$\partial_tu+\partial_xu+\partial_y u=0$, $(x, y) \in[-1,1] \times[-1,1]$, $u(x, y, 0)=\sin \left(\pi   x\right) \sin \left(\pi y\right)$,
	The exact solution is $u(x, y, t)=\sin \left( \pi (x-t)\right) \sin \left(\pi (y-t)\right)$. We use the upwind flux and compute the solution up to $t=0.4$ with $\mathcal{P}^3$ elements.  
    The numerical results with both periodic and inflow boundary conditions are given in \cref{table:2dlinear-dirichlet}. It can be observed that, on the same set of triangular meshes, both schemes achieve their designed order of accuracy with comparable numerical error under the periodic condition. For the inflow boundary condition, the RKDG method becomes suboptimal while the cRKDG method remains optimal. 
		\begin{table}[t!]
		\resizebox{1.\textwidth}{!}{
			\begin{tabular}{c|c|c|c|c|c|c|c|c|c}
				\hline  
				&&\multicolumn{4}{c|}{periodic boundary}   &\multicolumn{4}{c}{inflow boundary} \\
				\cline{3-10}  
				&$N$ & $L^2$ error & order & $L^\infty$ error & order & $L^2$ error & order& $L^\infty$ error & order \\ \hline
                \multirow{4}{*}{\rotatebox[origin=c]{90}{{\centering RKDG}}}
                &20      &  1.6470e-06 &         - &  6.1638e-06 &         -& 2.4337e-06 &  - &  2.7118e-05 &  - \\
                &40      &  1.0807e-07 &         3.93 &  4.1438e-07 &         3.89 & 2.8514e-07 &  3.09 &  6.6092e-06 &  2.04 \\
                &80      &  6.7371e-09 &         4.00 &  2.5981e-08 &         4.00 & 4.5845e-08 &  2.64 &  1.6418e-06 &  2.01 \\
                &160     &  4.2131e-10 &         4.00 &  1.6135e-09 &         4.01& 7.9968e-09 &  2.52 &  4.0979e-07 &  2.00 \\
                \hline
                \multirow{4}{*}{\rotatebox[origin=c]{90}{{\centering cRKDG}}}
                &20      &  1.4623e-06 &         - &  4.4906e-06 &         -& 1.7295e-06 &   - &  4.9173e-06 &  - \\
                &40      &  8.9533e-08 &         4.03 &  2.8503e-07 &         3.98 & 1.0770e-07 &   4.01 &  3.0924e-07 &  3.99 \\
                &80      &  5.6015e-09 &         4.00 &  1.7533e-08 &         4.02 & 6.7326e-09 &   4.00 &  1.9284e-08 &  4.00 \\
                &160      &  3.5131e-10 &         3.99 &  1.0718e-09 &         4.03& 4.2143e-10 &   4.00 &  1.2036e-09 &  4.00 \\
				\hline  
		\end{tabular}}
		\caption{Error table for the two-dimensional linear advection equation on triangular meshes with periodic and inflow boundary conditions in \cref{ex:2dadv}. $k = 3$. $\Delta t = h/20$ for the RKDG scheme and $\Delta t = h/30$ for the cRKDG scheme. } 
		\label{table:2dlinear-dirichlet}
	\end{table}	
\end{exmp}

\subsection{Tests with discontinuous solutions} 
We now test the cRKDG method for problems with discontinuous solutions. Only cell averages of the solutions are plotted. For one-dimensional problems, we apply the TVB minmod limiter for systems in \cite{rkdg3} to identify troubled cells and the WENO limiter in \cite{qiu2005discontinuous}  for reconstruction. For  two-dimensional problems, we apply the standard TVB minmod limiters in \cite{rkdg5} to identify troubled cells and reconstruct polynomials, with a suitable parameter $M$ to be specified for each problem.  For the RKDG method, the limiter is applied in every inner stage. For the cRKDG method, the limiter is only applied once at the final stage for each time step. 

Since the cRKDG method is equivalent to the LWDG method for linear problems, they share the same CFL limit. As has been tested in \cite{qiu2005discontinuous} with the von Neumann analysis, this CFL limit will be slightly more restrictive compared with the original RKDG method. In our numerical tests, the CFL numbers are taken as $0.3$ and $0.16$ for $\mathcal{P}^1$ and $\mathcal{P}^2$ cRKDG methods, and as $0.3$ and $0.18$ for $\mathcal{P}^1$ and $\mathcal{P}^2$ RKDG methods, respectively, unless otherwise noted. 

\subsubsection{One-dimensional tests}
\begin{exmp}[Buckley–Leverett equation]\label{ex:1dbl}
We solve two Riemann problems associated with the Buckley--Leverett equation $\partial_t u+\partial_x\left({4 u^2}/{(4 u^2+(1-u)^2)}\right)=0$. The initial condition is set as $u(x,0) = u_l$ for $x<0$ and $u(x,0) = u_r$ for $x\geq 0$, where we have $u_l = 2$ and $u_r = -2$ in the first test and $u_l = -3$ and $u_r = 3$ in the second test. We set $k = 1,2$ and compute to $t = 1$ with $80$ mesh cells. The Godunov flux and WENO limiter with TVB constant $M=1$ are employed in the simulation. We observe that the cRKDG method converges to the correct entropy solutions and its numerical results are in good agreement with the original RKDG method. We have also plotted the solution by the first-order Roe scheme in green dots. In contrast, the solution by the Roe scheme converges to a non-entropy solution.  
\medskip

\begin{figure}[!ht]
	\centering
\begin{subfigure}[t]{.24 \textwidth} 
			\centering	
   \includegraphics[trim=0cm 0cm 0cm 0cm,width=1.0\textwidth]{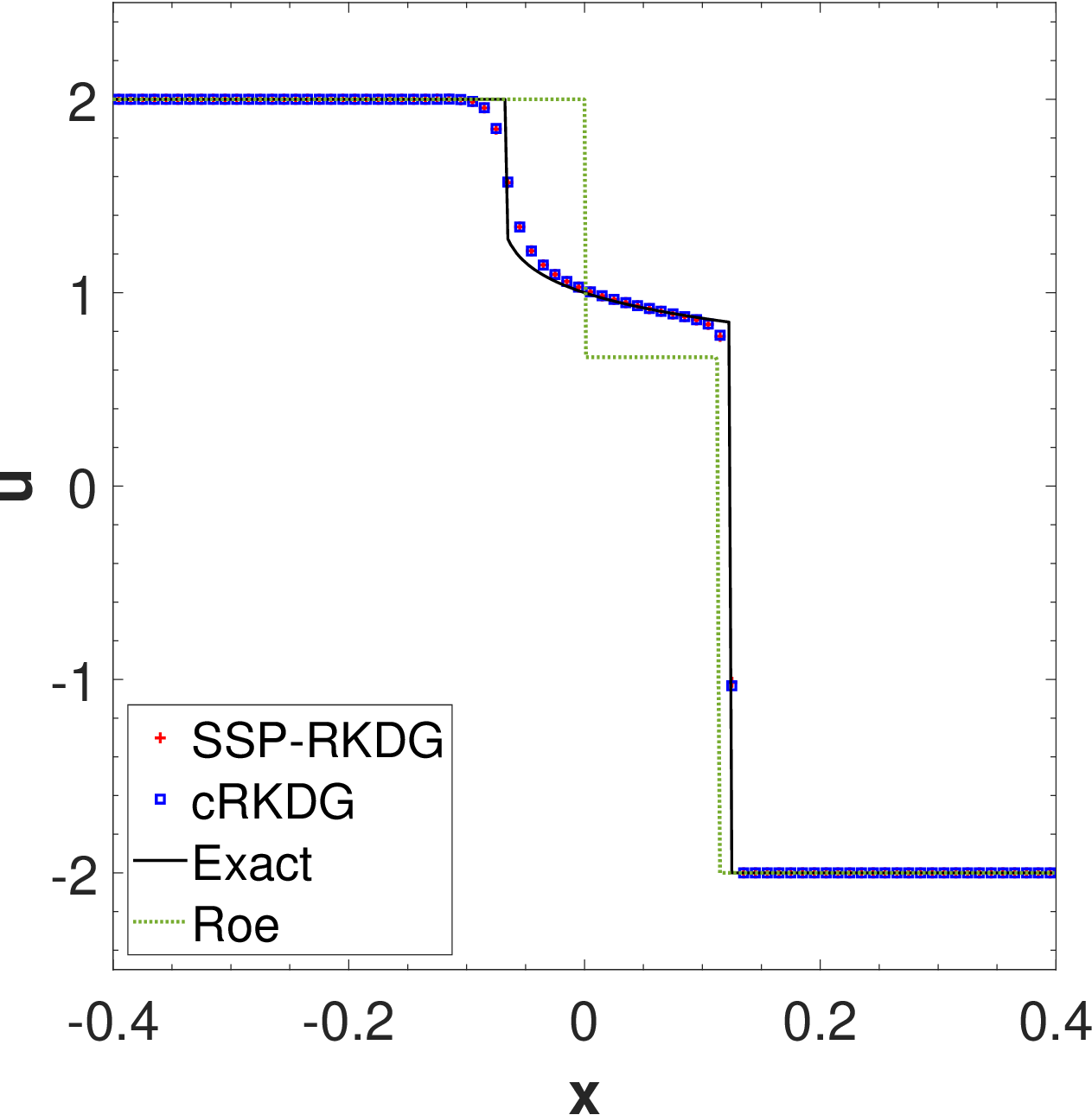}
			\caption{$k=1$, test 1}		 
      \end{subfigure}		 
		\begin{subfigure}[t]{.24 \textwidth}\centering\includegraphics[trim=0cm 0cm 0cm 0cm,width=1.0\linewidth]            {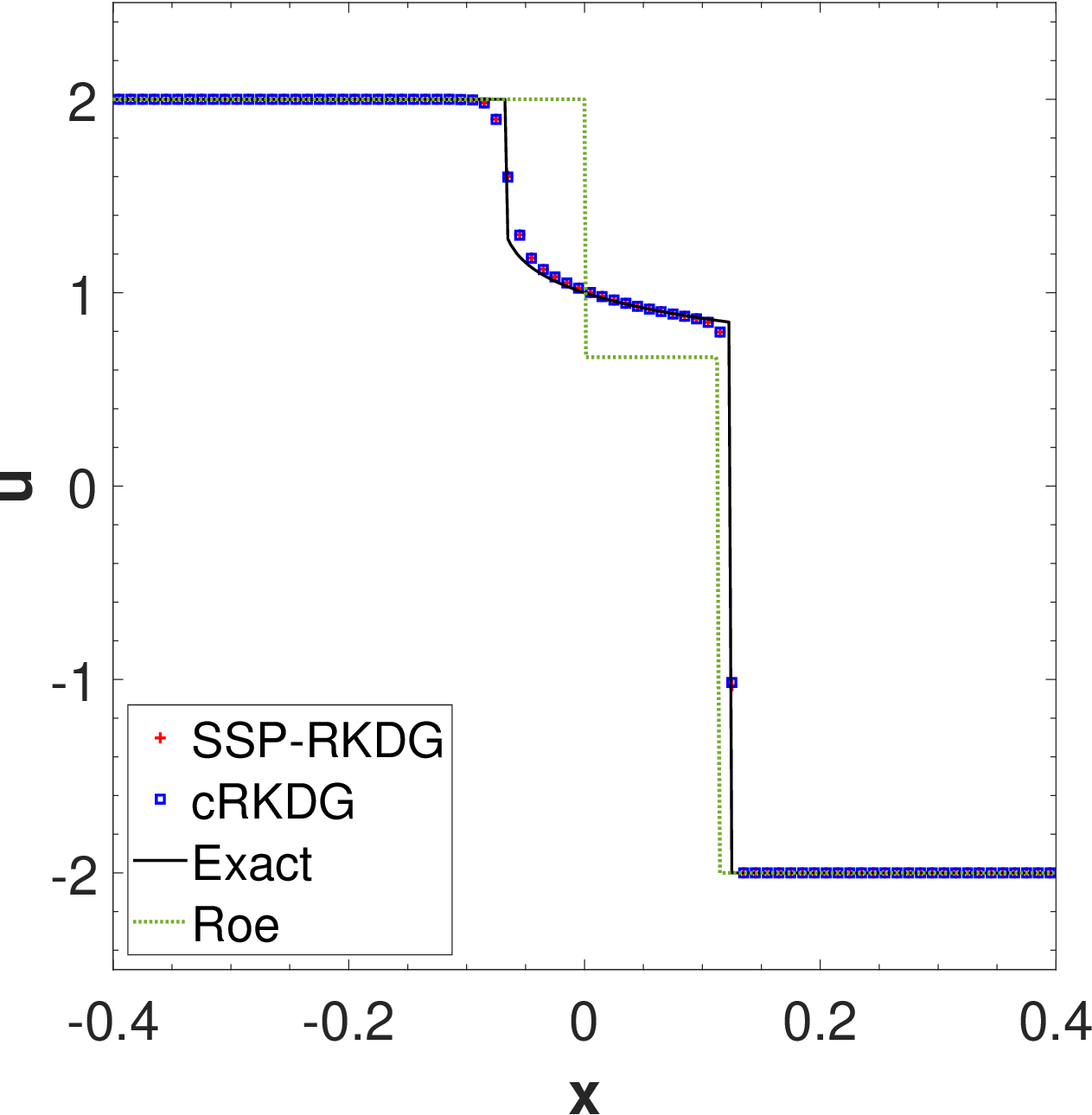}
 			\caption{$k=2$, test 1}
		\end{subfigure}
		\begin{subfigure}[t]{.24 \textwidth}
			\centering
\includegraphics[trim=0cm 0cm 0cm 0cm,width=1.0\linewidth]  {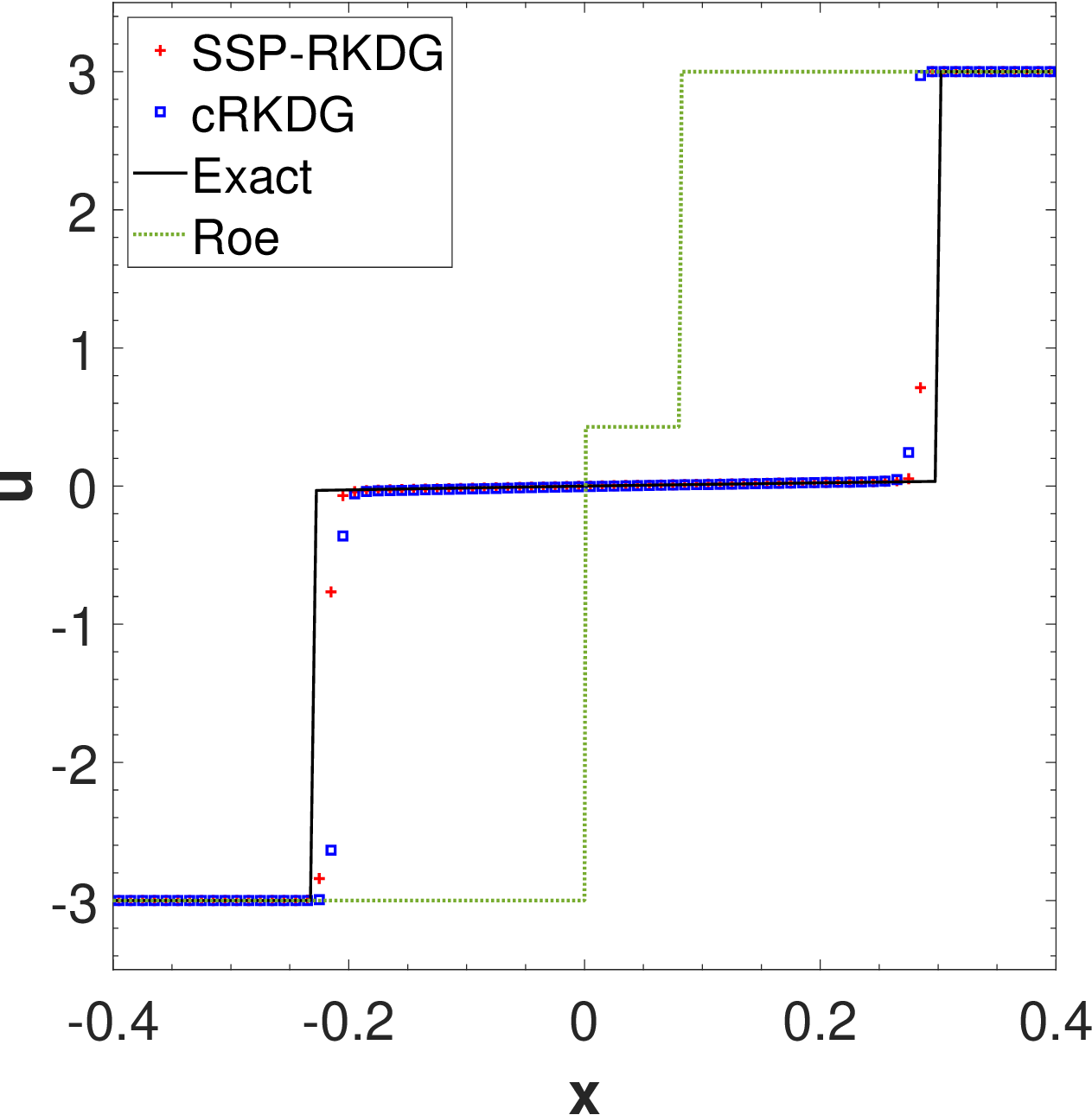}
			\caption{$k=1$, test 2} 
		\end{subfigure} 
		\begin{subfigure}[t]{.24 \textwidth}
			\centering
			\includegraphics[trim=0cm 0cm 0cm 0cm,width=1.0\linewidth]   {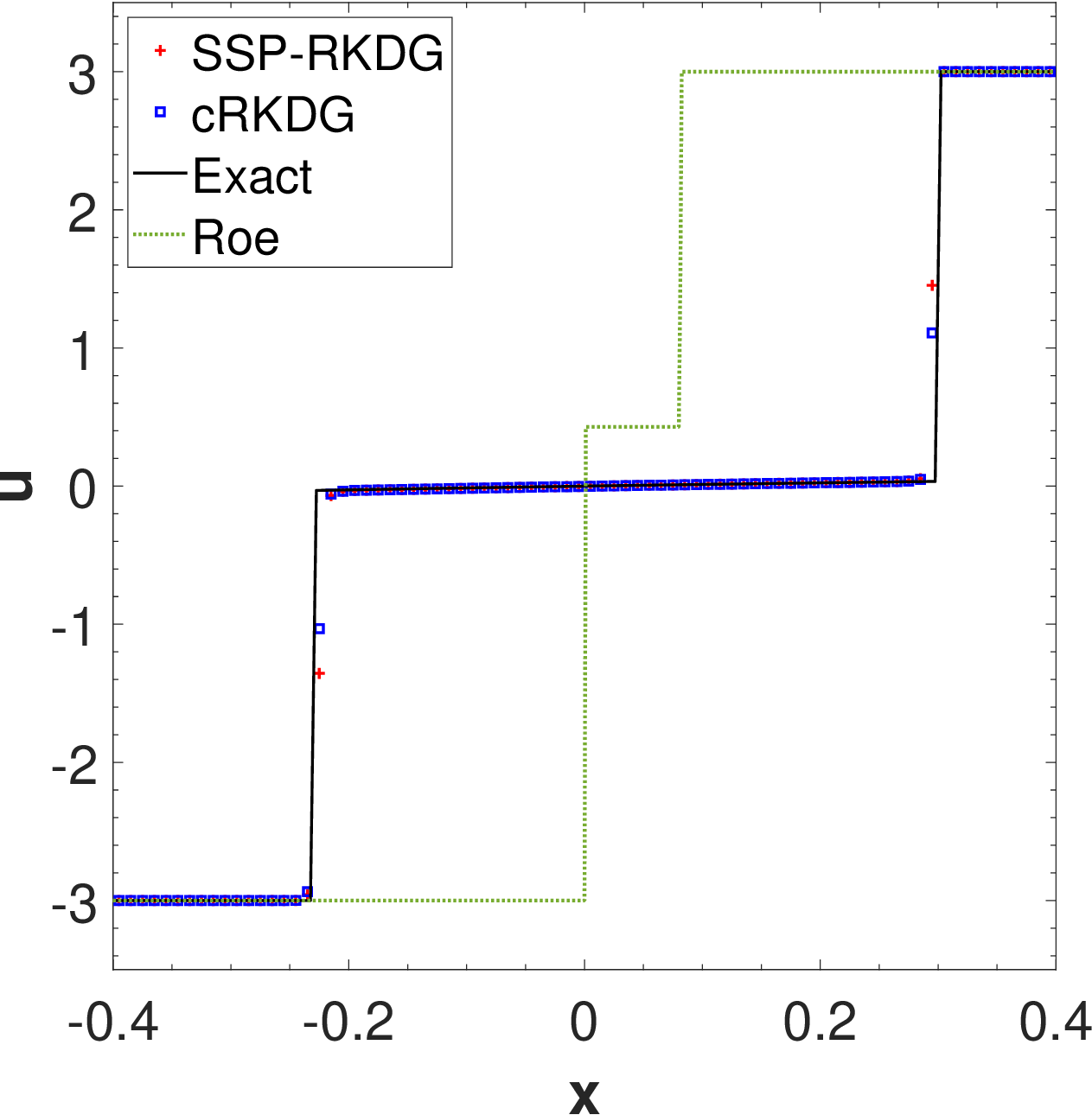}
			\caption{$k=2$, test 2}
		\end{subfigure}
		\caption{Solution profiles for two Riemann problems of the Buckley--Leverett equation in \cref{ex:1dbl}. $N = 80$ and $M = 1$. }
		 
	\end{figure}

\vspace{-0.7em}

\end{exmp}

\begin{exmp}[Sod problem]\label{ex:sod}
In this test, we solve a Riemann problem for the one-dimensional Euler equations given in \cref{ex:1dremark}. The initial condition
	is set as	
\begin{equation*}
	(\rho, w, p)= \begin{cases}(1,0,1), & x \leq 0.5, \\ (0.125,0,0.1), & x> 0.5.\end{cases}
\end{equation*}
 where $\gamma=1.4.$ We compute to $t = 0.2$ with $N=100$ elements. We use the local Lax--Friedrichs flux, WENO
limiter and TVB constant $M=1$. The solution profiles are given in \cref{fig:sod}, from which we can observe that the cRKDG method performs well in capturing the shock and contact discontinuity, and its numerical solution matches the RKDG solution and the exact solution.
	
	\begin{figure}[h!]
            \centering
		\begin{subfigure}[t]{.31\textwidth}
			\centering
			\includegraphics[trim=0cm 1cm 0cm 1cm, width=1. \linewidth]{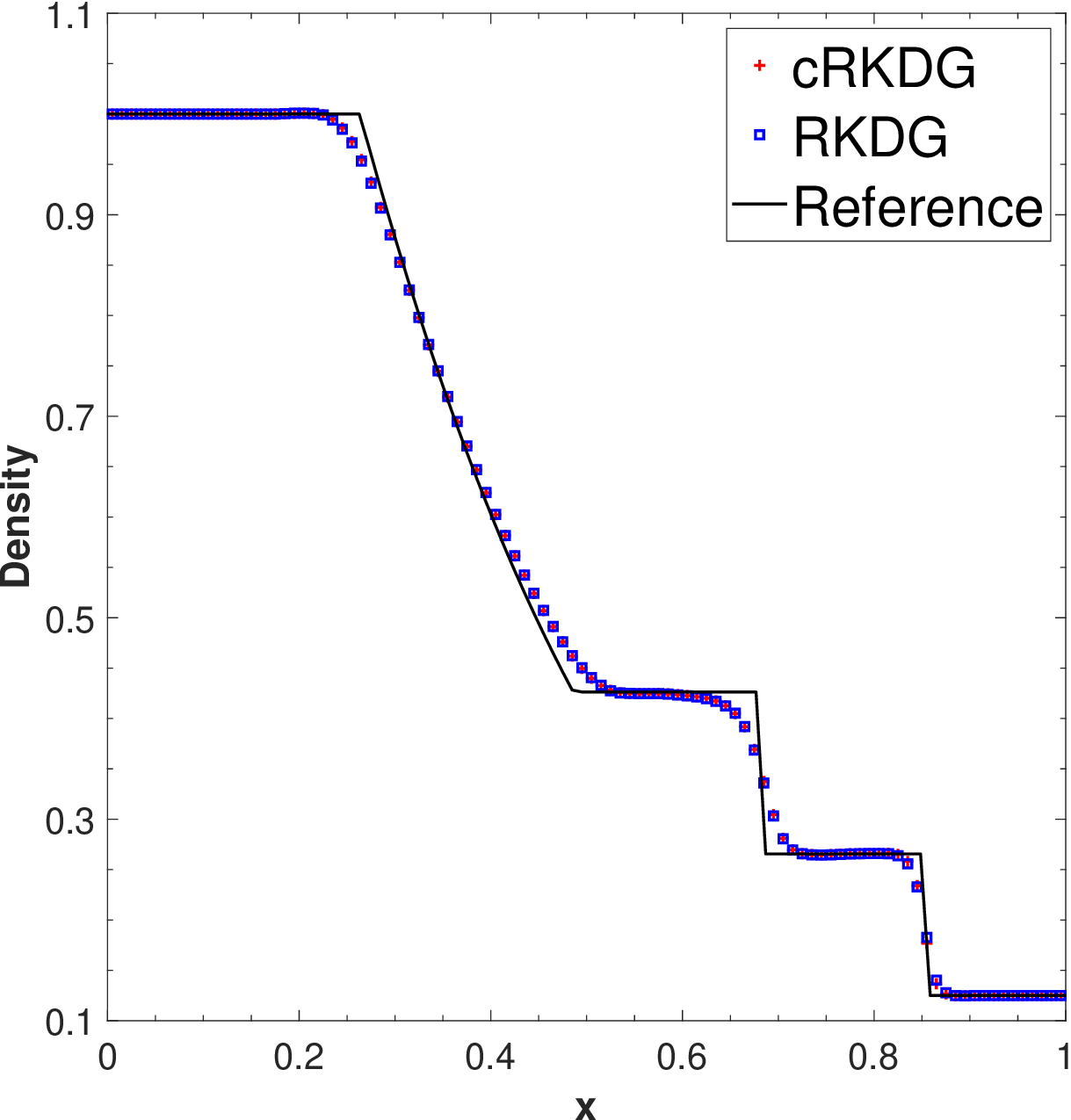}
			\caption{$k=1$ }
		\end{subfigure}%
		\hspace{0mm}
		\begin{subfigure}[t]{.31 \textwidth}
			\centering
			\includegraphics[trim=0cm 1cm 0cm 1cm, width=1. \linewidth]{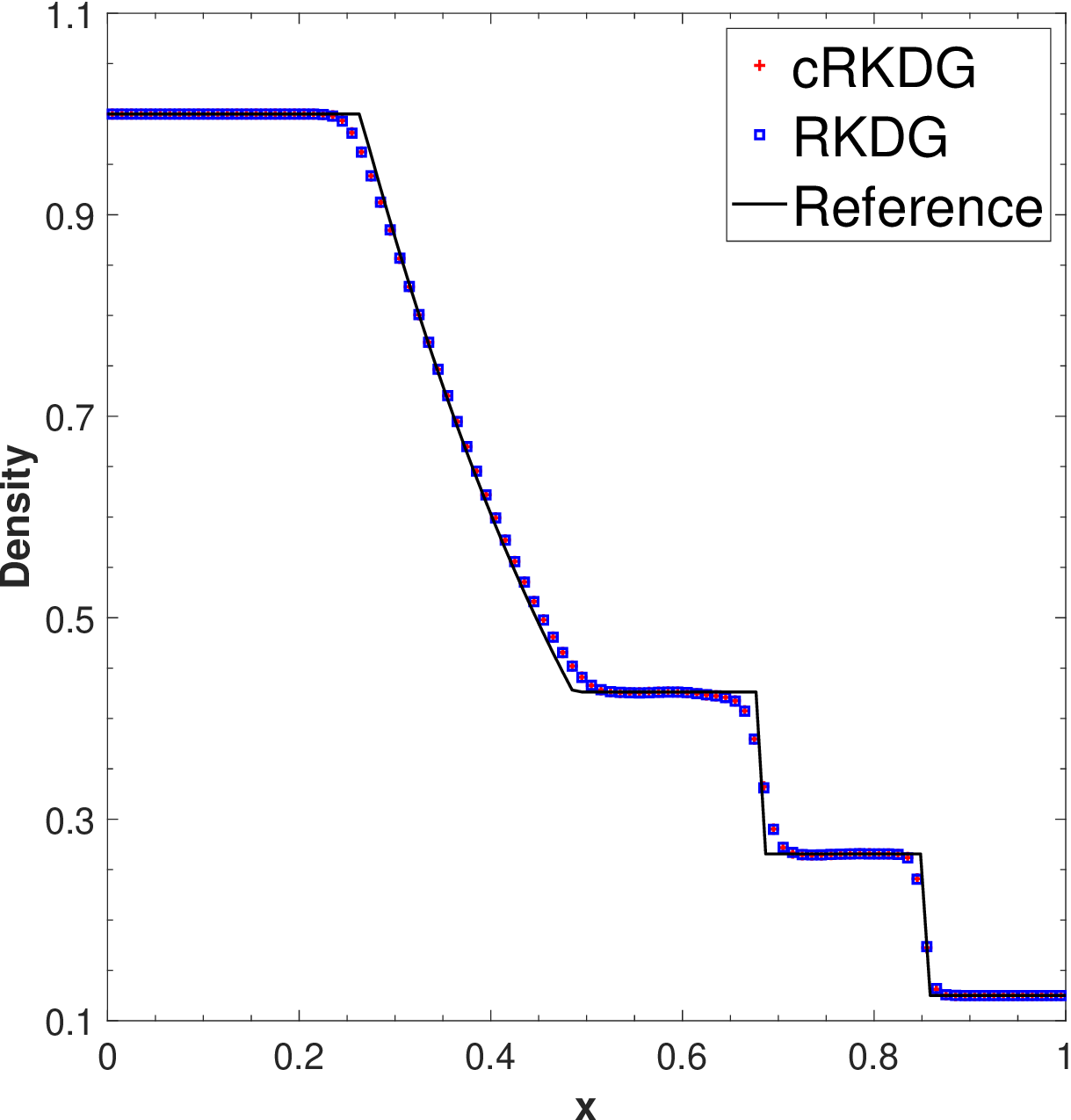}
			\caption{$k=2$}
		\end{subfigure} 
		\hspace{0mm}
		\begin{subfigure}[t]{.31 \textwidth}
			\centering
			\includegraphics[trim=0cm 1cm 0cm 1cm, width=1. \linewidth]{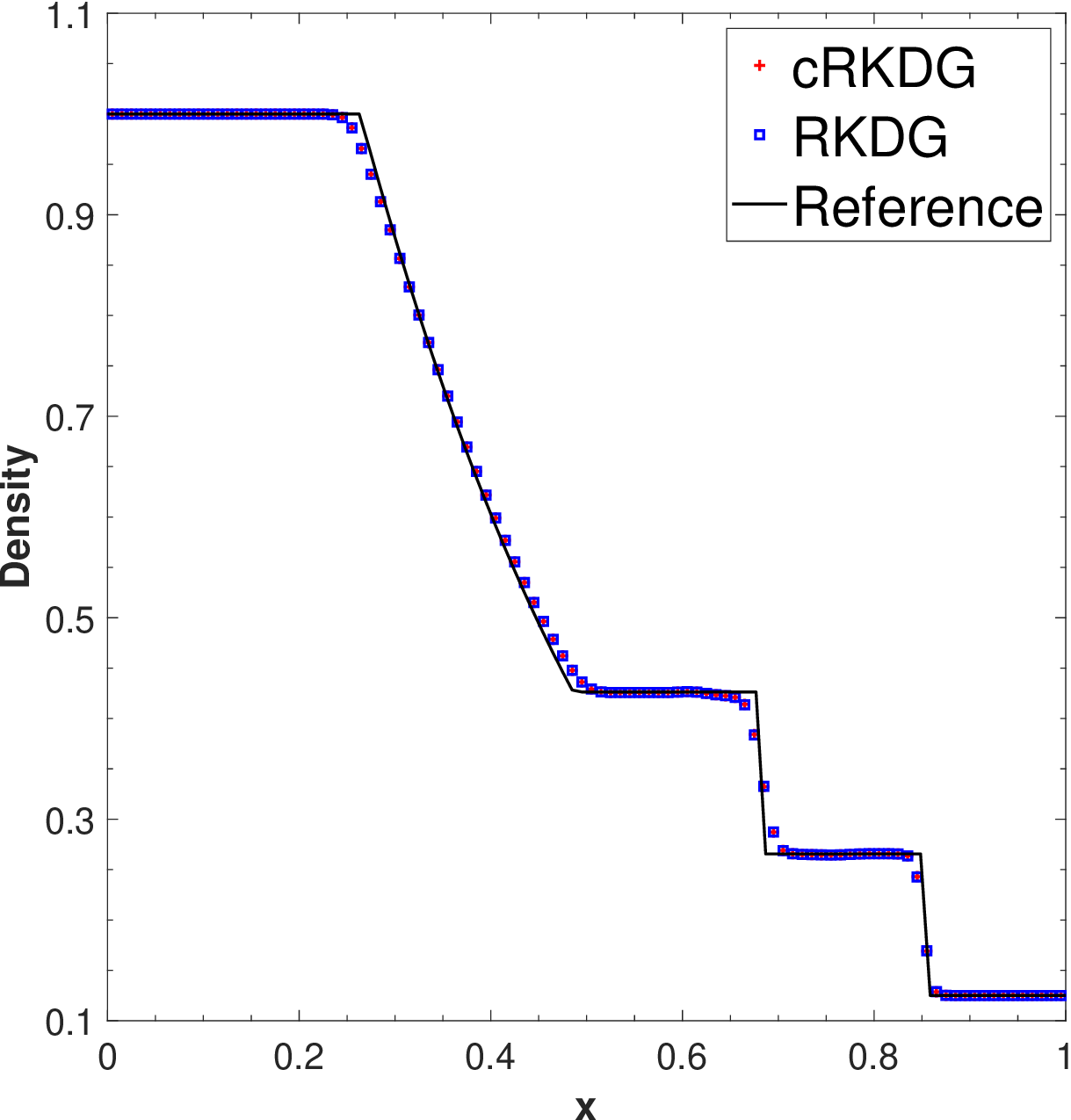}
			\caption{$k=3$}
		\end{subfigure} 
		\caption{Solution profiles for the Sod problem in \cref{ex:sod}. $N=100$ and $M=1$.}
		\label{fig:sod}
	\end{figure}

\end{exmp}

\begin{exmp}[Interacting blast waves]\label{ex:blast}
	We consider the interacting blast waves with Euler equations using the following initial condition
	\begin{equation*}
		(\rho, w, p)= \begin{cases}(1,0,1000), & x \leq 0.1, \\ (1,0,0.01), & 0.1<x \leq 0.9, \\ (1,0,100), & x>0.9.\end{cases}
	\end{equation*}
Reflective boundaries are imposed both at $x=0$ and $x=1$. We use the local Lax--Friedrichs flux, WENO
limiter and TVB constant $M=200$, and compute the solution to $t = 0.038$. Numerical results are shown in \cref{fig:blastwave}. The numerical density $\rho$ is plotted against the reference solution which is a converged solution computed by the fifth-order finite difference WENO scheme on a much refined mesh. It is observed that numerical solutions obtained from RKDG and cRKDG methods are very close.
	\begin{figure}
	\centering
\begin{subfigure}[t]{.24  \textwidth}
			\centering			
   \includegraphics[trim=0cm 1cm 0cm 1cm,width=1.0\textwidth]  {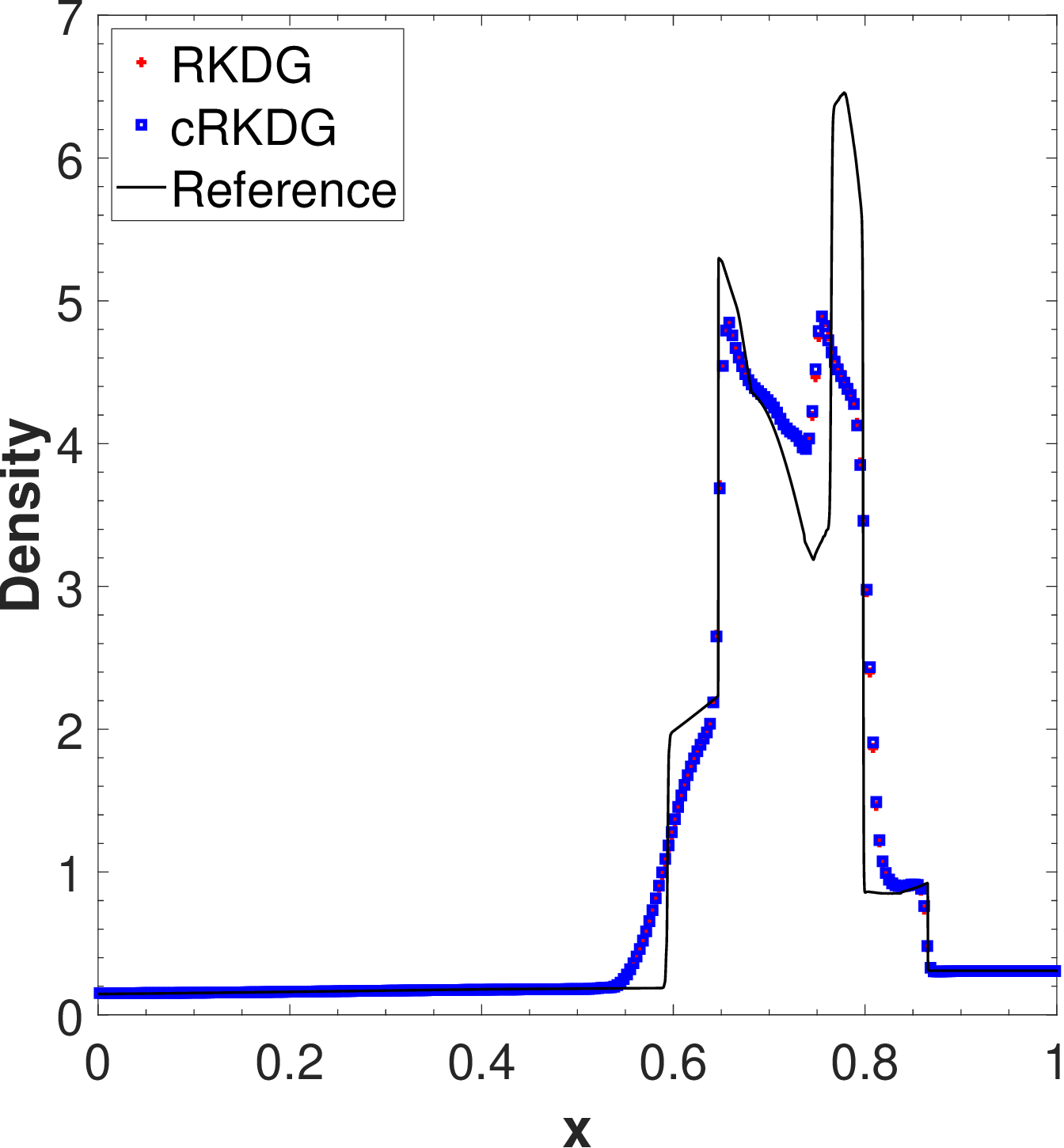}
			\caption{$k=1$}		 
      \end{subfigure}		 
		\begin{subfigure}[t]{.24  \textwidth}
			\centering			\includegraphics[trim=0cm 1cm 0cm 1cm,width=1.0 \linewidth]            {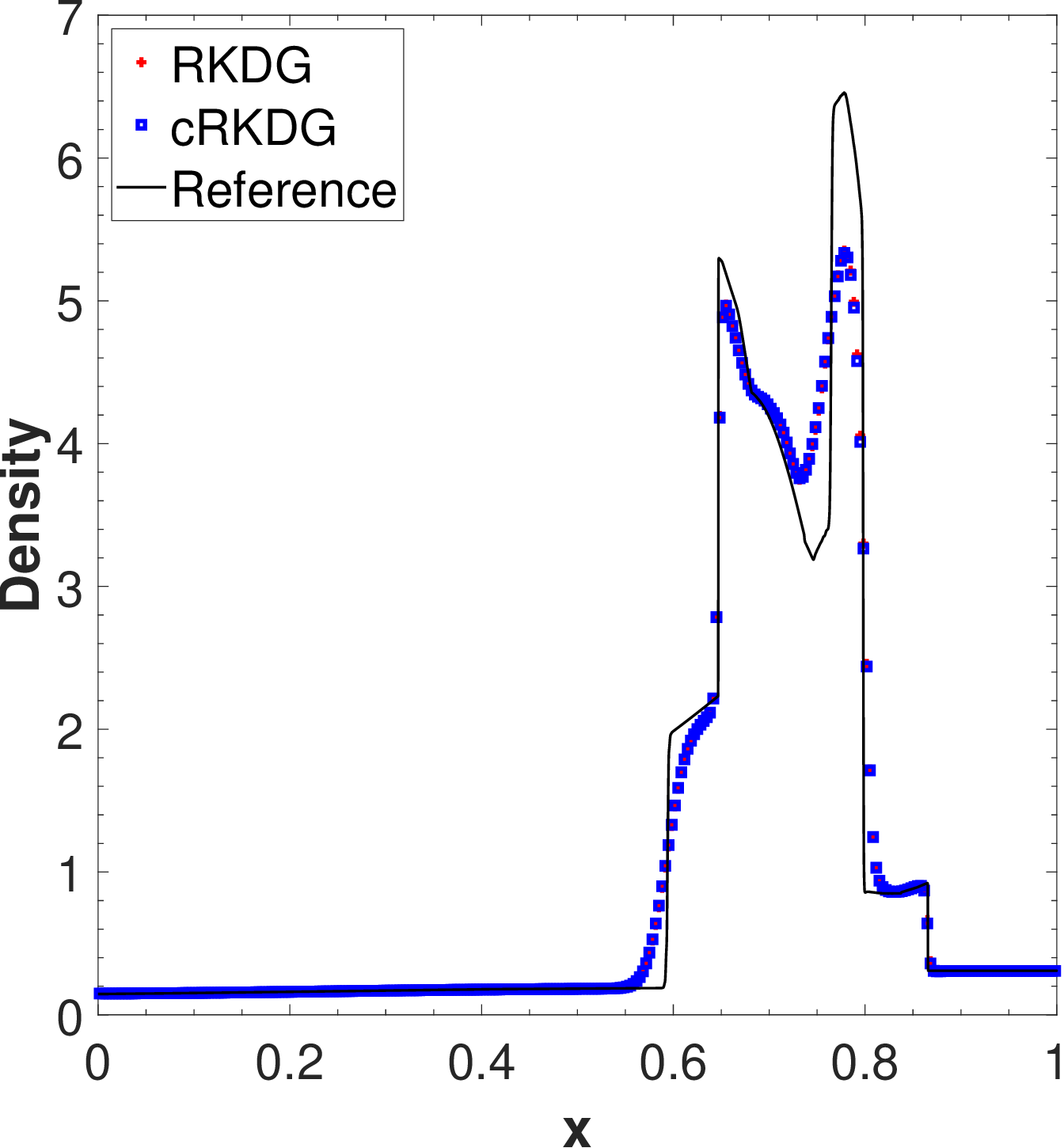}
			\caption{$k=2$}
		\end{subfigure}
		\begin{subfigure}[t]{.24 \textwidth}
			\centering			\includegraphics[trim=0cm 1cm 0cm 1cm,width=1.0 \linewidth]            {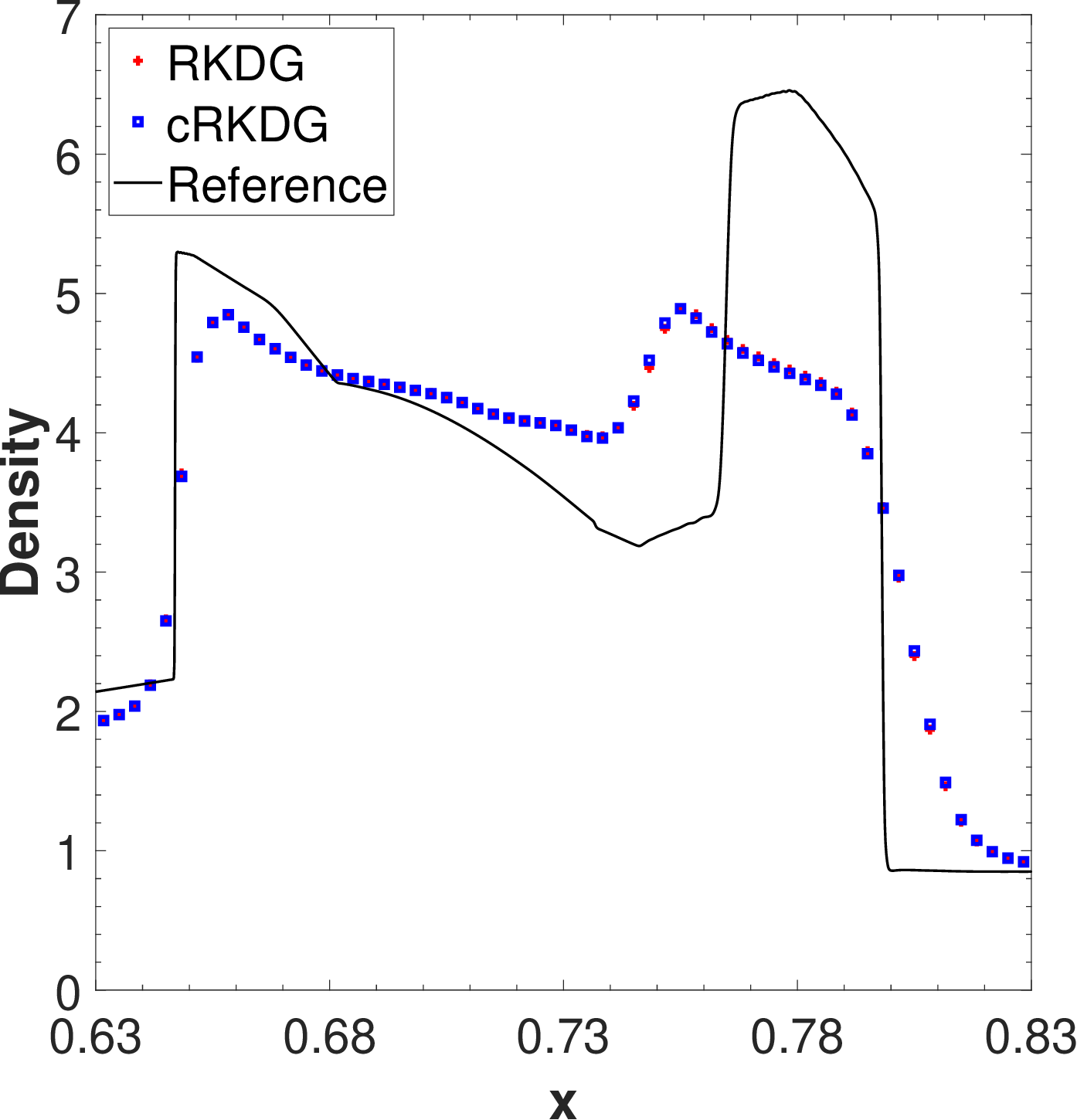}
			\caption{$k=1$ (Zoomed-in)}\label{fig:1dblastwave-zoom1}
		\end{subfigure}%
		\begin{subfigure}[t]{.24 \textwidth}
			\centering			\includegraphics[trim=0cm 1cm 0cm 1cm,width=1.0 \linewidth]   {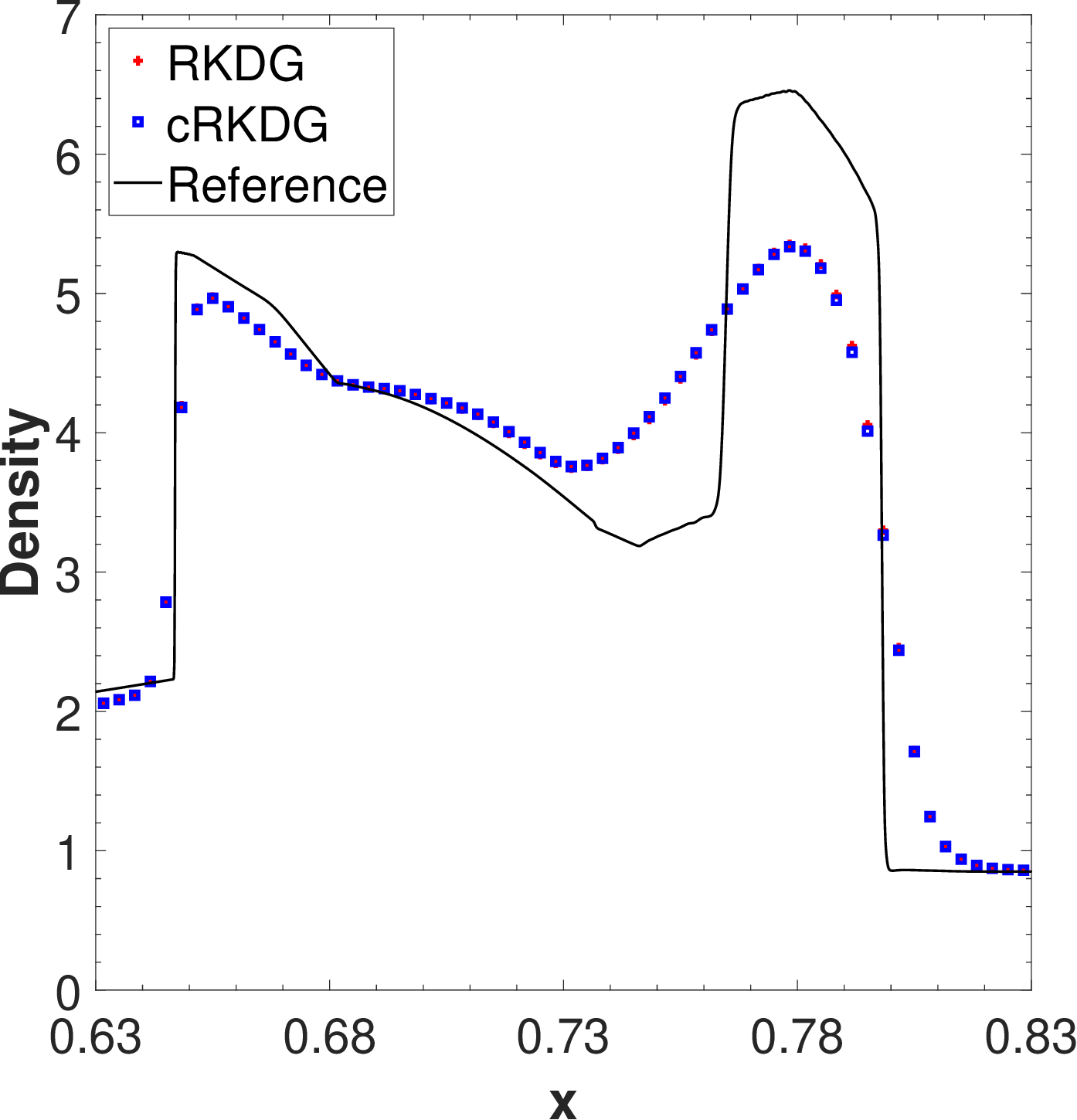}
			\caption{$k=2$ (Zoomed-in)}\label{fig:1dblastwave-zoom2}
		\end{subfigure}
		\caption{Solution profiles for the blast wave problem in \cref{ex:blast}. $N=300$ and $M=200$.  }
		\label{fig:blastwave}
	\end{figure}
\end{exmp}

\begin{exmp}[Shu--Osher problem]\label{ex:shuosher}
	We consider the Shu--Osher problem describing a Mach 3 shock interacting with sine waves in density. This is a problem of shock interaction with entropy waves  and thus contains both shocks and complex smooth region structures \cite{SHU198932}. The initial condition is set as
	\begin{equation*}
		(\rho, w, p)= \begin{cases}(3.857143,2.629369,10.333333), & x <-4, \\ (1+0.2 \sin (5 x), 0,1), & x\geq-4 .\end{cases}
	\end{equation*}
	The numerical density $\rho$ is plotted at $t=1.8$ against the reference solution which is computed by the fifth-order finite difference WENO scheme. In \cref{fig:shu-osher}, we plot the densities by cRKDG and RKDG methods with the local Lax--Friedrichs flux, WENO
limiter and TVB constant $M=200$. In addition, we also show a zoomed-in view of the solution at $x\in[0.5,2.5]$ in \cref{fig:shuosher-zoom1} and \cref{fig:shuosher-zoom2}. 
	
	\begin{figure}[h!]
	\centering
		\begin{subfigure}[t]{.24\textwidth}
			\centering
			\includegraphics[trim=0cm 1cm 0cm 1cm,width=1. \linewidth,height=1.\linewidth]{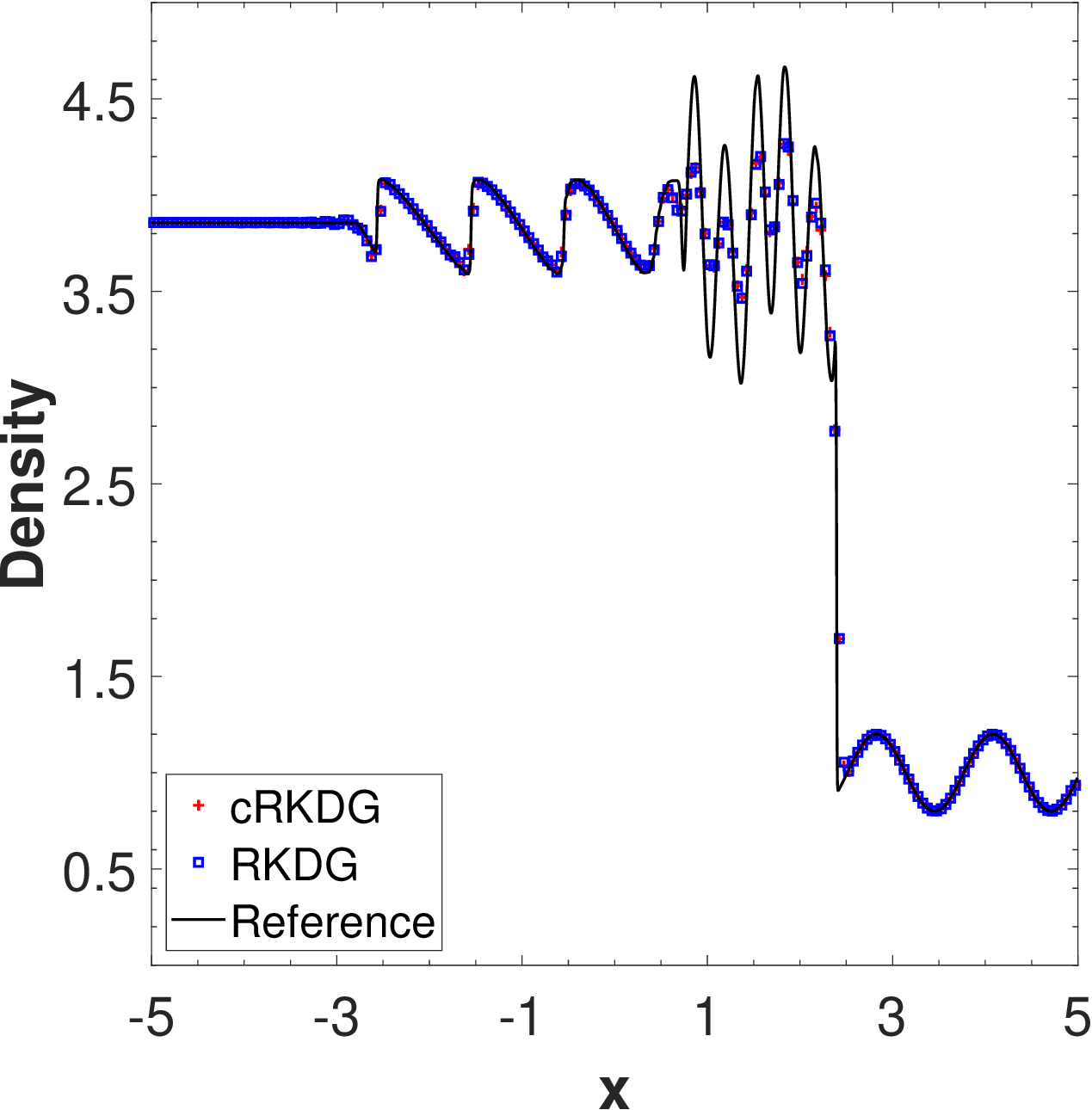}
			\caption{$k=1$}
		\end{subfigure}
		\begin{subfigure}[t]{.24\textwidth}
			\centering
			\includegraphics[trim=0cm 1cm 0cm 1cm,width=1. \linewidth,height=1.\linewidth]{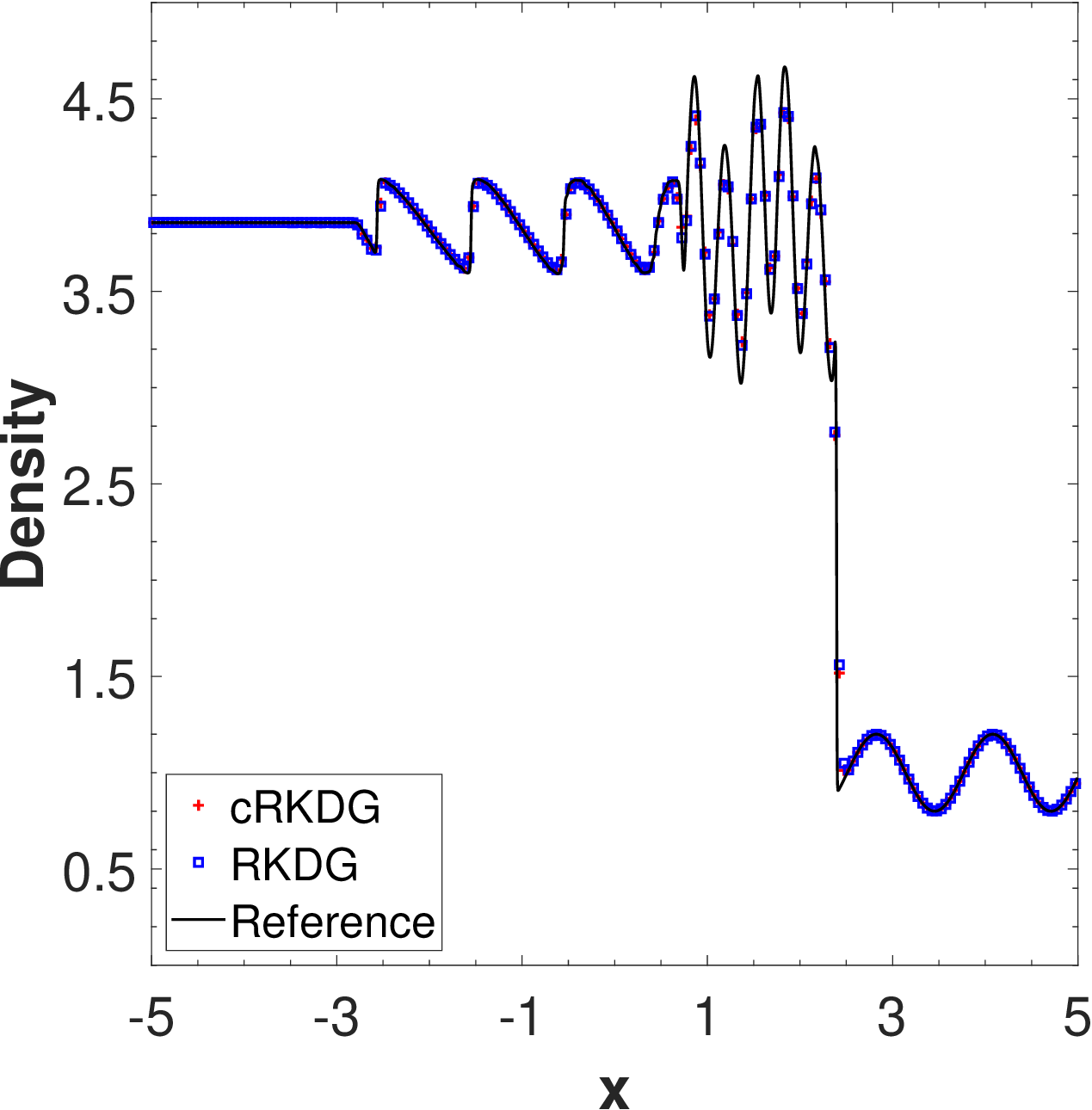}
			\caption{$k=2$}
		\end{subfigure} 
		\begin{subfigure}[t]{.24\textwidth}
			\centering
			\includegraphics[trim=0cm 1cm 0cm 1cm,width=1. \linewidth,height=1.\linewidth]{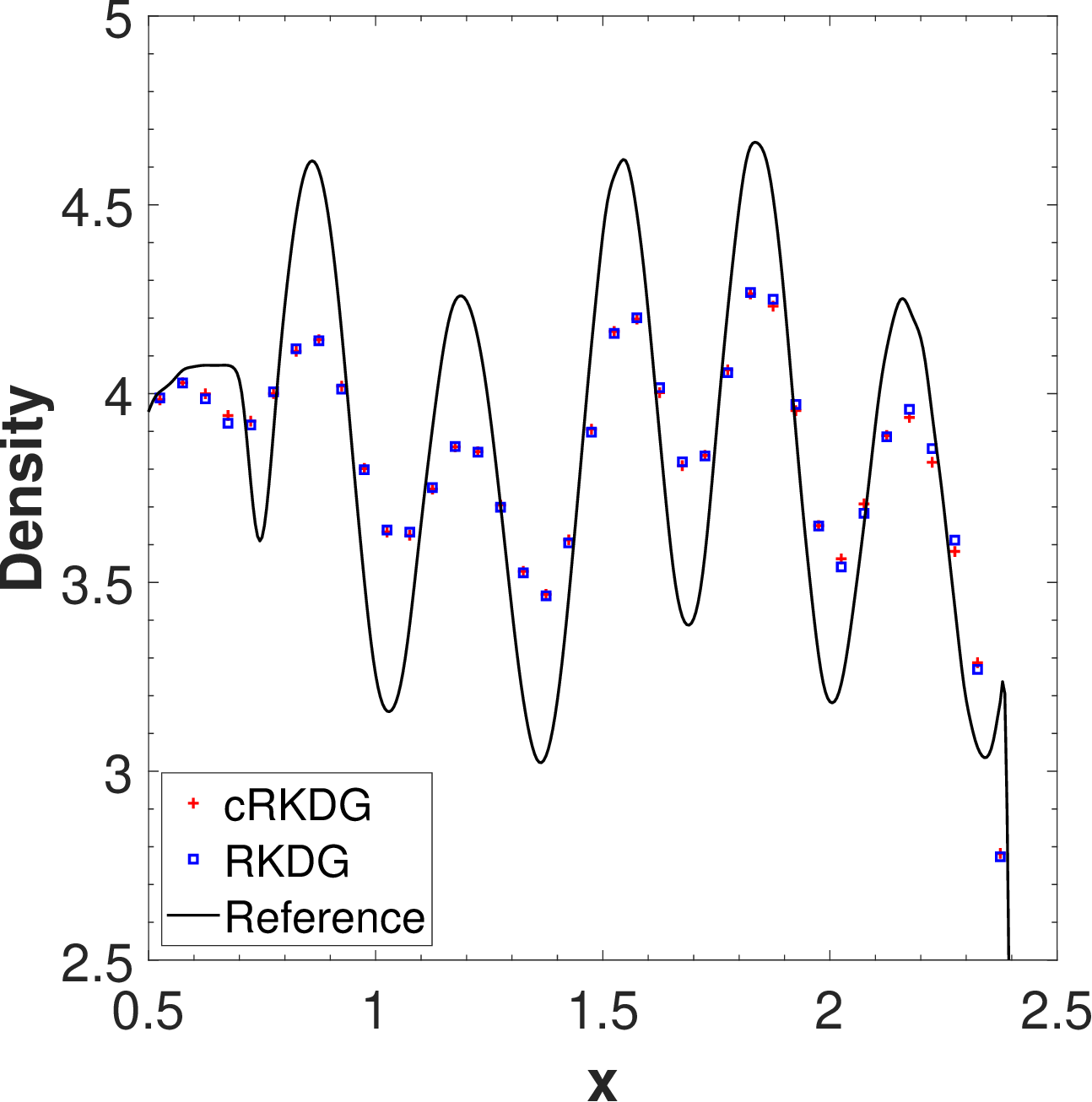}
			\caption{$k=1$ (Zoomed-in)}\label{fig:shuosher-zoom1}
		\end{subfigure}
		\begin{subfigure}[t]{.24\textwidth}
			\centering
			\includegraphics[trim=0cm 1cm 0cm 1cm,width=1. \linewidth,height=1.\linewidth]{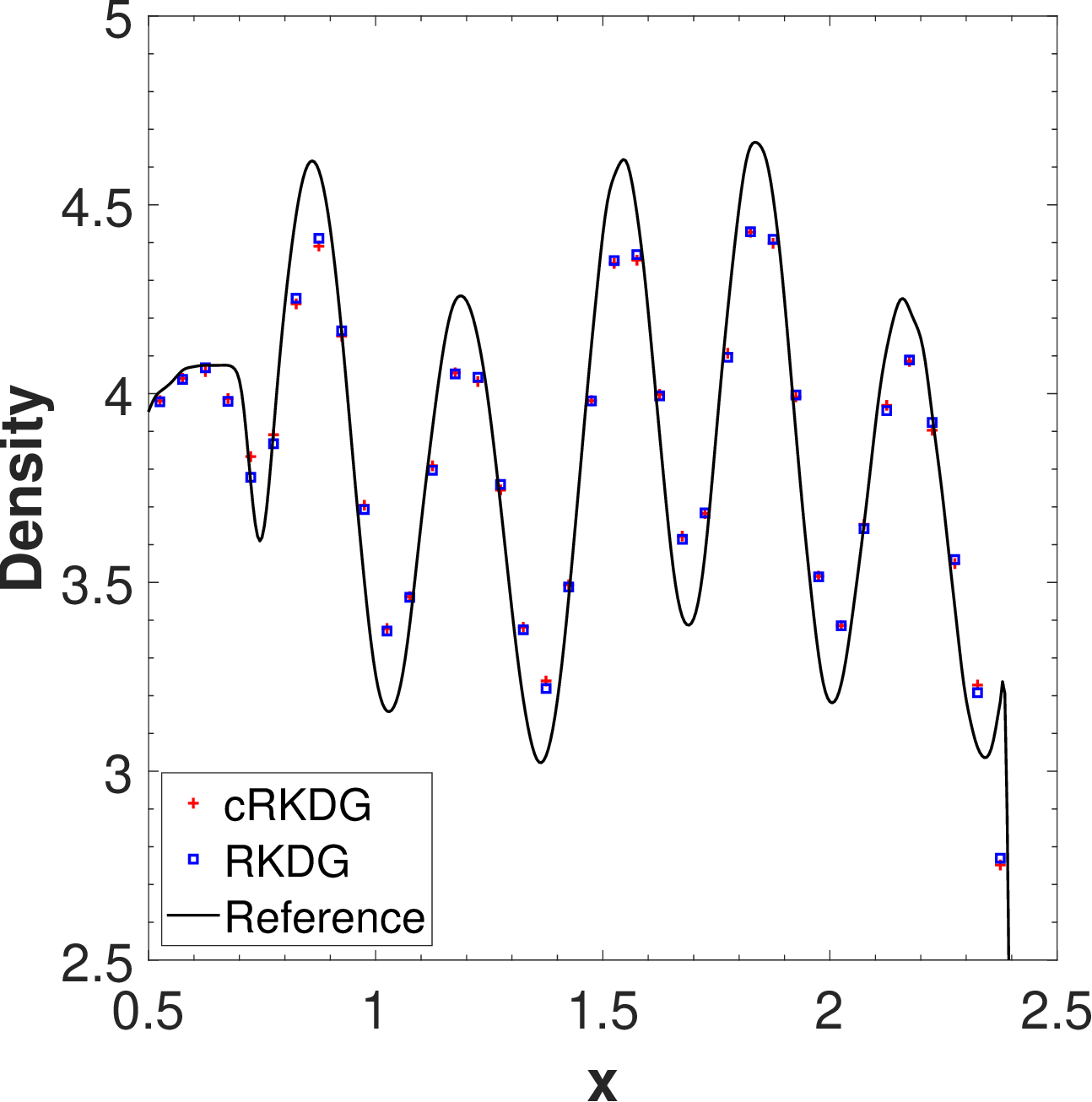}
			\caption{$k=2$ (Zoomed-in)}\label{fig:shuosher-zoom2}
		\end{subfigure} 
		\caption{Solution profiles for the Shu--Osher problem in \cref{ex:shuosher} at $t=1.8$. $M=200$ and $N=200$.}
		\label{fig:shu-osher}
	\end{figure}	
\end{exmp}

\subsubsection{Two-dimensional tests}
\begin{exmp}[Double Mach reflection]\label{ex:doublemach}
This problem is originally studied in \cite{woodward1984numerical} and describes reflections of planar shocks in the air from wedges. The computational domain is $[0,4] \times[0,1]$, and the reflecting wall lies at the bottom boundary, starting from $x={1}/{6}$. Therefore, for the bottom boundary, the exact post-shock condition is imposed for the region from $x=0$ to $x={1}/{6}$, while a reflective boundary condition is applied to the rest. At $t=0$, a right-moving Mach 10 shock is positioned at $x={1}/{6}$, $y=0$ and makes a $60^{\circ}$ angle with the $x$-axis.  At the top boundary, the flow values are set to describe the exact motion of the Mach 10 shock. The boundary conditions at the left and the right are inflow and outflow respectively. We compute the solution up to $t = 0.2$ and use the TVB limiter with $M = 50$. To save space, we only present the simulation results with $480\times 120$ mesh cells for $k = 1$ and $1960\times480$ mesh cells for $k = 1,2$ in \cref{fig:doublemach-big}. The corresponding zoomed-in figures around the double Mach stem are given in \cref{fig:doublemach-small}. For this problem, the resolutions of cRKDG and RKDG methods are comparable for the same order of accuracy and the same meshes. 
	\begin{figure}[h!]
	\centering
 \begin{subfigure}[t]{0.49\textwidth}
			\centering
			\includegraphics[trim=1cm 0cm 3cm 0cm, width=\textwidth]{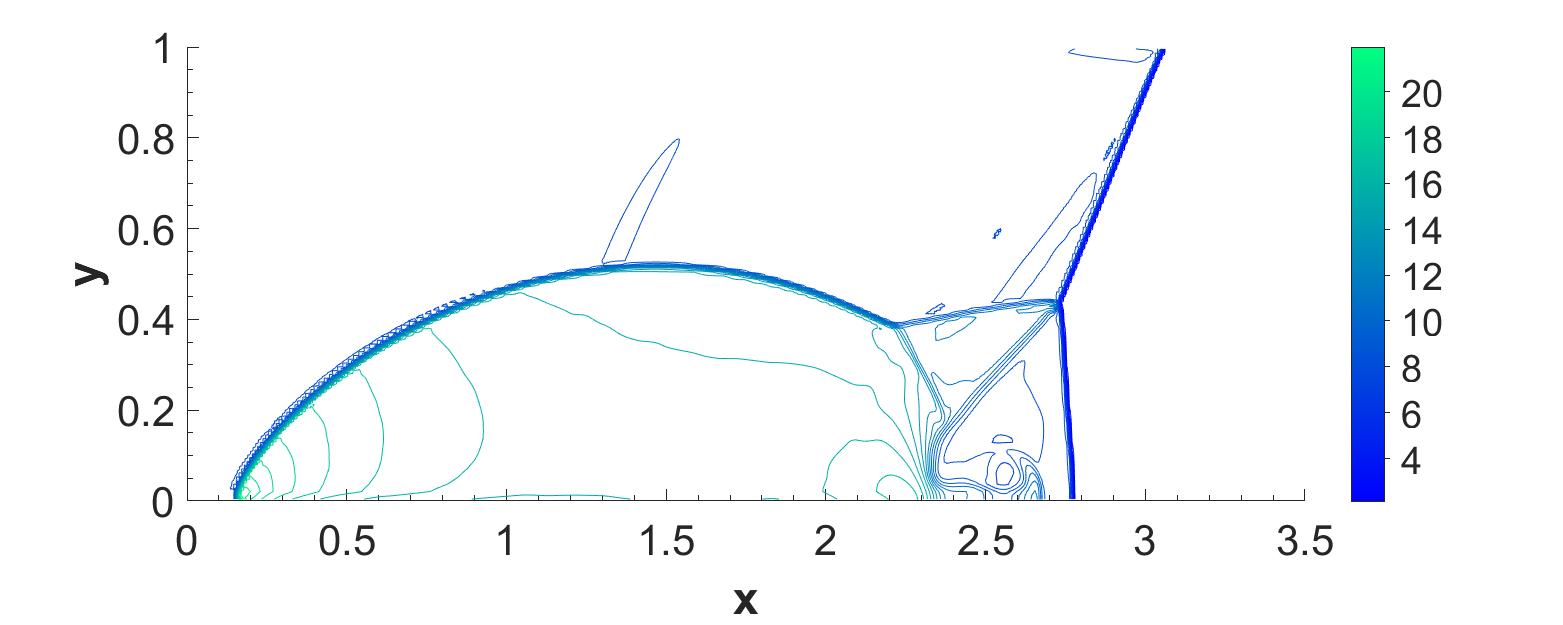}
			\caption{$k=1$, cRKDG, $480\times 120$ mesh}
		\end{subfigure}		 
		\begin{subfigure}[t]{0.49\textwidth}
			\centering
			\includegraphics[trim=1cm 0cm 3cm 0cm, width=\textwidth]{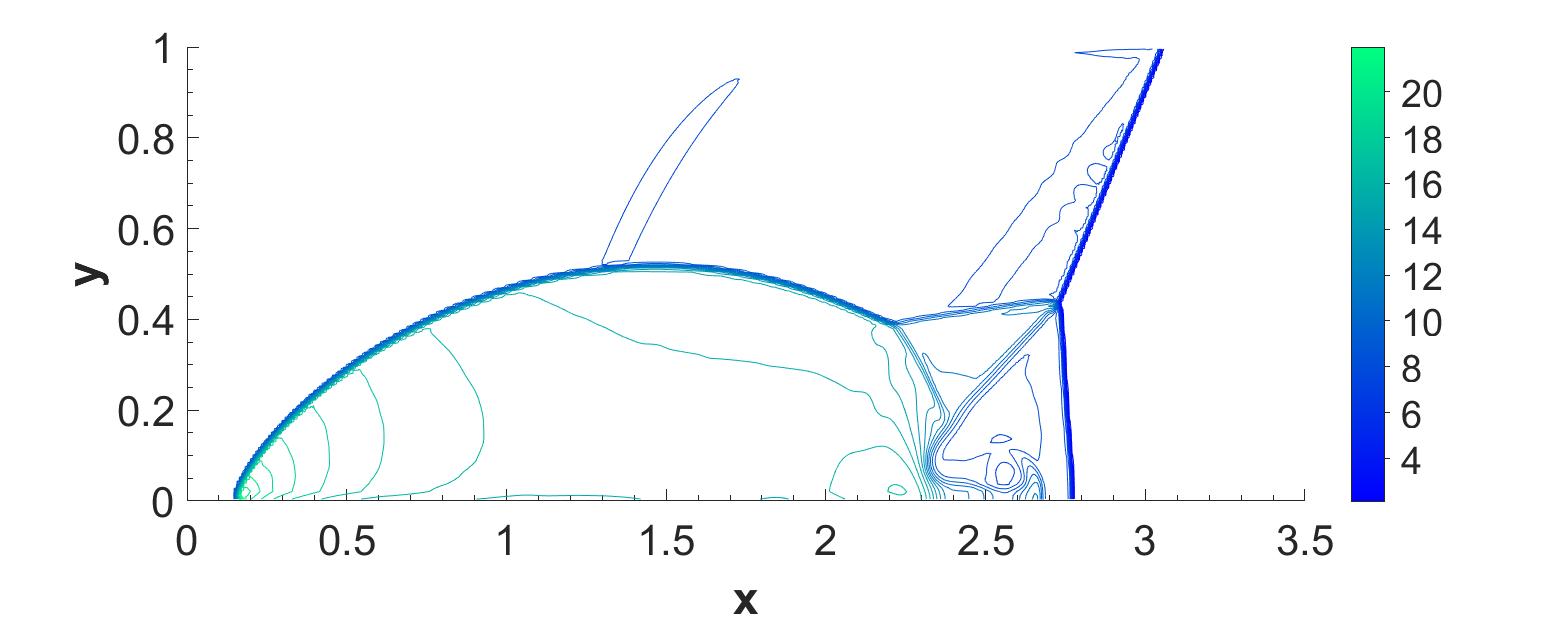}
			\caption{$k=1$, RKDG, $480\times 120$ mesh}
		\end{subfigure}
  \\
		\begin{subfigure}[t]{0.49\textwidth}
			\centering
			\includegraphics[trim=1cm 0cm 3cm 0cm, width=\textwidth]{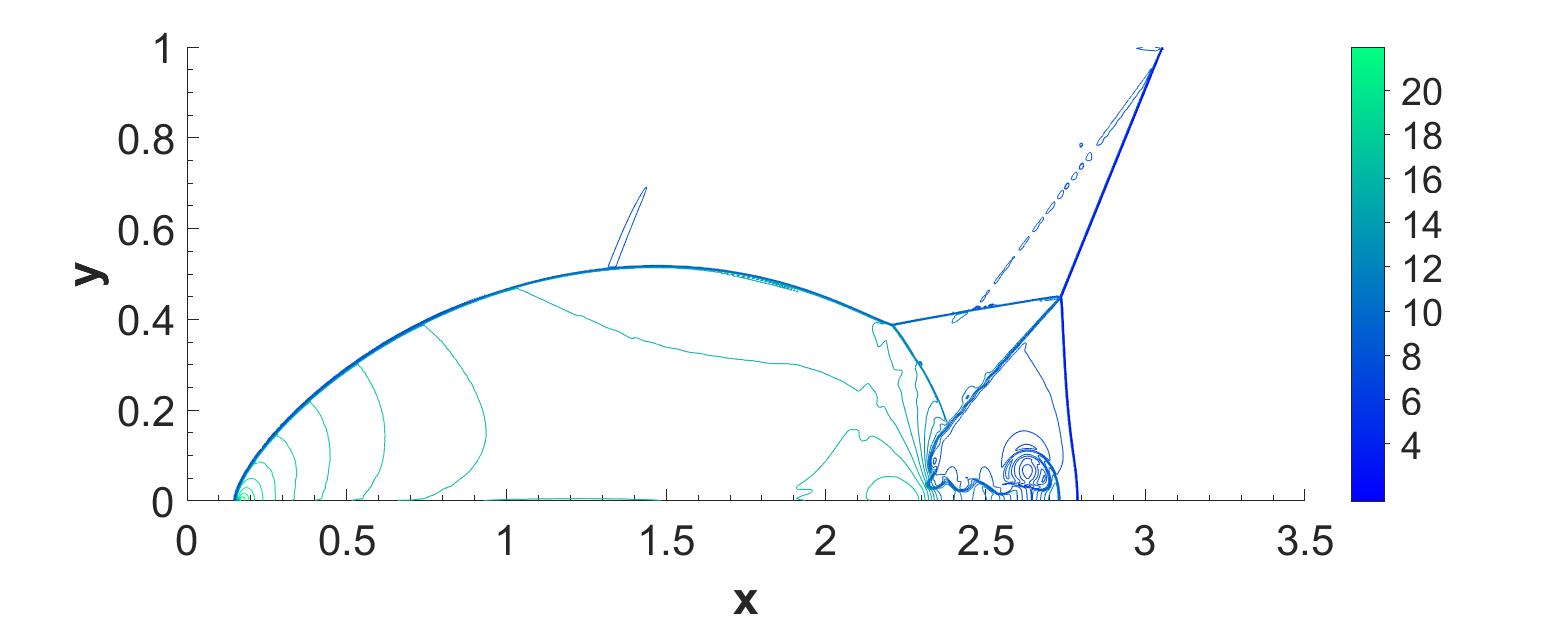}
			\caption{$k=1$, cRKDG, $1920\times 480$ mesh}
		\end{subfigure}		
		\begin{subfigure}[t]{0.49\textwidth}
			\centering		\includegraphics[trim=1cm 0cm 3cm 0cm, width=\textwidth]{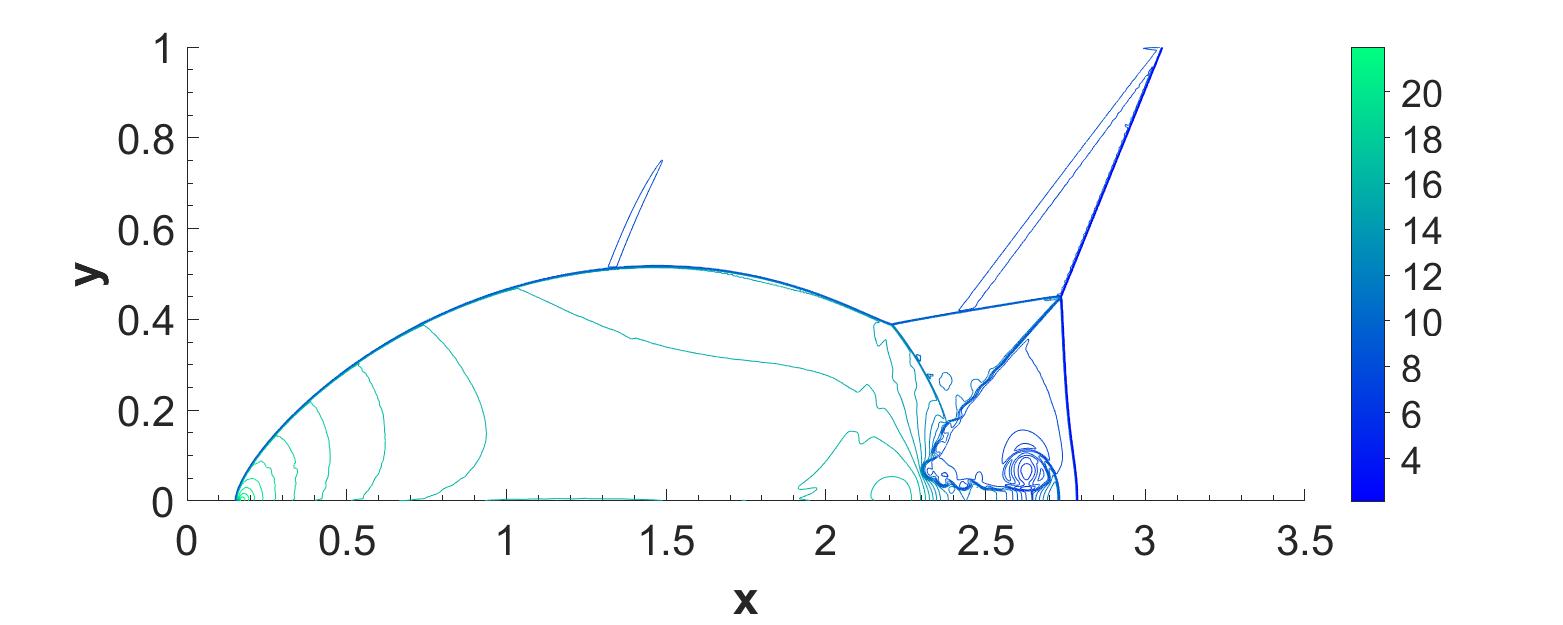}
			\caption{$k=1$, RKDG, $1920\times 480$ mesh}
		\end{subfigure}
  \\
  \begin{subfigure}[t]{0.49\textwidth}
			\centering
			\includegraphics[trim=1cm 0cm 3cm 0cm, width=\textwidth] {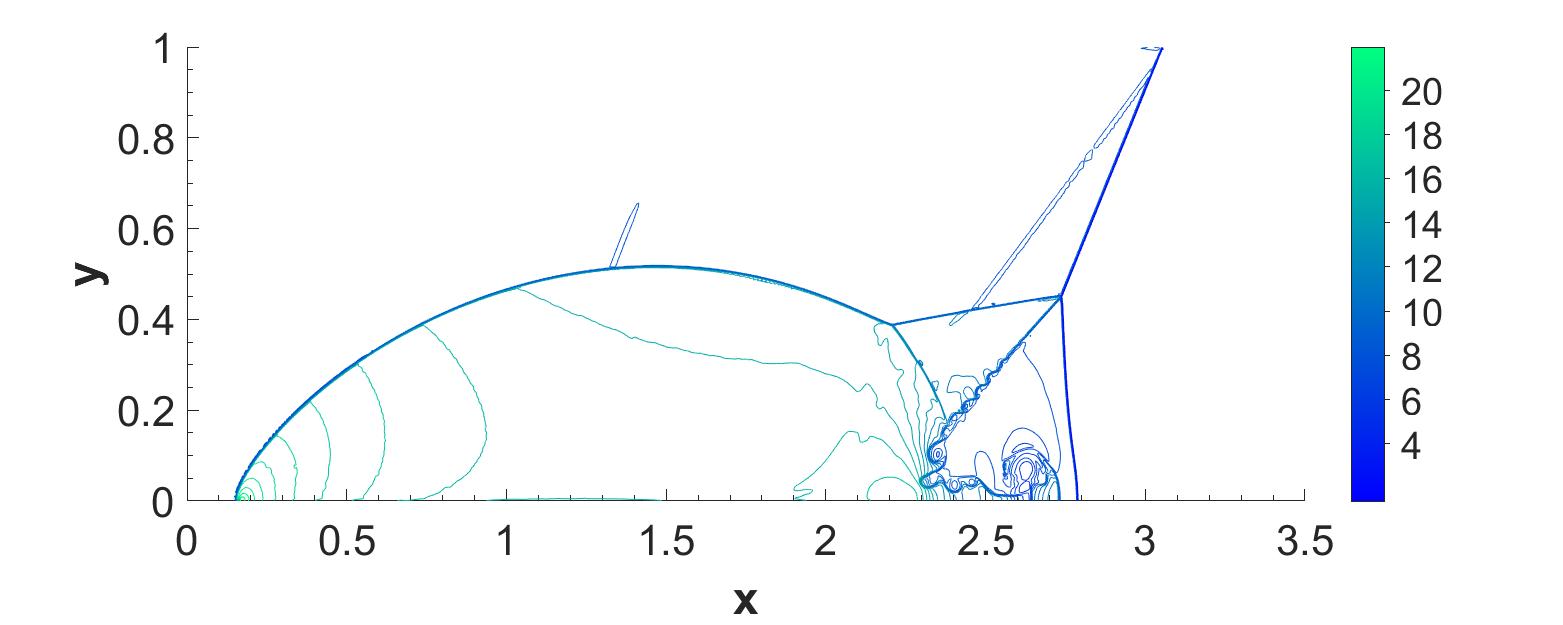}
			\caption{$k=2$, cRKDG, $1920\times 480$ mesh}
		\end{subfigure} 
		\begin{subfigure}[t]{0.49\textwidth}
			\centering
			\includegraphics[trim=1cm 0cm 3cm 0cm, width=\textwidth]{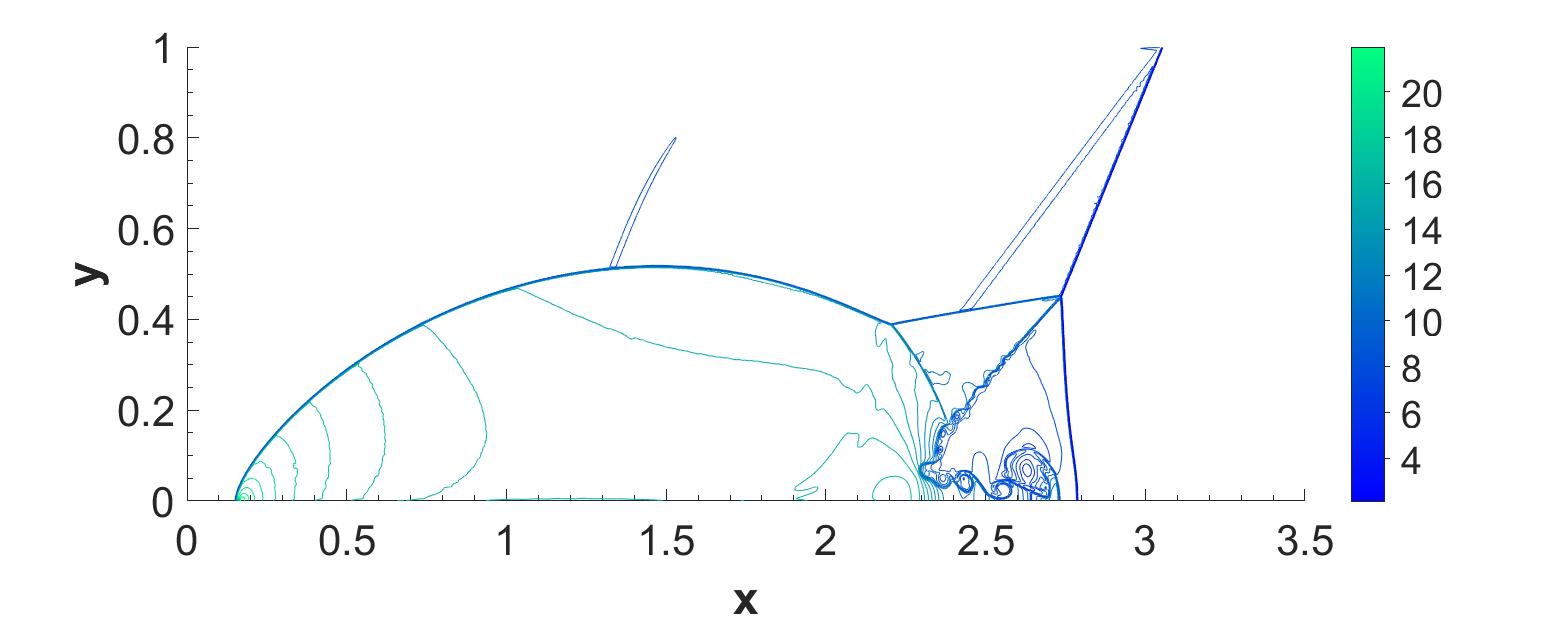}
			\caption{$k=2$, RKDG, $1920\times 480$ mesh}
		\end{subfigure} 
		\caption{Solution profiles for the double Mach problem in \cref{ex:doublemach} at $t=0.2$ with $M=50$. $30$ equally spaced density contours from $1.3695$ to $ 22.682$ are displaced.}
		\label{fig:doublemach-big}
	\end{figure}

	\begin{figure}[h!]
	\centering
 \begin{subfigure}[t]{.4\textwidth}
			\includegraphics[width=\textwidth,height=0.7\linewidth]{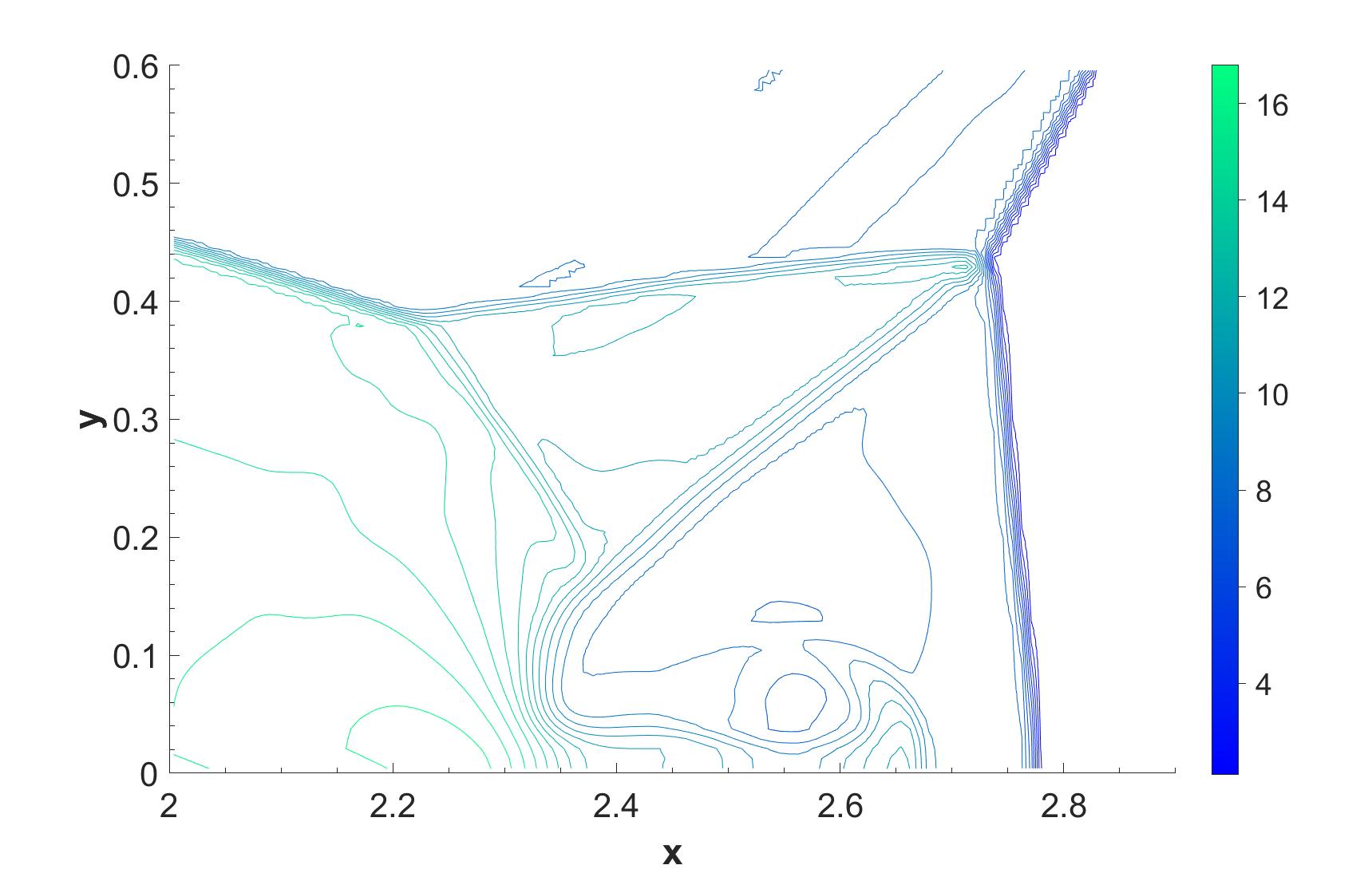}
			\caption{$k=1$, cRKDG, $480\times 120$ mesh}
		\end{subfigure}  
  \begin{subfigure}[t]{.4\textwidth}			\includegraphics[width=\textwidth,height=0.7\linewidth]{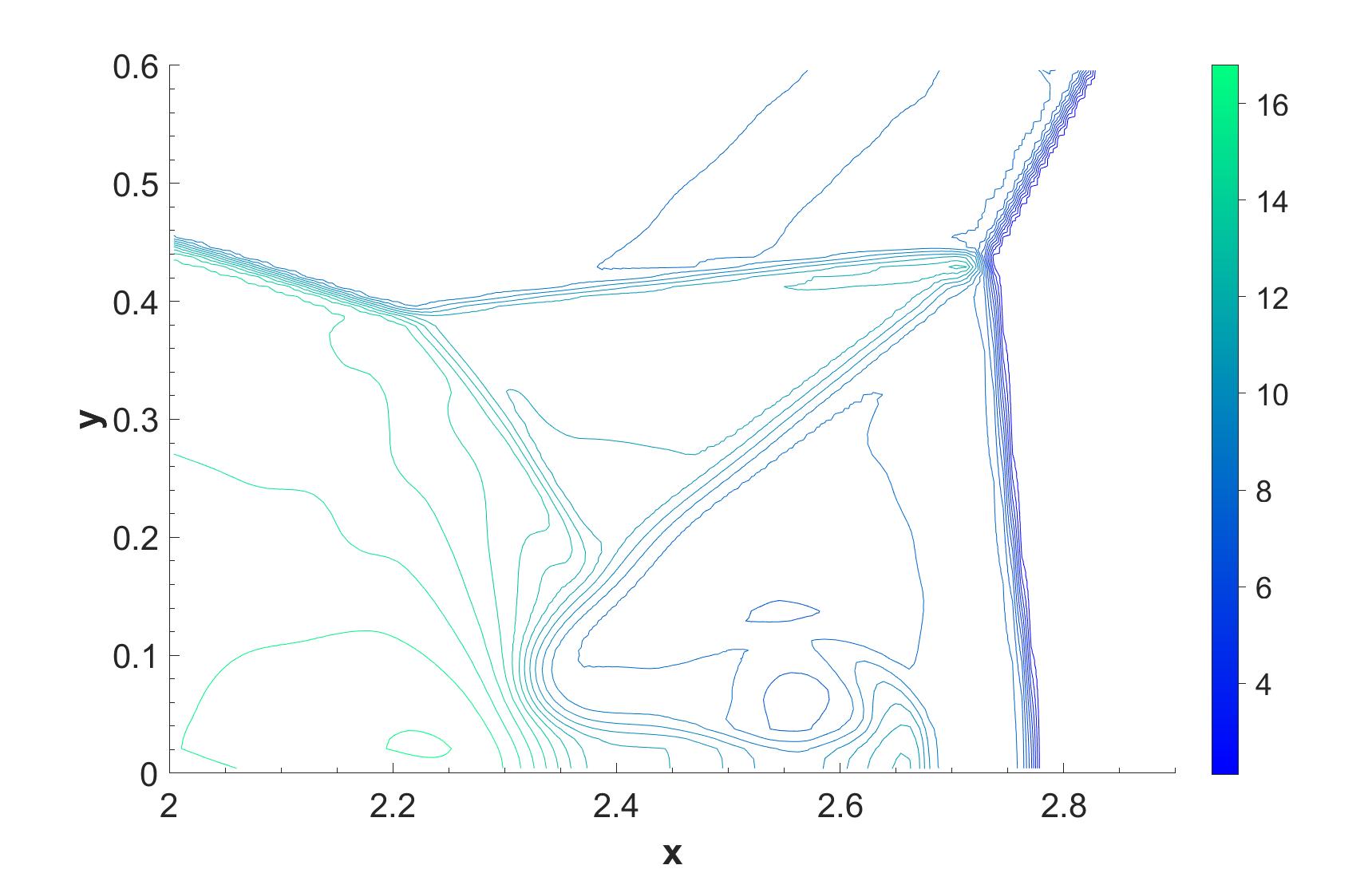}
			\caption{$k=1$, RKDG, $480\times 120$ mesh}
		\end{subfigure}
  \\
		\begin{subfigure}[t]{.4\textwidth}			\includegraphics[width=\textwidth,height=0.7\linewidth]{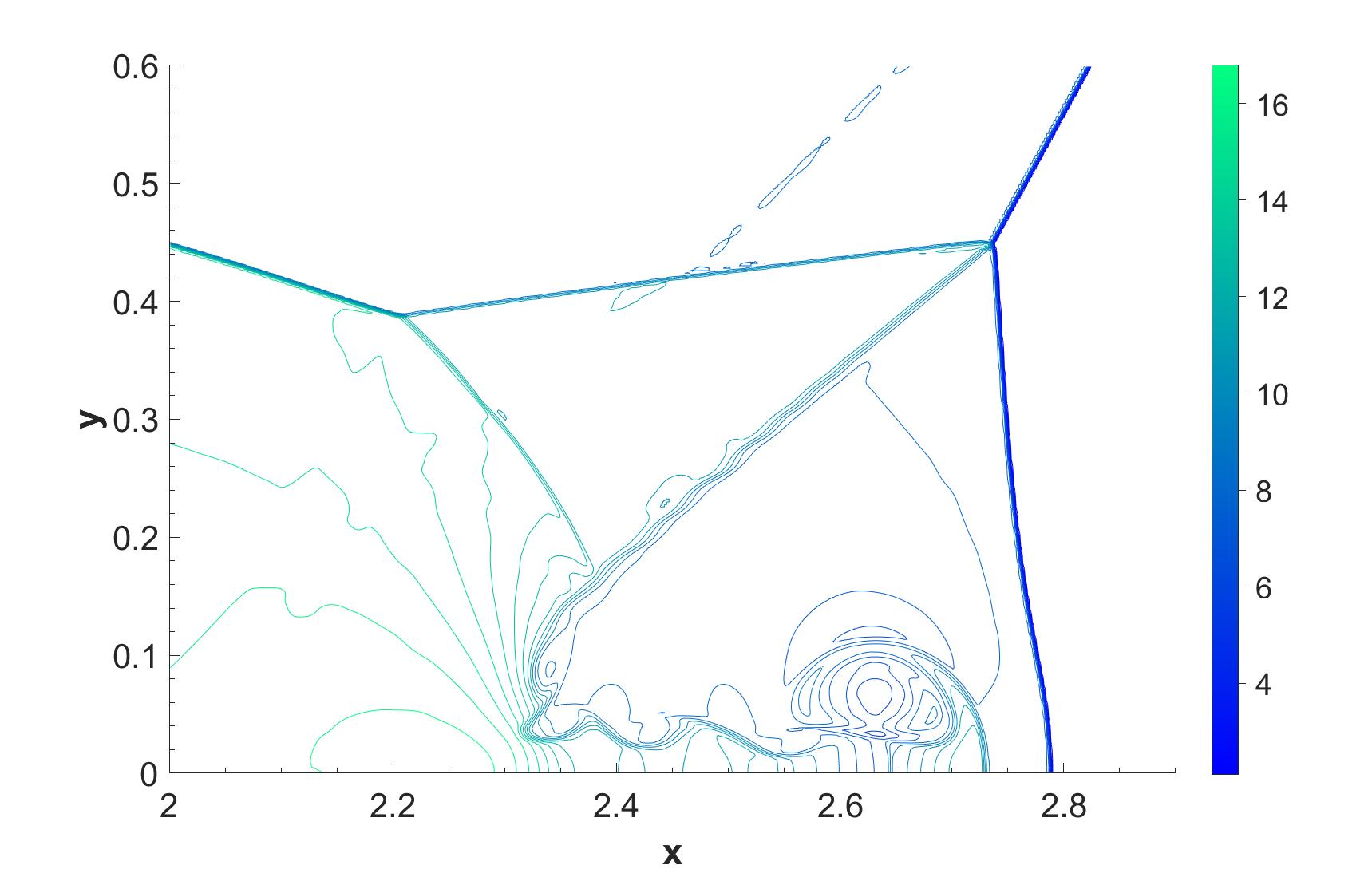}
			\caption{$k=1$, cRKDG, $1920\times 480$ mesh}
		\end{subfigure}
  \begin{subfigure}[t]{.4\textwidth}
			\includegraphics[width=\textwidth,height=0.7\linewidth]{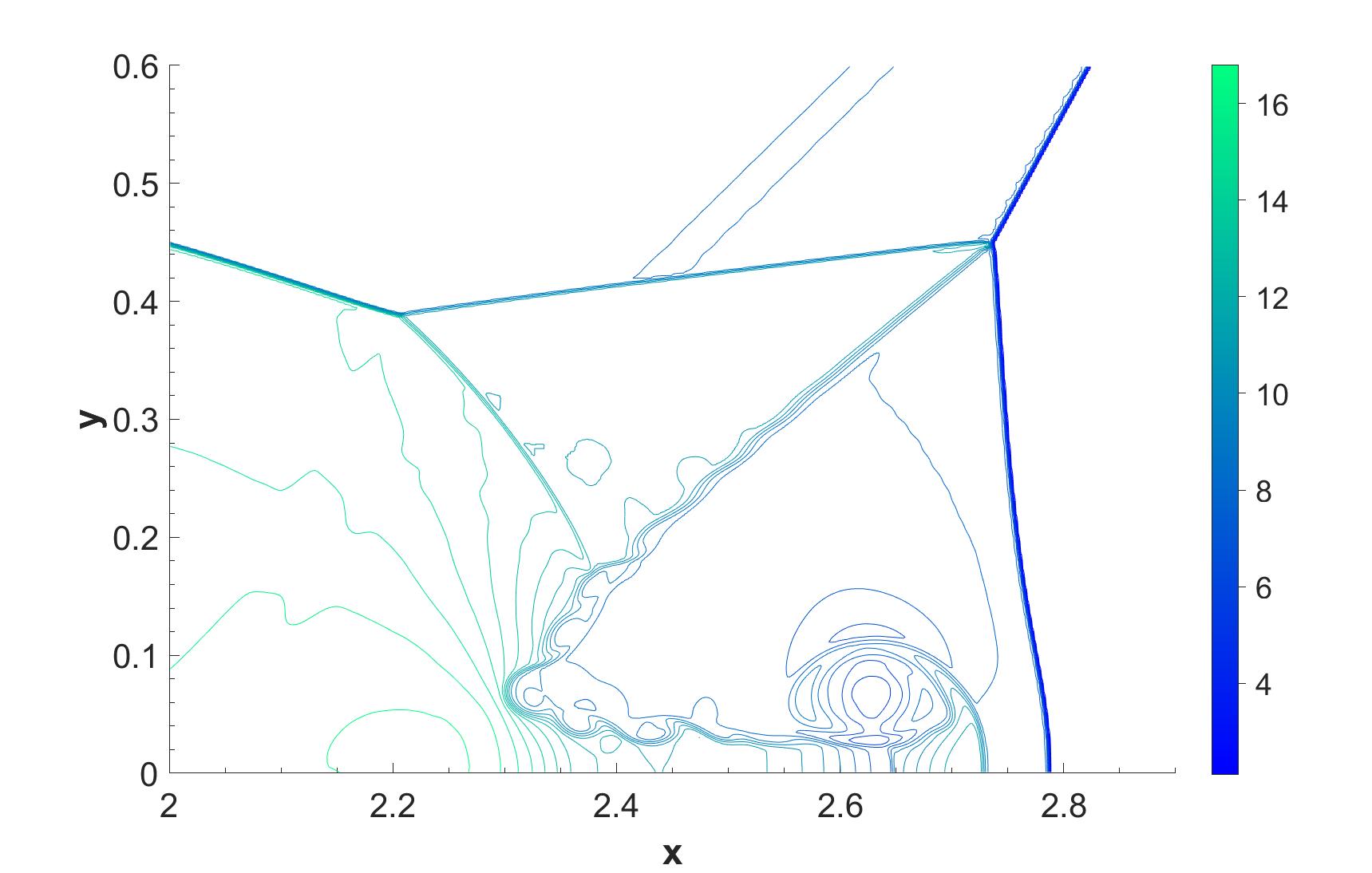}
			\caption{$k=1$, RKDG, $1920\times 480$ mesh}
		\end{subfigure}  
		\begin{subfigure}[t]{.4\textwidth}
			\includegraphics[width=\textwidth,height=0.7\linewidth]{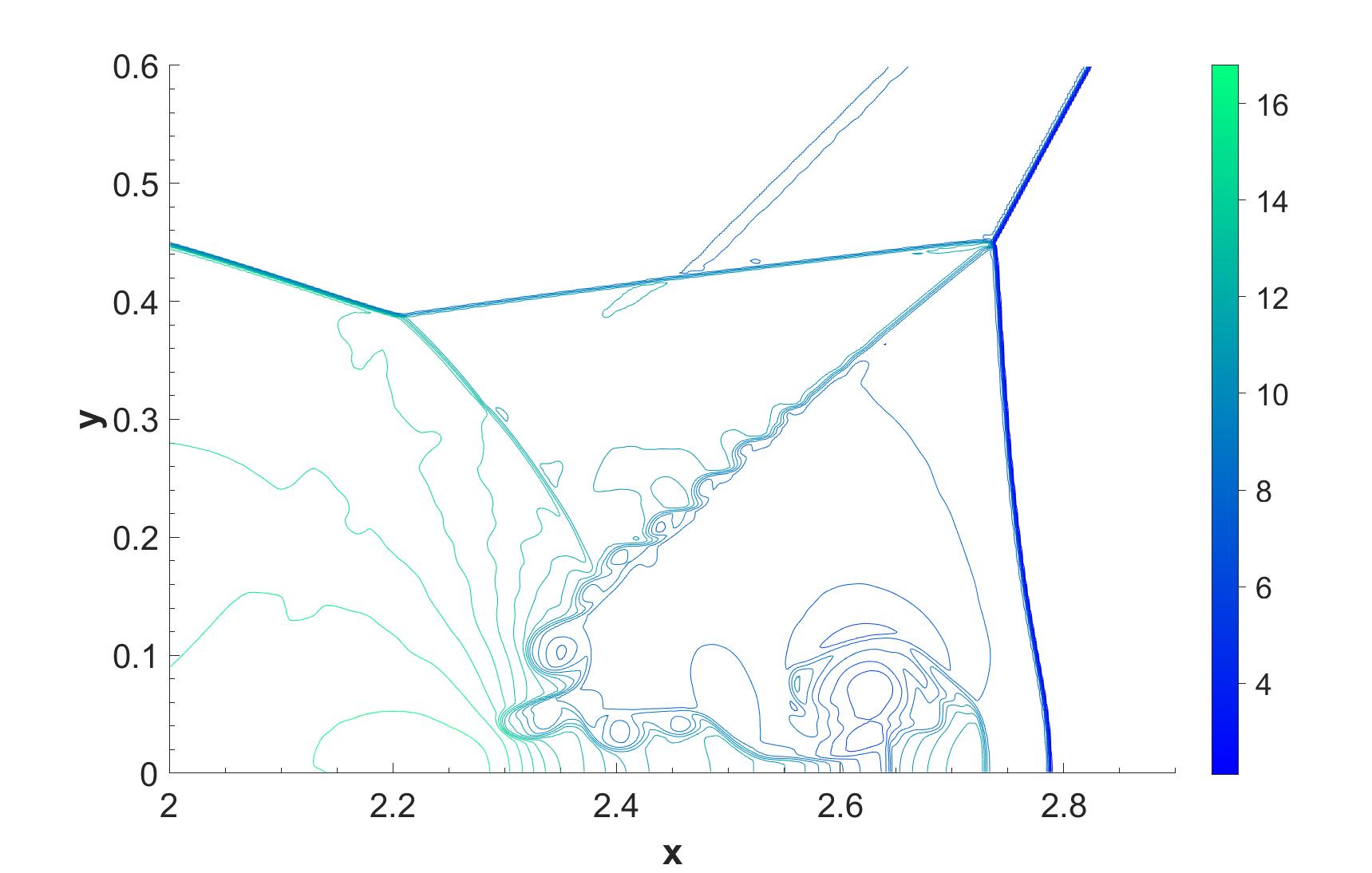}
			\caption{$k=2$, cRKDG, $1920\times 480$ mesh}
		\end{subfigure}
			\begin{subfigure}[t]{.4\textwidth}
			\includegraphics[width=\textwidth,height=0.7\linewidth]{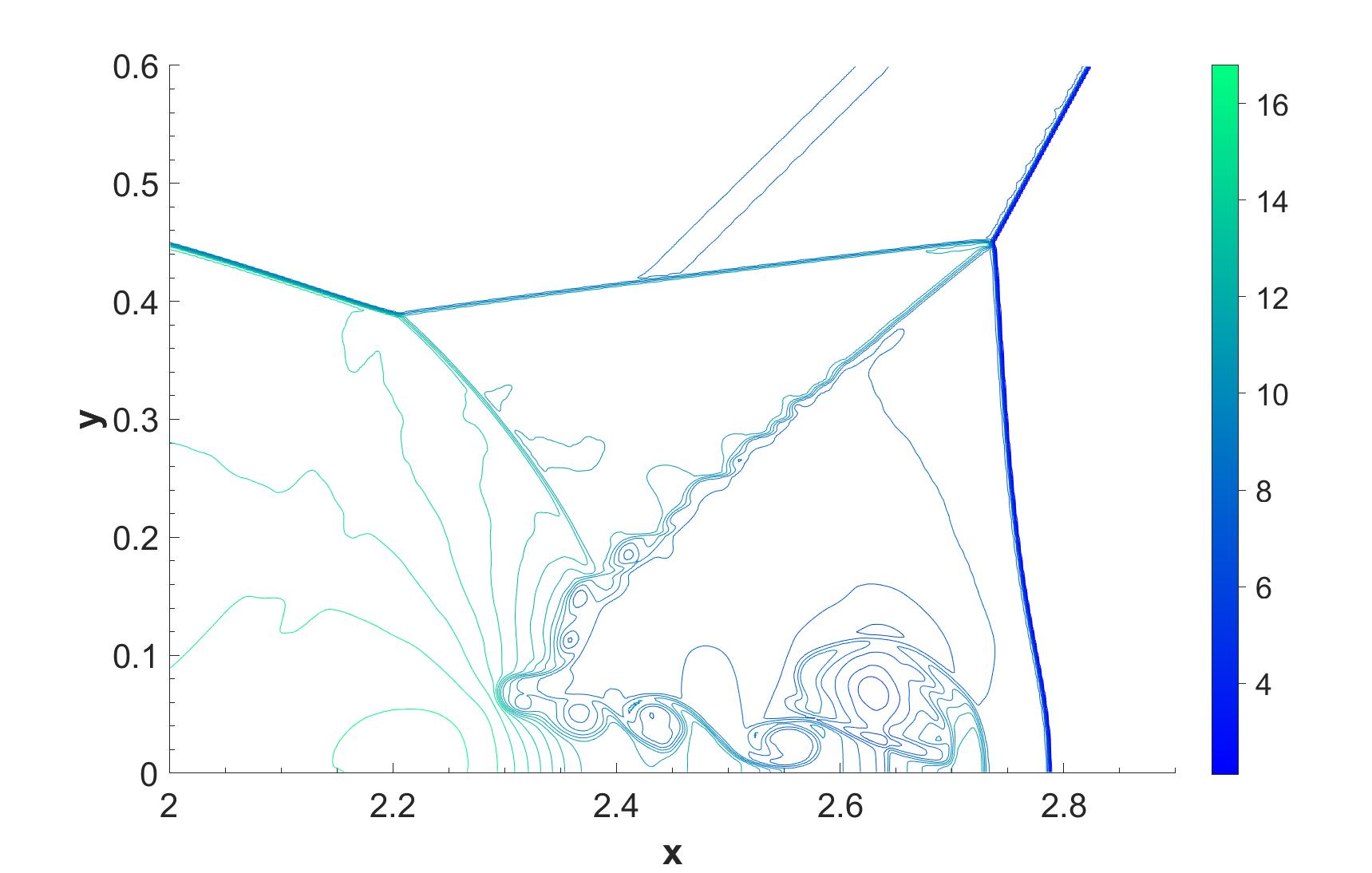}
			\caption{$k=2$, RKDG, $1920\times 480$ mesh}
		\end{subfigure}
		\caption{Zoomed-in solutions for the double Mach problem in \cref{ex:doublemach} at $t=0.2$ with $M=50$. $30$ equally spaced density contours from $1.3695$ to $22.682$ are displaced.}
		\label{fig:doublemach-small}
	\end{figure}
\end{exmp}

\begin{exmp}[Forward facing step]\label{ex:forward-step}
This is another classical test studied in \cite{woodward1984numerical}. In this test, a Mach 3 uniform flow travels to the right and enters a wind tunnel (of 1 length unit wide and 3 length units long), with the step of 0.2 length units high located 0.6 length units from the left-hand end of the tunnel. Reflective boundary conditions are applied along the wall of the tunnel, while inflow/outflow boundary conditions are applied at the entrance/exit. At the corner of the step, a singularity is present. Unlike in \cite{woodward1984numerical}, we do not modify our schemes or refine the mesh near the corner in order to test the performance of our schemes in handling such singularity. We compute the solution up to $t = 4$ and utilize the TVB limiter with a TVB constant $M = 50$. Due to the space limitation, we only present the simulation results with $240\times 80$ mesh cells for $k = 1$ and $960\times320$ mesh cells for $k = 1,2$ in \cref{fig:forward-step}. For this problem, the resolutions of cRKDG and RKDG methods are comparable for the same order of accuracy and mesh.

\begin{figure}[h!]
	\centering
 \begin{subfigure}[t]{.49\textwidth}
			\includegraphics[trim=1cm 0cm 1cm 0cm, width=\textwidth]
                {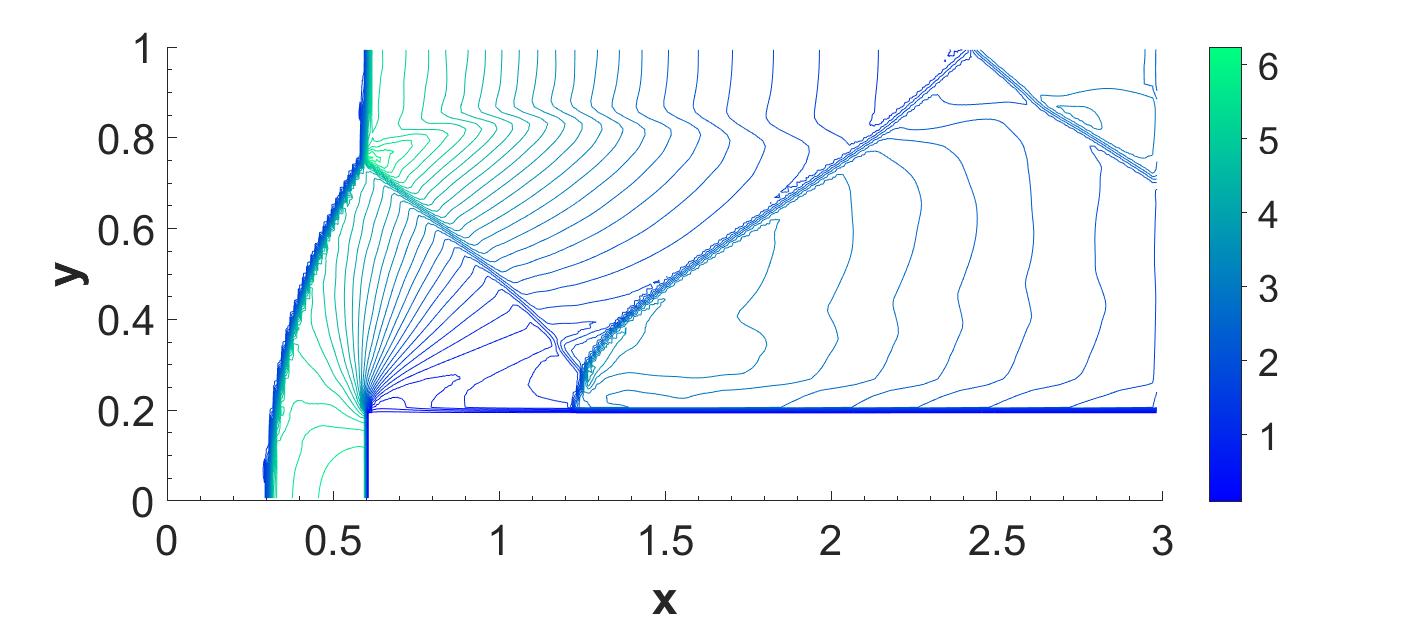}
			\caption{$k=1$, cRKDG, $240\times 80$ mesh}
		\end{subfigure}
		\begin{subfigure}[t]{.49\textwidth}
			\includegraphics[trim=1cm 0cm 1cm 0cm, width=\textwidth]
                {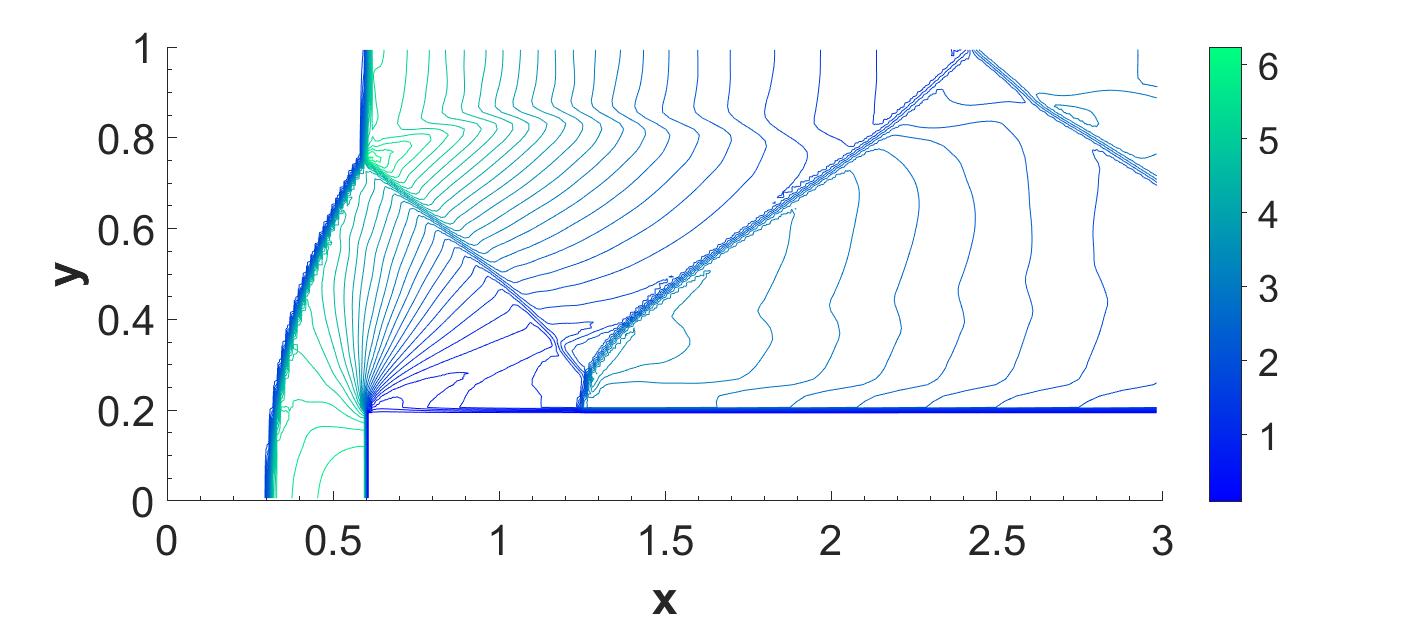}
			\caption{$k=1$, RKDG, $240\times 80$ mesh}
		\end{subfigure}

		\begin{subfigure}[t]{.49\textwidth}
			\includegraphics[trim=1cm 0cm 1cm 0cm, width=\textwidth]
                {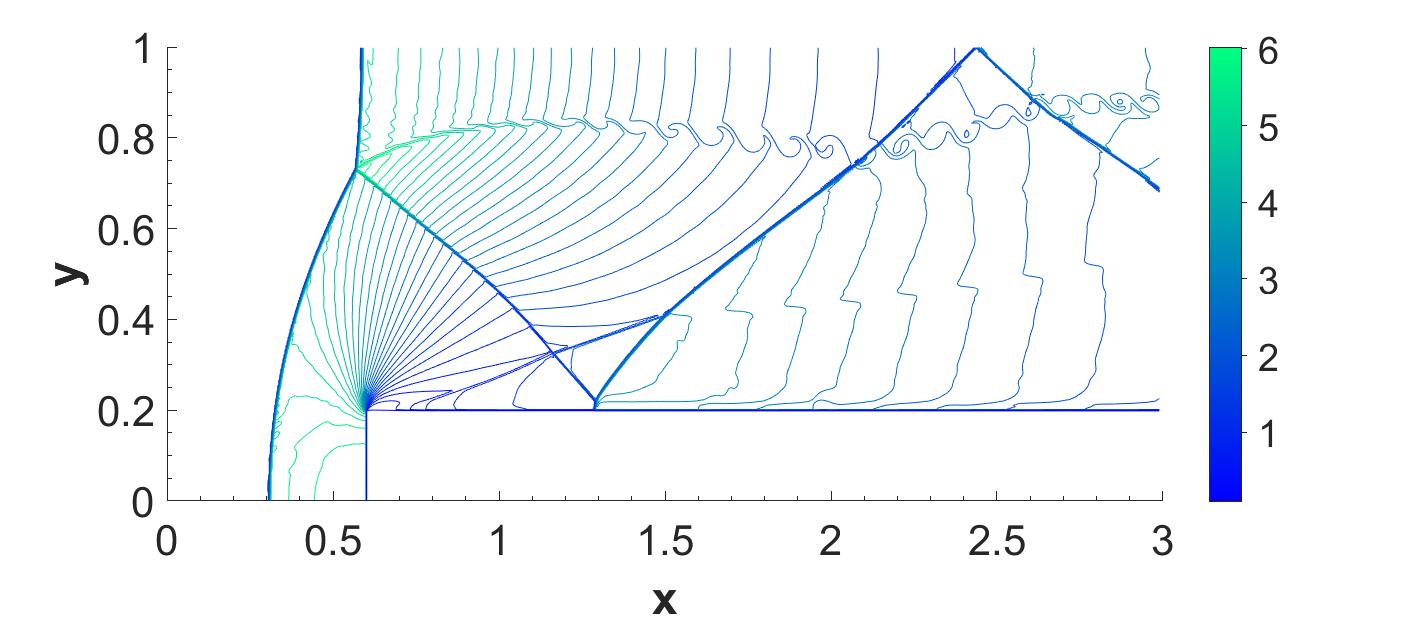}
			\caption{$k=1$, cRKDG, $960\times 320$ mesh}
		\end{subfigure}
		\begin{subfigure}[t]{.49\textwidth}
			\includegraphics[trim=1cm 0cm 1cm 0cm, width=\textwidth]
                {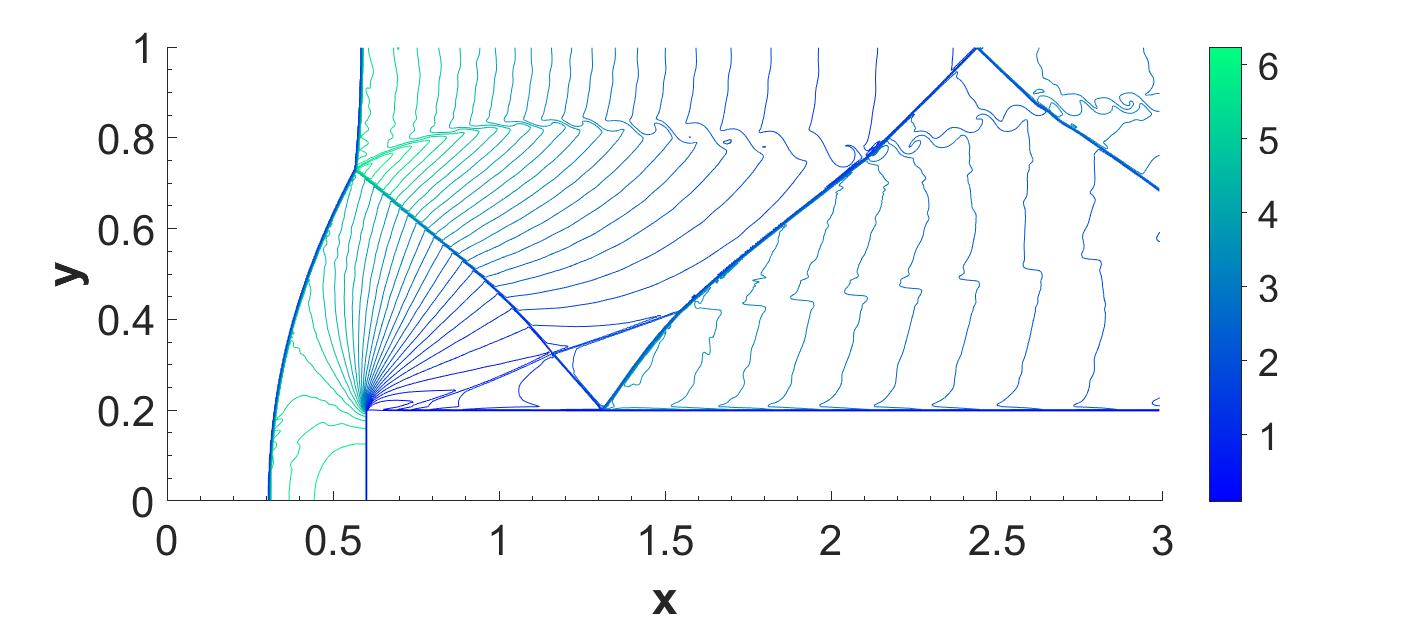}
			\caption{$k=1$, RKDG, $960\times 320$ mesh}
		\end{subfigure}
      \begin{subfigure}[t]{.49\textwidth}
			\includegraphics[trim=1cm 0cm 1cm 0cm, width=\textwidth]
                {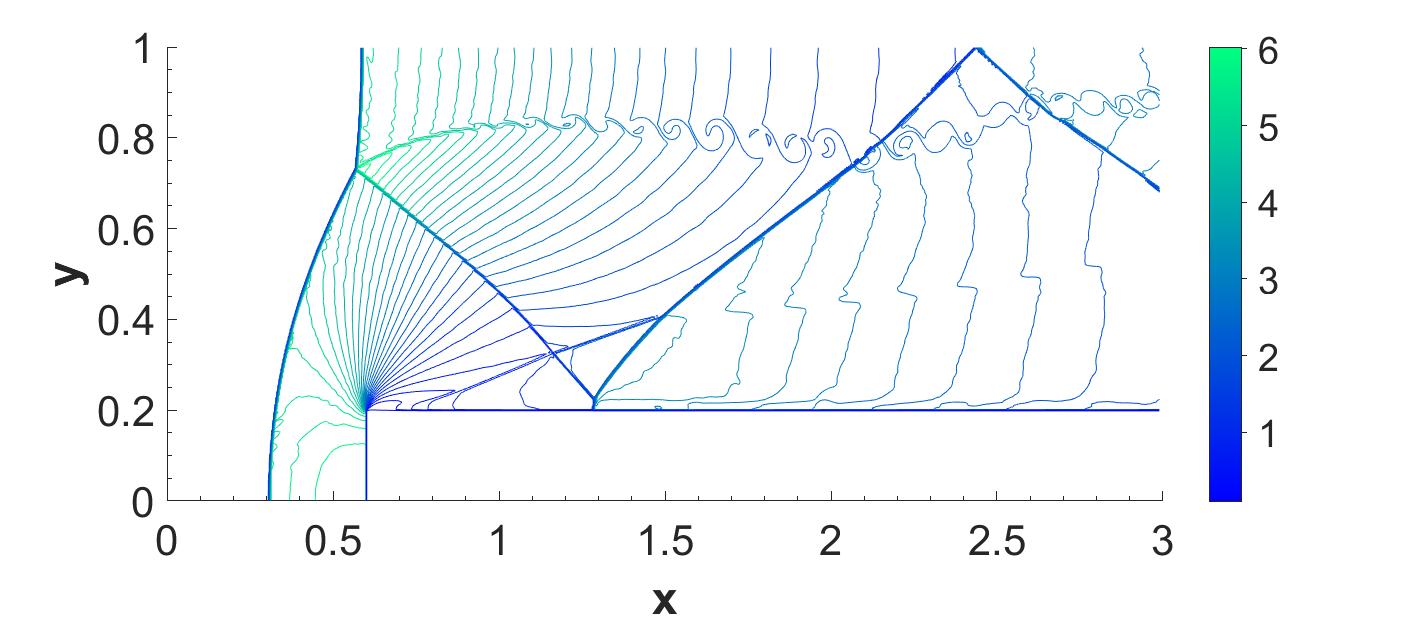}
			\caption{$k=2$, cRKDG, $960\times 320$ mesh}
		\end{subfigure}
		\begin{subfigure}[t]{.49\textwidth}
			\includegraphics[trim=1cm 0cm 1cm 0cm, width=\textwidth]
                {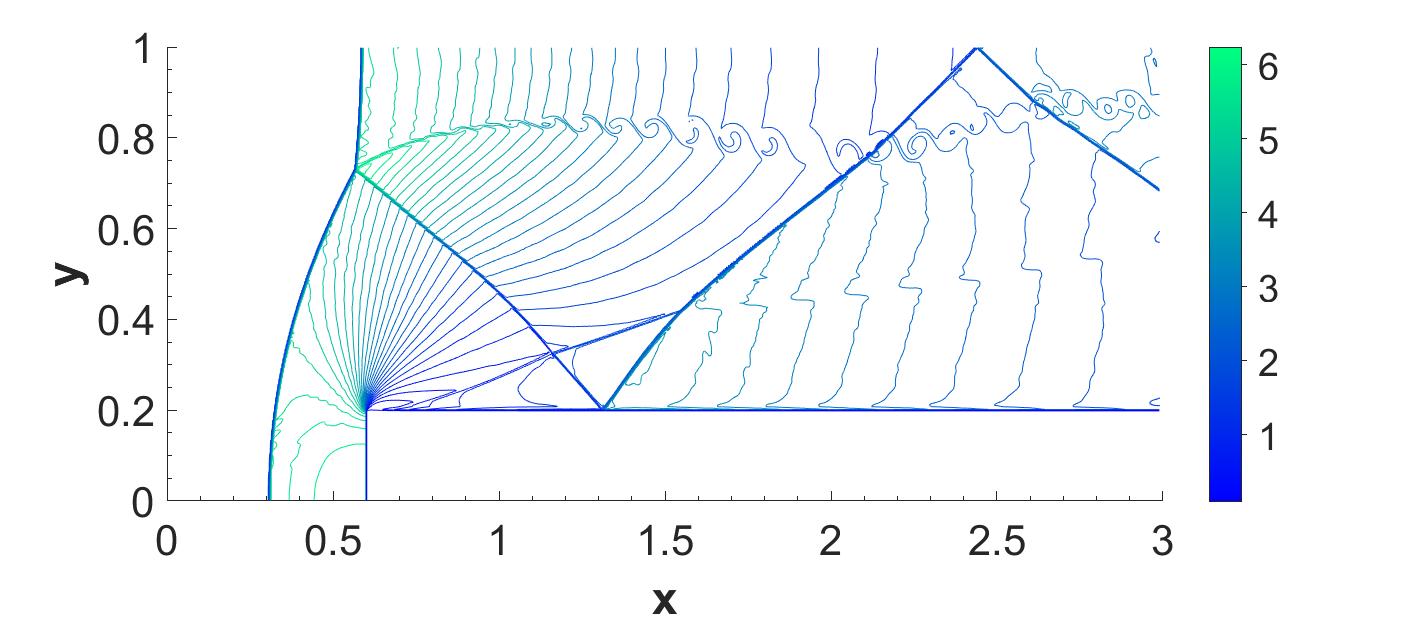}
			\caption{$k=2$, RKDG, $960\times 320$ mesh}
		\end{subfigure}
		\caption{Solution profiles for the forward step problem in \cref{ex:forward-step} at $t=4$ with $M=50$. $30$ equally spaced density contours from $0.090388$ to $6.2365$ are displaced. } 
		\label{fig:forward-step}
	\end{figure}
\end{exmp}

\section{Conclusions and future work}\label{sec:conclusions}
In this paper, we present a novel class of RKDG methods with compact stencils for solving the hyperbolic conservation laws. Our main idea is to replace the DG operator in the inner temporal stages of the fully discrete RKDG scheme by a local derivative operator. We prove a Lax--Wendroff type theorem which guarantees its convergence to the weak solution. Numerically, we observe the new method achieves the optimal convergence rate and does not suffer from the order degeneracy when the Dirichlet type inflow boundary condition is imposed. Moreover, the connections of the new method with the LWDG and ADER-DG methods are established. This is the first paper of a few of our future works, which include the rigorous stability and error analysis with the energy method, extensions to implicit time stepping and convection-dominated problems, and the design of structure-preserving schemes based on the cRKDG methods, etc.

\section*{Acknowledgments}
Z. Sun was partially supported by the NSF grant DMS-2208391.	
Y. Xing was partially supported by the NSF grant DMS-1753581 and DMS-2309590. The authors would also like to thank Professor Jianxian Qiu at Xiamen University for sharing the code for 2D Euler tests and providing helpful comments. 	

\appendix
\section{Proof of \cref{thm:lwthm}}\label{app:lwthm}
Recall that we denote by $u_h^{(1)} = u_h^n$. To prove a Lax--Wendroff convergence theorem, we will use the following result simplified from \cite[Theorems 2.3 and 3.2]{shi2018local}.
\begin{theorem}[Shi and Shu, 2018. \cite{shi2018local}]\label{thm:shishu}
	Let $f$ be Lipschitz continuous and $f'$, $f''$ be uniformly bounded in $L^\infty$. Consider a numerical scheme in $d$-dimensional space that yields \eqref{eq:avg}. 
	For any mesh cell $K$, its edge $e\in \partial K$, and its neighboring cell $K^\ext$, suppose $g_{e,K}$ satisfies the following properties on $B_K= K\cup K^\ext$. 
	\begin{enumerate}
		\item Consistency: if $u_h^n(x)\equiv u$ is a constant, then 
		$g_{e,K}(u_h^n) = |e|  {f(u)\cdot \nu_{e,K}}$.
		\item Boundedness: $|g_{e,K}({u}_{h}^n) - g_{e,K}({v}_{h}^n)|\leq C \|u_h^n - v_h^n\|_{L^\infty(B_K)}\cdot h^{d-1}.$
		\item Anti-symmetry: 
		$g_{e,K}(u_h^n) + g_{e,K^\ext}(u_h^n) = 0$, for $e = K \cap K^\ext$.
	\end{enumerate}
	If $u_h^n$ converges boundedly almost everywhere to a function $u$ as $\Delta t, h \to 0$, then $u$ is a weak solution to the conservation law saftisfying \eqref{eq:weak}.
\end{theorem}
Hence to analyze the convergence of the cRKDG method, it suffices to verify that $g_{e,K}$ defined in \eqref{eq:g} does satisfy three properties in \cref{thm:shishu}. 
\begin{lem}\label{lem:g}
	Under assumptions in \cref{thm:lwthm}, the combined flux $g_{e,K}$ defined in \eqref{eq:g} satisfies the three properties in \cref{thm:shishu}.
\end{lem}
Once \cref{lem:g} is proved, \cref{thm:lwthm} follows as a direct consequence of \cref{thm:shishu}. The rest of the section is dedicated to the proof of \cref{lem:g}, especially on the boundedness of $g_{e,K}$.

\begin{lem}\label{lem:L2projboundedness} Let $\rho$ be a $L^2$ and $L^\infty$ function. Then
$\|\Pi \rho \|_{L^\infty(K)} \leq C\| \rho\|_{L^\infty(K)}$.
\end{lem}
\begin{proof}
This lemma can be proved by selecting an orthonormal basis of $\mathcal{P}^k(K)$ and expand $\Pi \rho$ under this basis. Details are omitted. 
\end{proof}

\begin{lem}\label{lem:estpinablaf}
	For any $u_h,v_h \in V_h$ with $\|u_h\|_{L^\infty}, \|v_h\|_{L^\infty}\leq C$, we have 
	\begin{equation}\label{lem:estpinablaf1} 
		\|\Pi\nabla\cdot \left(f(u_h) - f(v_h)\right)\|_{L^\infty(B_K)}\leq \frac{C}{h}\|u_h-v_h\|_{L^\infty(B_K)}.
	\end{equation}
\end{lem}
\begin{proof}
Applying \cref{lem:L2projboundedness} and after some algebraic manipulations, we have \begin{equation}\label{eq:estpigradf}
\begin{aligned}
		&\|\Pi\nabla\cdot \left(f(u_h) - f(v_h)\right)\|_{L^\infty(B_K)}\\
		\leq&  C\|\nabla\cdot \left(f(u_h) - f(v_h)\right)\|_{L^\infty(B_K)}\\
		  = & C\|f'(u_h)\cdot\nabla u_h - f'(v_h)\cdot\nabla v_h\|_{L^\infty(B_K)}\\
		=&  C\|f'(u_h)\cdot \nabla u_h - f'(v_h)\cdot \nabla u_h + f'(v_h)\cdot \nabla u_h - f'(v_h)\cdot \nabla v_h\|_{L^\infty(B_K)}\\
			\leq & C\|\left(f'(u_h)- f'(v_h)\right)\cdot\nabla u_h\|_{L^\infty(B_K)} + C\|f'(v_h)\cdot \nabla \left(u_h - v_h\right)\|_{L^\infty(B_K)}\\
			\leq& C\|f''\|_{L^\infty}\|u_h - v_h\|_{L^\infty(B_K)}\|\nabla u_h\|_{L^\infty(B_K)}+C\|f'\|_{L^\infty}\|\nabla(u_h-v_h)\|_{L^\infty(B_K)}.
		\end{aligned}	
	\end{equation}
	With the inverse estimate, we have $
		\|\nabla u_h\|_{L^\infty(B_K)}\leq {C}h^{-1}\|u_h\|_{L^\infty(B_K)}$ and $\|\nabla(u_h-v_h)\|_{L^\infty(B_K)}\leq {C}{h^{-1}}\|u_h-v_h\|_{L^\infty(B_K)}$.
	Recall that we assumed $f'$, $f''$, and $u_h$ are bounded in $L^\infty$. \eqref{lem:estpinablaf1} can 
 be obtained after substituting these estimates into \eqref{eq:estpigradf}.
\end{proof}

\begin{lem}\label{lem:eststage}
		$\|u_h^{(i)}-v_h^{(i)}\|_{L^\infty(B_K)}\leq C\|u_h^n-v_h^n\|_{L^\infty(B_K)}$, for all $1\leq i \leq s$. 
\end{lem}
\begin{proof}
	We prove the lemma by induction. For $i = 1$, by definition we have $u_h^{(1)} = u_h^n$ and $v_h^{(1)} = v_h^n$. The inequality is true with $C = 1$ for all $u_h^n,v_h^n \in V_h$. 
	
	Now we assume that the inequality is true with $i\leq m$.
	
	First, we want to show that the induction hypothesis implies
	\begin{equation*}
\|u_h^{(i)}\|_{L^\infty(B_K)}\leq C\|u_h^n\|_{L^\infty(B_K)} \quad \forall 1\leq i\leq m.
	\end{equation*}
	Indeed, note that when $v_h^n \equiv 0$ on $B_K$,  $f(v_h^n)\equiv 0$ is a constant. Hence $v_h^{(2)} = 0 - a_{21}\Delta t \dxl 0 = 0$. Similarly, we have $v_h^{(i)}\equiv 0$ for all $1\leq i\leq m$. By the induction hypothesis, we can see that for all $1\leq i\leq m$, 
	\begin{equation*}
		\|u_h^{(i)}\|_{L^\infty(B_K)} = \|u_h^{(i)}-v_h^{(i)}\|_{L^\infty(B_K)}\leq C\|u_h^n-v_h^{n}\|_{L^\infty(B_K)} =  C\|u_h^n\|_{L^\infty(B_K)}.
	\end{equation*}
	
	Then we prove the lemma is true for $i = m+1$. It can be seen that 
	\begin{equation*}
		\begin{aligned}
			&\|u_h^{(m+1)}-v_h^{(m+1)}\|_{L^\infty(B_K)} \\
			=& \left\|\left(u_h^n - {\Delta t}\sum_{j = 1}^{m} a_{ij}\Pi \nabla\cdot f\left(u_h^{(j)}\right)\right)-\left(v_h^n - {\Delta t}\sum_{j = 1}^{m} a_{ij}\Pi \nabla\cdot f\left(v_h^{(j)}\right)\right)\right\|_{L^\infty(B_K)} \\
			\leq & \|u_h^n-v_h^n\|_{L^\infty(B_K)}+{\Delta t}\sum_{j = 1}^{m} |a_{ij}|\left\|\Pi \nabla\cdot \left(f\left(u_h^{(j)}\right)-f\left(v_h^{(j)}\right)\right)\right\|_{L^\infty(B_K)}.
		\end{aligned}
	\end{equation*}
	According to the first part of the proof, $\nm{u_h^{(j)}}_{L^\infty}\leq C\nm{u_h^n}_{L^\infty}\leq C$ and $\nm{v_h^{(j)}}_{L^\infty}\leq C\nm{v_h^n}_{L^\infty}\leq C$ are bounded. Hence with \cref{lem:estpinablaf}, it yields
	\begin{equation*}
	    \|u_h^{(m+1)}-v_h^{(m+1)}\|_{L^\infty(B_K)} \leq \|u_h^n-v_h^n\|_{L^\infty(B_K)}+\frac{C\Delta t}{h}\sum_{j = 1}^{m} |a_{ij}|\left\|u_h^{(j)}-v_h^{(j)}\right\|_{L^\infty(B_K)}.
	\end{equation*}
    One can then prove the lemma after using the CFL condition $\Delta t/h \leq C$ and the induction hypothesis $	\|u_h^{(j)}-v_h^{(j)}\|_{L^\infty(B_K)}\leq C\|u_h^n-v_h^n\|_{L^\infty(B_K)}$ for all $1\leq j \leq m$. 
\end{proof}

\begin{proof}[Proof of \cref{lem:g}]
	The consistency of $g_{e,K}$ can be obtained from the consistency of $\widehat{f}$ and the consistency of the RK method $\sum_{i=1}^{s}b_i = 1$. The anti-symmetry of $g_{e,K}$ can be obtained from the anti-symmetry of $\widehat{f}$. The key is to show the boundedness of $g_{e,K}$ as follows.
	\begin{equation*}
		\begin{aligned}
			|g_{e,K}(u_h) - g_{e,K}(v_h)|
			=&\left|\int_e \left(\sum_{i=1}^s b_i\widehat{f\cdot \nu_{e,K}}\left(u_h^{(i)}\right)\right) \dd l - \int_e \left(\sum_{i=1}^s b_i\widehat{f\cdot \nu_{e,K}}\left(v_h^{(i)}\right)\right) \dd l\right|\\
			\leq& \sum_{i=1}^{s-1}\int_e  \left|b_i\right|\left|\widehat{f\cdot \nu_{e,K} }\left(u_h^{(i)}\right)-\widehat{f\cdot \nu_{e,K} }\left(v_h^{(i)}\right)\right| \dd l\\
			\leq& C\sum_{i=1}^{s-1}|b_i||e| \nm{u_h^{(i)}-v_h^{(i)}}_{L^\infty} 
			\leq C\nm{u_h^{n}-v_h^{n}}_{L^\infty}\cdot h^{d-1}.
		\end{aligned}
	\end{equation*}
	Here we have used the Lipschitz continuity of $\widehat{f\cdot \nu_{e,K}}$ in the second last inequality and \cref{lem:eststage} in the last inequality. 
\end{proof}

\bibliography{references_abbr}  

\begin{thebibliography}{10}

\bibitem{bassi2005discontinuous}
{\sc F.~Bassi, A.~Crivellini, S.~Rebay, and M.~Savini}, {\em Discontinuous
  {G}alerkin solution of the reynolds-averaged {N}avier--{S}tokes and
  k--$\omega$ turbulence model equations}, Comput. \& Fluids, 34 (2005),
  pp.~507--540.

\bibitem{bassi1997high2}
{\sc F.~Bassi, S.~Rebay, G.~Mariotti, S.~Pedinotti, and M.~Savini}, {\em A
  high-order accurate discontinuous finite element method for inviscid and
  viscous turbomachinery flows}, in Proceedings of the 2nd European Conference
  on Turbomachinery Fluid Dynamics and Thermodynamics, Antwerpen, Belgium,
  1997, pp.~99--109.

\bibitem{BCKX2013}
{\sc J.~L. Bona, H.~Chen, O.~A. Karakashian, and Y.~Xing}, {\em Conservative,
  discontinuous {G}alerkin-methods for the {G}eneralized {K}orteweg-de {V}ries
  equation}, Math. Comp., 82 (2013), pp.~1401--1432.

\bibitem{carpenter1995theoretical}
{\sc M.~H. Carpenter, D.~Gottlieb, S.~Abarbanel, and W.-S. Don}, {\em The
  theoretical accuracy of {R}unge--{K}utta time discretizations for the initial
  boundary value problem: a study of the boundary error}, SIAM J. Sci. Comput.,
  16 (1995), pp.~1241--1252.

\bibitem{cheng2008discontinuous}
{\sc Y.~Cheng and C.-W. Shu}, {\em A discontinuous {G}alerkin finite element
  method for time dependent partial differential equations with higher order
  derivatives}, Math. Comp., 77 (2008), pp.~699--730.

\bibitem{rkdg4}
{\sc B.~Cockburn, S.~Hou, and C.-W. Shu}, {\em The {R}unge--{K}utta local
  projection discontinuous {G}alerkin finite element method for conservation
  laws. {I}{V}. the multidimensional case}, Math. Comp., 54 (1990),
  pp.~545--581.

\bibitem{rkdg3}
{\sc B.~Cockburn, S.-Y. Lin, and C.-W. Shu}, {\em {T}{V}{B} {R}unge--{K}utta
  local projection discontinuous {G}alerkin finite element method for
  conservation laws {I}{I}{I}: one-dimensional systems}, J. Comput. Phys., 84
  (1989), pp.~90--113.

\bibitem{rkdg2}
{\sc B.~Cockburn and C.-W. Shu}, {\em {T}{V}{B} {R}unge--{K}utta local
  projection discontinuous {G}alerkin finite element method for conservation
  laws. {I}{I}. general framework}, Math. Comp., 52 (1989), pp.~411--435.

\bibitem{rkdg1}
{\sc B.~Cockburn and C.-W. Shu}, {\em The {R}unge--{K}utta local projection $
  {P}^1$-discontinuous-{G}alerkin finite element method for scalar conservation
  laws}, ESAIM Math. Model. Numer. Anal., 25 (1991), pp.~337--361.

\bibitem{rkdg5}
{\sc B.~Cockburn and C.-W. Shu}, {\em The {R}unge--{K}utta discontinuous
  {G}alerkin method for conservation laws {V}: multidimensional systems}, J.
  Comput. Phys., 141 (1998), pp.~199--224.

\bibitem{cockburn2001runge}
{\sc B.~Cockburn and C.-W. Shu}, {\em {R}unge--{K}utta discontinuous {G}alerkin
  methods for convection-dominated problems}, J. Sci. Comput., 16 (2001),
  pp.~173--261.

\bibitem{dumbser2008unified}
{\sc M.~Dumbser, D.~S. Balsara, E.~F. Toro, and C.-D. Munz}, {\em A unified
  framework for the construction of one-step finite volume and discontinuous
  {G}alerkin schemes on unstructured meshes}, J. Comput. Phys., 227 (2008),
  pp.~8209--8253.

\bibitem{dumbser2008finite}
{\sc M.~Dumbser, C.~Enaux, and E.~F. Toro}, {\em Finite volume schemes of very
  high order of accuracy for stiff hyperbolic balance laws}, J. Comput. Phys.,
  227 (2008), pp.~3971--4001.

\bibitem{dumbser2006building}
{\sc M.~Dumbser and C.-D. Munz}, {\em Building blocks for arbitrary high order
  discontinuous {G}alerkin schemes}, J. Sci. Comput., 27 (2006), pp.~215--230.

\bibitem{gaburro2023high}
{\sc E.~Gaburro, P.~{\"O}ffner, M.~Ricchiuto, and D.~Torlo}, {\em High order
  entropy preserving {A}{D}{E}{R}-{D}{G} schemes}, Appl. Math. Comput., 440
  (2023), p.~127644.

\bibitem{gottlieb2011strong}
{\sc S.~Gottlieb, D.~I. Ketcheson, and C.-W. Shu}, {\em Strong stability
  preserving {R}unge-{K}utta and multistep time discretizations}, World
  Scientific, 2011.

\bibitem{gottlieb2001strong}
{\sc S.~Gottlieb, C.-W. Shu, and E.~Tadmor}, {\em {Strong stability-preserving
  high-order time discretization methods}}, SIAM Rev., 43 (2001), pp.~89--112.

\bibitem{grant2022perturbed}
{\sc Z.~J. Grant}, {\em Perturbed {R}unge--{K}utta methods for mixed precision
  applications}, J. Sci. Comput., 92 (2022), p.~6.

\bibitem{guo2015new}
{\sc W.~Guo, J.-M. Qiu, and J.~Qiu}, {\em {A new Lax--Wendroff discontinuous
  {G}alerkin method with superconvergence}}, J. Sci. Comput., 65 (2015),
  pp.~299--326.

\bibitem{luo2010reconstructed}
{\sc H.~Luo, L.~Luo, R.~Nourgaliev, V.~A. Mousseau, and N.~Dinh}, {\em A
  reconstructed discontinuous {G}alerkin method for the compressible
  {N}avier--{S}tokes equations on arbitrary grids}, J. Comput. Phys., 229
  (2010), pp.~6961--6978.

\bibitem{peraire2008compact}
{\sc J.~Peraire and P.-O. Persson}, {\em The compact discontinuous {G}alerkin
  ({C}{D}{G}) method for elliptic problems}, SIAM J. Sci. Comput., 30 (2008),
  pp.~1806--1824.

\bibitem{qiu2007numerical}
{\sc J.~Qiu}, {\em A numerical comparison of the {L}ax--{W}endroff
  discontinuous {G}alerkin method based on different numerical fluxes}, J. Sci.
  Comput., 30 (2007), pp.~345--367.

\bibitem{qiu2005discontinuous}
{\sc J.~Qiu, M.~Dumbser, and C.-W. Shu}, {\em {The discontinuous {G}alerkin
  method with Lax--Wendroff type time discretizations}}, Comput. Methods Appl.
  Mech. Engrg., 194 (2005), pp.~4528--4543.

\bibitem{rannabauer2018ader}
{\sc L.~Rannabauer, M.~Dumbser, and M.~Bader}, {\em {A}{D}{E}{R}-{D}{G} with
  a-posteriori finite-volume limiting to simulate tsunamis in a parallel
  adaptive mesh refinement framework}, Comput. \& Fluids, 173 (2018),
  pp.~299--306.

\bibitem{reed1973triangular}
{\sc W.~H. Reed and T.~Hill}, {\em Triangular mesh methods for the neutron
  transport equation}, tech. report, Los Alamos Scientific Lab., N. Mex.(USA),
  1973.

\bibitem{shi2018local}
{\sc C.~Shi and C.-W. Shu}, {\em {On local conservation of numerical methods
  for conservation laws}}, Comput. \& Fluids, 169 (2018), pp.~3--9.

\bibitem{SHU198932}
{\sc C.-W. Shu and S.~Osher}, {\em Efficient implementation of essentially
  non-oscillatory shock-capturing schemes, {II}}, J. Comput. Phys., 83 (1989),
  pp.~32--78.

\bibitem{sun2017stability}
{\sc Z.~Sun and C.-W. Shu}, {\em {Stability analysis and error estimates of
  Lax--Wendroff discontinuous {G}alerkin methods for linear conservation
  laws}}, ESAIM Math. Model. Numer. Anal., 51 (2017), pp.~1063--1087.

\bibitem{titarev2002ader}
{\sc V.~A. Titarev and E.~F. Toro}, {\em {A}{D}{E}{R}: Arbitrary high order
  {G}odunov approach}, J. Sci. Comput., 17 (2002), pp.~609--618.

\bibitem{toro2001towards}
{\sc E.~F. Toro, R.~Millington, and L.~Nejad}, {\em Towards very high order
  {G}odunov schemes}, in {G}odunov Methods: Theory and Applications, Springer,
  2001, pp.~907--940.

\bibitem{van2007discontinuous}
{\sc B.~van Leer, M.~Lo, and M.~van Raalte}, {\em A discontinuous {G}alerkin
  method for diffusion based on recovery}, in 18th AIAA Computational Fluid
  Dynamics Conference, 2007, p.~4083.

\bibitem{van2005discontinuous}
{\sc B.~Van~Leer and S.~Nomura}, {\em Discontinuous {G}alerkin for diffusion},
  in 17th AIAA Computational Fluid Dynamics Conference, 2005, p.~5108.

\bibitem{woodward1984numerical}
{\sc P.~Woodward and P.~Colella}, {\em {The numerical simulation of
  two-dimensional fluid flow with strong shocks}}, J. Comput. Phys., 54 (1984),
  pp.~115--173.

\bibitem{yan2002local}
{\sc J.~Yan and C.-W. Shu}, {\em A local discontinuous {G}alerkin method for
  {K}d{V} type equations}, SIAM J. Numer. Anal., 40 (2002), pp.~769--791.

\bibitem{zhang2011third}
{\sc Q.~Zhang}, {\em {Third order explicit {R}unge-{K}utta discontinuous
  {G}alerkin method for linear conservation law with inflow boundary
  condition}}, J. Sci. Comput., 46 (2011), pp.~294--313.

\end{thebibliography}
\bibliographystyle{siamplain}

\end{document}